\documentclass[10pt]{article}
\usepackage{color}
\usepackage{amsmath}
\usepackage{amssymb}
\usepackage{mathrsfs} 
\usepackage[utf8]{inputenc}
\usepackage[english]{babel}
\usepackage{amsthm}
\usepackage{genyoungtabtikz}
\usepackage{tikz}
\usepackage[normalem]{ulem}

\usepackage[left=3cm,top=2cm,right=3cm,bottom=2cm]{geometry}
\usepackage{graphicx}

\newtheorem{theorem}{Theorem}
\newtheorem{proposition}[theorem]{Proposition}
\newtheorem{corollary}[theorem]{Corollary}
\newtheorem{lemma}[theorem]{Lemma}

\newtheorem{example}[theorem]{Example}
\newtheorem{definition}[theorem]{Definition}

\newcommand{\JM}{ \mathcal L }

\newcommand{\once}{11}

\newcommand{\diez}{10}
\newcommand{\diezR}{\color{red}10}
\newcommand{\doce}{12}

\newcommand{\doceR}{\color{red} 12}
\newcommand{\trece}{13}
\newcommand{\quince}{15}
\newcommand{\catorce}{ 14}
\newcommand{\cuarentacuatro}{ 44}
\newcommand{\veintidos}{22}

\newcommand{\treintatresR}{\color{red}33}

\newcommand{\II}{I_e}
\newcommand{\shape}{\textsf{shape}}
\newcommand{\MP}{{\rm Par }_{l,n}}
\newcommand{\F}{ { \mathbb F}}
\newcommand{\N}{ { \mathbb N}}
\newcommand{\No}{ { \mathbb N}_0}

\newcommand{\K}{\mathcal{K}}

\newcommand{\OO}{\mathcal{O}}

\newcommand{\A}{\mathcal{A}}
\newcommand{\Z}{\mathbb{Z}}

\newcommand{\BB}{\mathcal{B}^{\F}_{l,n}(\kappa)}
\newcommand{\B}{\mathcal{B}_n}

\newcommand{\BBB}{\mathcal{B}_{n+1}}
\newcommand{\Basis}{\mathcal{C}_n}

\newcommand{\res}{ \textrm{res} }
\newcommand{\q}{\hat{q}}
\newcommand{\m}{\mathfrak{m}}

\newcommand{\Par}{{\rm Par}_n}

\newcommand{\spa}{{\rm span}}

\newcommand{\s}{\mathfrak{s}}

\newcommand{\V}{\mathfrak{v}}
\newcommand{\T}{  \mathfrak{t}}
\newcommand{\g}{  \mathfrak{g}}

\newcommand\bi{\boldsymbol{i}}
\newcommand\bn{\boldsymbol{n}}
\newcommand\bj{\boldsymbol{j}}

\newcommand{\OnePar}{{ \rm Par}^1_{l,n}}
\newcommand{\MC}{{ {\rm Comp}}_{l,n}}
\newcommand{\MCm}{{ {\rm Comp}}_{l,m}}

\newcommand{\Si}{\mathfrak{S}}
\newcommand{\std}{{\rm Std}}
\newcommand{\nstd}{{\rm NStd}}
\newcommand{\tab}{{\rm Tab}}
\newcommand{\Snake}{{\rm Snake}}

\newcommand{\HH}{ \mathcal{H}_n}
\newcommand{\HHO}{ \mathcal{H}^{\OO}_n}
\newcommand{\HHK}{ \mathcal{H}^{\K}_n}
\newcommand{\HHtwo}{ \mathcal{H}_2}
\newcommand{\HHKtwo}{ \mathcal{H}^{\K}_2}
\newcommand{\HHOtwo}{ \mathcal{H}^{\OO}_2}

\newcommand{\R}{ \mathcal{R}_n}

\newcommand\blambda{{\boldsymbol\lambda}}
\newcommand\btau{{\boldsymbol\tau}}
\newcommand\bnu{{\boldsymbol\nu}}

\newcommand\bmu{{\boldsymbol\mu}}

\newcommand{\bT}{\pmb{\mathfrak{t}}}
\newcommand{\Bg}{\pmb{\mathfrak{g}}}

\newcommand{\bTI}{ \bT^{-1} (\bT_{\theta}^{\blambda}(1))}
\newcommand{\bTII}{ \bT^{-1} (\bT_{\theta}^{\blambda}(2))}
\newcommand{\bTj}{ \bT^{-1} (\bT_{\theta}^{\blambda}(j))}
\newcommand{\bTn}{ \bT^{-1} (\bT_{\theta}^{\blambda}(n))}

\newcommand{\Bs}{\pmb{\mathfrak{s}}}

\newcommand{\Bu}{\pmb{\mathfrak{u}}}
\newcommand{\Bv}{\pmb{\mathfrak{v}}}

\begin{document}
  \Yvcentermath1
\title{Graded cellular basis and Jucys-Murphy elements for generalized blob algebras}
\author{\sc  Diego Lobos Maturana{\thanks{Supported in part by CONICYT-PCHA/Doctorado Nacional/2016-21160722}} \, and steen ryom-hansen\thanks{Supported in part by FONDECYT grant 1171379  } }

\maketitle

\pagenumbering{roman}

\begin{abstract}
  We give a concrete construction of a graded cellular basis for the generalized blob algebra $ \B $
  introduced by Martin and Woodcock. The construction uses the isomorphism between KLR-algebras
  and cyclotomic Hecke algebras, proved by Brundan-Kleshchev and Rouquier. It 
  gives rise to a family of Jucys-Murphy elements for $ \B$.
\end{abstract}

\pagenumbering{arabic}

\section{Introduction}
This paper is concerned with the generalized blob algebra $ \B$ introduced by Martin and Woodcock.
Its representation theory has received a considerable amount of interest in recent years, see for example \cite{bcs}, \cite{bowman}, \cite{LiPl}.

\medskip
The original blob algebra $b_n =  b_n(q,m) $, also known as the 
Temperley-Lieb algebra of type $ B $, 
was introduced by Martin and Saleur via 
considerations in statistical mechanics. 
The usual Temperley-Lieb algebra $ TL_n = TL_n(q) $ can be realized as
a quotient of the Hecke algebra $ {\mathcal H}_n(q) $ of finite type $ A$
and similarly it has also been known for some time that $ b_n $ is a quotient of the two-parameter Hecke algebra $ {\mathcal H}_n(Q, q) $ of type
$ B$. 
Since $ {\mathcal H}_n(Q, q) $ 
is the special case $ l=2 $ of a cyclotomic Hecke algebra $ {\mathcal H}_n(q_1, \ldots, q_l ) $  
one could now hope that this construction make{\color{black}{s}} sense for any cyclotomic Hecke algebra. 
Martin and Woodcock showed in \cite{MW} that this indeed is the case.
They obtain $ b_n $ as the quotient of $ {\mathcal H}_n(Q, q) $ by the ideal generated by the idempotents for the irreducible $ {\mathcal H}_2(Q, q) $-modules
associated with the bipartitions $ ( (2), \emptyset) $ and $ (  \emptyset, (2)) $
and showed that this idea generalizes to every $ {\mathcal H}_n(q_1, \ldots, q_l ) $. 
The quotient algebras of $ {\mathcal H}_n(q_1, \ldots, q_l ) $ that arise this way are the generalized 
blob algebras $ \B = \B(q_1, \ldots, q_l )$ of the title. The parameter $ l $ is known as the level parameter and
the generalized blob algebras
can therefore be considered as the 
{\it Temperley-Lieb algebras at level $ l$}.

\medskip
We are interested in the modular, that is non-semisimple, representation theory of $ \B$. This is the case where 
the ground field $ \F $ is of positive characteristic or where 
the parameters $ q_i $ are
roots of unity. The modular representation theories of 
$ TL_n  $ and $ b_n$ are well understood and may be considered 
as approximations of the modular representation theory of $ \B$.
The modular representation theory of $ \B $ is more complicated. In characteristic $ 0$
it involves inverse Kazhdan-Lusztig polynomials of type $ \tilde{A}_{n-1}$, see \cite{bcs} and \cite{MW}, 
and in characteristic $ p$ it involves the $p$-canonical basis, at least conjecturally, see \cite{LiPl}.

\medskip
Cellular algebras were introduced 
by Graham-Lehrer as a general framework for studying modular representation theory. 
They are finite dimensional algebras endowed with a 
basis such that the structure constants with respect to the basis satisfy certain natural conditions.
A cellular algebra $ \cal A $ is always equipped with a family $ \left\{ \Delta(\lambda)  \right\} $ of 'cell modules' for 
$ \lambda $ running over a poset 
$ \Lambda $ 
which is part of the cellular basis data. Each cell module $ \Delta(\lambda)$ is endowed with a
billinear form $ \langle \cdot, \cdot \rangle $ 
and the irreducible modules $ \{ L(\lambda) \} $ all arise as quotients by the
radical of the form $ L(\lambda) = \Delta(\lambda) / {\rm rad} \langle \cdot, \cdot \rangle $.
Using this, there is for a cellular algebra $ \cal A $ a concrete way of obtaining the irreducible $  \cal A $-modules,
at least in principle.

\medskip
Two of the motivating examples for cellular algebra were the Temperley-Lieb algebra $ TL_n $ with its
diagram basis and the Hecke algebra $ {\mathcal H}_n(q)  $ with its cell basis derived from
the Kazhdan-Lusztig basis.  
In fact, one parameter Hecke algebras of finite type are always cellular, as was shown by Geck, \cite{Geck}.
For Hecke algebras $ {\mathcal H }(W,S) $ with unequal parameters associated with a finite Coxeter system, 
Lusztig's cell theory
depends on the choice of a \emph{weight function} on $ W$, and
conjecturally it leads to a cellular basis as well, see \cite{BGIL}.
For the cyclotomic Hecke algebra $ {\mathcal H}_n(q_1, \ldots, q_l ) $
there is also a concept of a \emph{weighting function} $ \theta $, 
which plays a key role for the Fock space approach to the representation theory
of $ {\mathcal H}_n(q_1, \ldots, q_l ) $, see \cite{A}, \cite{FLOTW}, \cite{JMMO}, \cite{Uglov}.  
For $ {\cal H }_n(Q,q) $ and for the zero weighting $ \theta_0 $, Lusztig's approach 
\emph{does} induce a cellular algebra structure on $ {\cal H }_n(Q,q) $
and this was shown in \cite{Steen1} to be compatible with the diagram basis on $ b_n $.

\medskip
In this work we show that $ \B $ is a cellular algebra with respect to the zero weighting. 
There is however neither a natural Temperley-Lieb like diagram basis nor a Lusztig cell theory available for $ \B $
and in fact our methods for showing cellularity of $ \B $ are completely new. They
are based on the seminal work by Brundan-Kleshchev and Rouquier that establishes an isomorphism
between the KLR-algebra $ \R $ and the cyclotomic Hecke algebra $ {\mathcal H}_n(q_1, \ldots, q_l ) $.
The KLR-algebra $ \R $ is a $ \mathbb Z $-graded algebra and our graded cellular basis on $ \B $ inherits
this $ \mathbb Z $-grading, making it a graded cellular basis. 

\medskip
The KLR-algebra has already been used by Hu-Mathas, \cite{hu-mathas},  
and by Plaza and Ryom-Hansen, \cite{PlazaRyom}, to construct $ \mathbb Z$-graded
cellular bases for $ {\mathcal H}_n(q_1, \ldots, q_l ) $ and for $ b_n(q) $,
but contrary to the present work those papers rely in a decisive way on already existing non-graded
cellular bases on the algebras in question.
Indeed Hu-Mathas rely in \cite{hu-mathas} on Murphy's standard basis for $ {\mathcal H}_n(q_1, \ldots, q_l ) $,
and in \cite{PlazaRyom} the diagram basis for $ b_n $ is needed in order to derive the
graded cellular bases.
Note that
Murphy's standard basis only exists for the classical dominance
order on $ \MP$, which is unrelated to the zero weighting.

\medskip
The representation theory of $ {\mathcal H}_n(q_1, \ldots, q_l ) $ is
parametrized by $l$-multipartitions $ \MP$ of $ n$ whereas the 
representation theory of $ \B $ is parametrized by one-column $l$-multipartitions $\OnePar$ of $ n $.
Our $ \mathbb Z $-graded cellular basis 
\begin{equation} \Basis = \{ m_{\Bs \bT} \mid \blambda \in \OnePar, \, \Bs, \bT \in \std(\blambda) \} \end{equation}
shares notationally several features of Murphy's standard basis and just like that basis it depends
on the existence of a unique maximal $\blambda$-tableau $ \bT^{\blambda}$ for each 
$ \blambda \in \OnePar$, with respect to $ \theta_0$. As we point out in section 2 of our paper, for $ \blambda \not\in \OnePar$ 
there are in general many maximal $\blambda$-tableaux and so our methods do not
generalize to give a cellular basis for $ {\mathcal H}_n(q_1, \ldots, q_l ) $, with respect to $ \theta_0$ 
and in particular we do not recover the more general results of Bowman and Stroppel-Webster
on Schur algebras for $ {\mathcal H}_n(q_1, \ldots, q_l ) $, 
see \cite{bowman}, \cite{StroppelWebster} and \cite{W}.

\medskip
Let us explain in more detail the contents of the paper.
In the following section 2 we set up the combinatorial concepts and notations that are needed for our work, including 
multipartitions, tableaux, and so on.
We also present the various order relations on multipartitions and tableaux that play a role throughout the paper. They all depend on the choice of a 
weighting $ \theta \in {\mathbb Z}^l $ 
and so the material of this section 2 is not completely standard. 
For $ \blambda \in \OnePar$ we prove a version of
Ehresmann's Theorem relating 
the order relation $ \unlhd_{\theta} $ on $ \tab(\blambda) $ with the Bruhat order on the symmetric group
$ \Si_n$.
Although this and a few other of our results are valid for general $ \theta $ we soon concentrate on the zero 
weighting $ \theta_0 $.

\medskip
Assume that $ q \in \F $ be a primitive $ e$'th root of unity and define $ \II := \Z/e \Z$. 
In section 3 we first introduce the concept of a strongly adjacency-free multicharge $ \hat{\kappa} $ and 
next give the formal definitions of the generalized blob algebra $\B $ and of the KLR-algebra $ \R$, in terms of generators
$\{\psi_1,\dots,\psi_{n-1}\}\cup{\{y_1,\dots,y_n\}}\cup{\{e(\bi) \mid \bi\in{\II^n}\}}$
and a long list of relations between them.
Both algebras depend on $ \hat{\kappa} $. 
We give diagrammatical as well as algebraic definitions of the algebras. At the end of the section we prove a series of simple Lemmas. All 
of our inductive arguments in the following sections are based on these Lemmas, or on the defining relations for $ \B$.

\medskip
In section 4 we obtain our first important results. 
The language used to present them is reminiscent of Murphy's theory, although the underlying combinatorics and hence the proofs are completely different.
For each $ \blambda \in \OnePar $ there is an associated $ \bi^{\blambda} \in  \II^n$ and 
we first show that only the idempotents $ e (\bi^{\blambda}) $ are needed for generating $ \B$.
The proof for this is a subtle induction argument. We introduce a symbolic notation for $ e (\bi) $ and $ y_i e (\bi)  $ which helps us formulate
the induction process and give several examples for the use of this notation. We consider the idempotents $ e (\bi^{\blambda}) $'s as analogs of
the initial elements $ m_{\blambda} $ of Murphy's standard basis and accordingly choose 
elements $ m_{\Bs \bT} = \psi_{d(\Bs)}^{\ast} e(\bi^{\blambda}) \psi_{d(\bT)} $, 
where $ \Bs$ and $ \bT $ run over $ \blambda $-multitableaux. 
We then show that these elements span $ \B$, once again arguing by induction over $ \blambda$. 
The final result of this section is the proof that only standard tableaux $ \Bs$ and $ \bT $ are needed for generating $ \B$.
For this result we develop a theory of Garnir tableaux which turns out to be quite different from the classical Garnir theory. 
In particular, our Garnir tableaux are not uniquely characterized by their point of non-standardness but 
still we are able to classify the Garnir tableaux of shape $ \blambda$ and this way prove our results.

\medskip
In section 5 we prove the linear independence of $ \{ m_{\Bs \bT} \}$ for 
$ \Bs$ and $ \bT $ running over standard $ \blambda $-multitableaux. 
For this we rely on the Brundan-Kleshchev and Rouquier isomorphism between 
$ \R $ and $ {\mathcal H}_n(q_1, \ldots, q_l ) $ 
together with Hu-Mathas's key insight, stating that the images of the idempotents $ e(\bi)  $
can be described in terms of the idempotents associated with 
Young's seminormal form.

\medskip
In section 6 we obtain our main result showing that the set $ \{ m_{\Bs \bT} \}$
is a graded cellular basis for $ \B $ with respect to the dominance order associated with
$ \theta_0 $ and that Jucys-Murphy elements $ \{ L_i \} $ for $ {\mathcal H}_n(q_1, \ldots, q_l ) $
are 
JM-elements with respect to this basis, in the sense of Mathas. Given that the set 
$ \{ m_{\Bs \bT} \}$ has been shown to be a basis
in the previous sections, the multiplicative conditions for being a cellular basis are reduced to
two combinatorial Lemmas, that we prove. This gives at the same time the JM-property.

\medskip
Finally, in the last section 7 we take the opportunity to show that our definition of $ \B $ is equivalent to the
original definition given by Martin-Woodcock in \cite{MW}. As indicated above,
in the original definition we have 
$ \B =  {\mathcal H}_n(q_1, \ldots, q_l )/J $ where $ J $ is the ideal generated by idempotents associated
with certain irreducible $ {\mathcal H}_2(q_1, \ldots, q_l )$-modules. These idempotents turn out to be
instances of the idempotents associated with Young's seminormal form, that already appeared in section 5. The
equivalence of the two definitions follows from this observation.

\medskip
It is a pleasure to thank Anton Cox for stimulating conversations during the initial phase of this project.
{\color{black}{It is also a great pleasure to thank the anonymous referee for many useful comments that helped us
    improve this article.}}

\section{Combinatorics and tableaux}
Let us recall the basic combinatorial concepts and notations associated with the representation
theory of the symmetric group $\mathfrak{S}_n$ and the wreath product $ C_l \wr \Si_n $.

We denote by $ \N $ the positive integers and by $ \No $ the non-negative integers.
For $n \in \No $, a \emph{composition} $\lambda$ of $n$
is a sequence $ \lambda = (\lambda_1,\lambda_2,\dots)$ of elements of $ \No $ 
such that $|\lambda|:=\sum_{k}\lambda_k=n$. If $ k $ is minimal such that
$\lambda_i=0  \, \mbox{ for all} \,  \, i > k$ we also
write $\lambda =(\lambda_1,\dots,\lambda_k)$ for  $\lambda$.
We say that a composition $\lambda= (\lambda_1, \lambda_2, \ldots) $ of $ n$ is a \emph{partition} of $n$ if it satisfies that $\lambda_k\geq \lambda_{k+1}$ for all $k\geq1$.

For integers $l  > 0$ and $ n \ge 0 $, an $l$-\emph{multicomposition} of $n$ is an $l$-tuple of
compositions $\blambda  =(\lambda^{(1)},\dots,\lambda^{(l)})$ such that $\sum_{m=1}^l{|\lambda^{(m)}|}=n$.
An $l$-\emph{multicomposition} $\blambda=(\lambda^{(1)},\dots,\lambda^{(l)})$ of $ n $ is called
an $l$-multipartition of $ n$  if all its components $ \lambda^{(i)} $ are partitions.
The set of all $l$-multicompositions of $n$ is denoted by $\MC$ and the
set of all $l$-multipartitions of $n$ is denoted by $\MP$.

Let
$\blambda=(\lambda^{(1)},\dots,\lambda^{(l)})$ be an $l$-multicomposition.
Then $\blambda $ is called a \emph{one-column} $l$-multicomposition
if all of its components $ \lambda^{(i)}$ are one-column compositions, that
is each $ \lambda^{(i)} $ is of the form
$ \lambda^{(i)} = ( \lambda_1^{(i)},  \lambda_2^{(i)}, \ldots,  \lambda_r^{(i)})$ where $  \lambda_j^{(i)} $ is either $ 0 $ or $1$ for all $ j$.

A \emph{one-column} $l$-\emph{multipartition} is a one-column $l$-multicomposition which is also an $l$-multipartition.
For $\blambda$ a one-column $l$-multipartition each of its components $\lambda^{(m)}$ is a partition of the form
$\lambda^{(m)} = (1,1,\ldots, 1) $ that is $\lambda^{(m)} = (1^{a_m}) $ where $ a_m =  | \lambda^{(m)}|$.
In other words, a one-column $l$-multipartition is of the form $\blambda=((1^{a_1}),\dots,(1^{a_l})) $ for
certain non-negative integers $ a_i$.
The set of all one-column $l$-multipartitions of $n$ is denoted by $\OnePar$.

We shall hold $ l $ fixed throughout the article, and
shall therefore frequently refer to $ l$-multicompositions (resp. $ l$-multipartitions, etc) simply as  multicompositions (resp. multipartitions, etc).

Let $ \lambda = (\lambda_1, \lambda_2, \ldots, \lambda_k)  $ be a composition of $ n$.
Then we represent $ \lambda $ graphically via its \emph{Young diagram} $ [\lambda]$.
We use English notation so it consists of an array of
$ k$ left adjusted lines of
boxes denoted the \emph{nodes} of the diagram, the first line containing $ \lambda_1 $ nodes, the second line $ \lambda_2 $ nodes, and so on.
The nodes are labelled using matrix convention, that is the $j$'th node of the $ i$'th line
of $ [\lambda]$
is labelled $(i,j)$ and in this case we write $ (i,j ) \in [\lambda]$. For example, if $ \lambda = (4,2,6,1) $ then the Young diagram $ [\lambda]$ is
$$ [\lambda]=  \yng(4,2,6,1) \, \, \, . $$

For an $l$-multicomposition
$ \blambda=(\lambda^{(1)},\dots,\lambda^{(l)})$ we define
its Young diagram $[ \blambda] $ to be the $l$-tuple of Young diagrams
$ ( [\lambda^{(1)}],\dots,[\lambda^{(l)}]) $. The nodes of $ \blambda $ are labelled by
the triples $ (i,j,k) $ where $ (i,j) $ is a node of $ [\lambda^{(k)}] $.
For example, if  $\blambda =((1,1,1,1),(1),(1,0,1))$ we have
  that    
  
 \begin{equation}\label{firstex}
   [\blambda]=\left(\,
   \begin{tikzpicture}[scale=1, baseline={([yshift=-\the\dimexpr\fontdimen22\textfont2\relax]
                    current bounding box.center)},]
     \tgyoung(0cm,0cm,;,;,;,;);
     \tgyoung(0.7cm,0cm,;,:,:,:); 
     \tgyoung(1.4cm,0cm,;,:,;,:); 
     \draw[black](1.4,0.45)--(1.4,-0.92);
     \node at(0.6,-0.5) {,};
     \node at(1.2,-0.5) {,};
   \end{tikzpicture}
       \, \right)
 \end{equation}
or if $\bmu=((1^4),(1^0),(1^3))$ we have that
\begin{equation}{\label{secondex}}[\bmu]=\left(\,    \gyoung(;,;,;,;) \, ,\emptyset, \, \gyoung(;,;,;,:)
   \, \right).
\end{equation}

 \medskip
For a multipartition $ \blambda$ we define the \emph{i'th row} of $ \blambda$ as the set of nodes of the form
$ (i,j,k) $.

\medskip

 There is a well known way to make $ \MC $ into a poset, the associated order relation being the
 dominance order on $ \MC $ studied for example in \cite{DJM}. However, this is not the only interesting
 order relation on $ \MC $.

Let us fix a tuple $ \theta = (\theta_1, \ldots, \theta_l ) \in  {\mathbb Z}^l $, called
a \emph{weighting}. Let
$ \gamma = (i,j,b) $ and $ \gamma^{\prime} = (i^{\prime},j^{\prime},b^{\prime}) $ be nodes
 of multipartitions $ \blambda$ and $ \bmu$, or more generally elements of
 $  \N \times \N \times \{1,\dots,l\}$.
Then we write $ \gamma \lhd_{\theta} \gamma^{\prime} $ if either $ (\theta_b + j - i) < (\theta_{b^{\prime}} + j^{\prime} - i^{\prime} )$
or if $ (\theta_b + j - i) = (\theta_{b^{\prime}} + j^{\prime} - i^{\prime} )$
and $ b > b^{\prime}$. (The last inequality is not an error). We write $  \gamma \unlhd_{\theta} \gamma^{\prime} $ if
$  \gamma \lhd_{\theta} \gamma^{\prime} $ or if $  \gamma = \gamma^{\prime} $.

This defines an order on $ \N \times \N \times \{1,\dots,l\}$ that we extend
to multipartitions as follows. 
Suppose that
$\blambda \in \MC$ and $ \bmu \in \MCm$. Then we write $\blambda \unlhd_{\theta} {\bmu}$ if for each $ \gamma_0 \in
  \N\times \N \times \{1,\dots,l\}$ we have that
  \begin{equation}\label{defdominance}
    | \{ \gamma \in{[\blambda]} \colon \gamma \rhd_{\theta}  \gamma_0  \}|\leq|\{ \gamma \in{[\bmu]}
     \colon \gamma \rhd_{\theta} \gamma_0\}|.
     \end{equation}
   This order relation $ \lhd_{\theta} $ depends highly on the initial choice of weighting $ \theta$.
   When restricted to $ \MP$ and choosing $ \theta $ such that
   $ \theta_i > \theta_{i+1} +n $ for all $ i $ we recover the dominance order used in [DJM] which we refer to as
   $ \unlhd_{\infty}$.
   This is the
   \emph{separated} case, but in this article we shall be mostly interested in another limit case, namely
   the one given by the zero weighting $ \theta = (0,0,\ldots, 0)$. We refer to the corresponding order as $ \unlhd_{0}$.

   \medskip
   Note that for $ l=1 $, we have that $ \unlhd_{\theta} $ is just the usual dominance order, for any $ \theta$.

\medskip
   
In general, the order $ \unlhd_{\theta} $ is only a partial order on the nodes of $\MP$ or 
$ \N \times \N \times \{1,\ldots, l\}$, but it becomes a total order upon restriction 
to the nodes of $ \OnePar$ or $ \N \times \{1\} \times \{1,\ldots, l\}$.
Using this we can prove the following useful Lemma that we shall use implicitly throughout the paper. It
says that $ {\blambda}\unlhd_{\theta}{\bmu}$ if and only if $ \bmu $ can be obtained from $ \blambda $ by
moving nodes of $ \blambda$ upwards.

\begin{lemma}\label{order-bijection}
Suppose that ${\blambda}, {\bmu} \in \OnePar$. Then ${\blambda}\unlhd_{\theta}{\bmu}$
  if and only if there is a bijection $\Theta:[{\blambda}]\rightarrow[{\bmu}]$ such that
  $\Theta(\gamma)\unrhd_{\theta}\gamma$ for all $\gamma\in[{\blambda}].$
\end{lemma}
\begin{proof}
  As mentioned  $ \unlhd_{\theta} $ is a total order on the nodes of $ \N \times \{1\} \times \{1,\ldots, l\}$
  and so there is an order preserving
  bijection  
  from these 
  nodes to $ \N$, where $ \N $ is endowed with the opposite of the natural order, that is '1' is the maximal element.
  Using this, we may view $ \blambda $ and $ \bmu$ as ordered subsets of $ \N$. But in this situation
  one easily checks the equivalence of (\ref{defdominance}) with the existence of $ \Theta$.
  \end{proof}

To illustrate the difference between $ \unlhd_{\infty} $ and $ \unlhd_{0} $
we consider their restriction 
to $ \OnePar$. In each case there is a unique maximal element but the two maximal elements are different.
The unique maximal elements with respect to $ \unlhd_{\infty} $ is 
\begin{equation}
  \bmu_n^{max, \infty}:=
  ( (1^{n}), \emptyset,  \emptyset, \ldots, \emptyset) {\color{black}{)}}
 \end{equation} 
To describe $\bmu_n^{max, 0} $, the unique maximal element with respect to 
$ \unlhd_{0} $, we use integer division to write $ n = q l + r $ where $ q,l  \in \mathbb Z $ such that
$ 0 \leq r < l $. Then we have that 
$\bmu_n^{max, 0}$ is given by 
\begin{equation}{\label{lambda-max}}
  \bmu_n^{max, 0} = ( \overbrace{(1^{q+1}), \ldots, (1^{q+1})}^{r {\, terms}},
  \overbrace{(1^{q}), \ldots, (1^{q})}^{ l-r \, terms} ).
\end{equation}
For example, for $ n = 7 $ and $ l = 3 $ we have that
\begin{equation}{\label{maxpartitions}}
\bmu_n^{max, \infty}= \left(\,\gyoung(;,;,;,;,;,;,;), \, \emptyset,  \, \emptyset \right), \, \, \, \, \, \, \, \, \, 
\bmu_n^{max, 0}= \left(\,\gyoung(;,;,;), \, \gyoung(;,;,:),  \, \gyoung(;,;,:) \right).
\end{equation}

In general, with respect to $ \unlhd_{\infty} $ the big multipartitions tend to have their center of mass
to the left of the diagram, whereas with respect to $ \unlhd_{0} $ the big multipartitions tend to have their center
of mass in the middle of the diagram.

\medskip
For $ l = 2 $, the restriction of $ \unlhd_{0} $ to $ \OnePar $ is the total order
used for example in \cite{PlazaRyom} and \cite{Steen1}. 
Here is the $ n =3 $ case:
\begin{equation}
\left(  \emptyset, (1^3)  \right)  \unlhd_0    
\left((1^3),  \emptyset   \right)  \unlhd_0  
\left((1), \, (1^2)    \right)  \unlhd_0
  \left((1^2),  (1)  \right).
\end{equation}  
For $ l \ge 3$, the restriction of $ \unlhd_0 $ to $ \OnePar $ is only a partial order.
Here we illustrate the $ n =l=3 $ case:
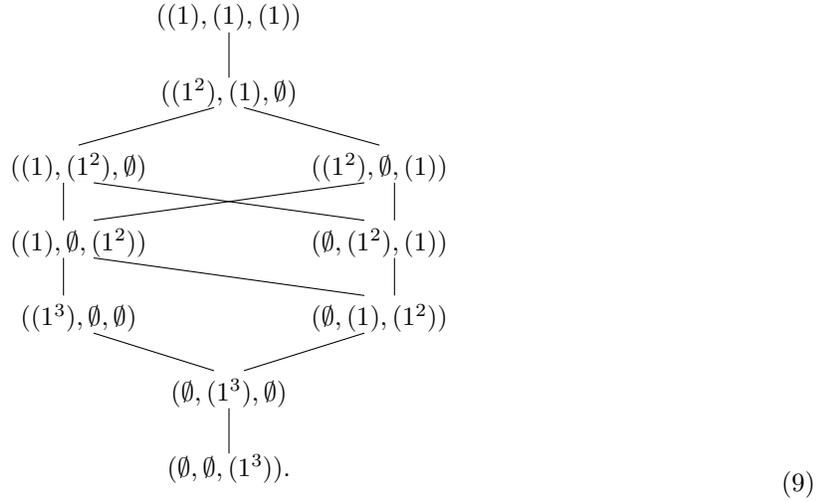
\begin{equation}
\begin{tikzpicture}
\node at(0,0) {$ ( (1), (1),(1))$ };
\draw[](0,-0.2)--(0,-0.8);
\node at(0,-1) {$ ( (1^2), (1), \emptyset)$ };
\draw[](-0.2,-1.2)--(-2,-1.7);
\draw[](0.2,-1.2)--(2,-1.7);
\node at(-2,-2) {$ ( (1), (1^2), \emptyset)$ };
\draw[](2-0.2,-2.2)--(-2+0.2,-2.7);
\draw[](-2+0.2,-2.2)--(2-0.2,-2.7);
\draw[](2+0.2,-2.2)--(2+0.2,-2.7);
\draw[](-2-0.2,-2.2)--(-2-0.2,-2.7);
\node at(2,-2) {$ ( (1^2),  \emptyset, (1))$ };
\draw[](2+0.2,-3.2)--(2+0.2,-3.7);
\draw[](-2-0.2,-3.2)--(-2-0.2,-3.7);
\draw[](-2+0.2,-3.2)--(2-0.2,-3.7);
\node at(-2,-3) {$ ( (1), \emptyset,  (1^2) )$ };
\node at(2,-3) {$ (  \emptyset, (1^2),  (1))$ };
\draw[](2-0.2,-4.2)--(+0.2,-4.7);
\draw[](-2+0.2,-4.2)--(-0.2,-4.7);
\node at(-2,-4) {$ ( (1^3), \emptyset, \emptyset  )$ };
\node at(2,-4) {$ (  \emptyset, (1),  (1^2))$ };
\node at(0,-5) {$ (  \emptyset, (1^3),  \emptyset)$ };
\draw[](0,-5.2)--(0,-5.8);
\node at(0,-6) {$ (  \emptyset,   \emptyset, (1^3)) {\color{black}.}$};
\end{tikzpicture}  
\end{equation}

\medskip
 Let ${\lambda}$ be a composition of $n$. A \emph{tableau} of shape $\lambda$ or simply
 a $\lambda$-\emph{tableau} is a bijection $\T:\{1,\dots,n\} \rightarrow[\lambda]$.
 In this case we write $\shape(\T)=\lambda$.
 A $\lambda$-tableau $\mathfrak{t}$ is represented graphically via a labelling of the nodes
 of $ [\lambda]$ using the numbers $ \{ 1, 2, \ldots, n\} $ where the labelling
 of the node $ (i,j) $ is given by $ \T^{-1}(i,j) $. 
 In this case we say that the $(i,j)$'th node of $ \T $ is \emph{filled in with} $ \T^{-1}(i,j) $ via $\T$.
Let $\blambda$ be an $l$-multicomposition.
The concept of $\blambda$-\emph{tableaux} is defined the same way as for ordinary $ \lambda$-tableaux,
 that is a $\blambda$-{tableau} is a bijection $ \bT:\{1,\dots,n\} \rightarrow  [\blambda].$

 A $ \lambda$-tableau $\mathfrak{t}$ is called \emph{standard} if the corresponding labelling
 of $ [\lambda] $ has increasing numbers from left to right along rows and from top to bottom along columns.
 Similarly, for a tableau ${\bT}$ of a multicomposition $ \blambda$ we say that it is standard if all its components
 are standard.
 For a composition $\lambda$, we denote by $ \tab(\lambda) $ and ${\std}({\lambda})$
 the set of all $ \lambda$-tableaux and the set of all standard $ \lambda$-tableaux and we use a similar
 similar notation for $\blambda$-tableaux of a multicomposition $\blambda$.

 For a composition ${\lambda}$
 and a ${\lambda}$-{tableau} ${\T}$ and $1\leq k\leq n$ we denote by ${\T} \! \mid_k$ the restriction of $\T$
 to the set $ \{1,2, \ldots, k \}$. A similar notation is used for tableaux for multipartitions.
 Let $\bmu$ be as in ({\ref{secondex}). Then the following are $\bmu$-tableaux
\begin{equation}{\label{tableauxex}}
\bT= \left(\,    \gyoung(;1,;4,;5,;7) \, ,\emptyset, \, \gyoung(;2,;3,;6,:)  \, \right), \, \,
\Bs =\left(\,    \gyoung(;1,;5,;4,;6) \, ,\emptyset, \, \gyoung(;3,;2,;7,:)  \, \right)
\end{equation}
but only the first is standard.
Note that for all $1\leq k \leq n$ we have that $\shape(\mathfrak{t} \! \mid_k)$ is a multipartition,
but in the case of $\Bs$ we have $$\shape(\Bs \! \mid_4)=((1,0,1), {\color{black}{ \emptyset}}, (1,1))$$ which is not a
multipartition, only a multicomposition.

We extend the order $ \unlhd_{\theta}$ to tableaux for multipartitions
$ n$, as follows.
Let $ \blambda$ and $ \bmu$ be multicompositions of $ m $ and $ n $ and
let $\Bs$ and $\bT$ be tableaux of shapes $ \blambda$ and $ \bmu$. Then we write ${\bT}\unlhd_{\theta} \Bs$
if for all $ 1 \le  k \le \min(m,n) $ we have that
$$\shape( \bT \!  \mid_k)\unlhd_{\theta} \shape(\Bs \! \mid_k).$$
For example, considering the tableaux $\Bs $ and $ \bT$ from ({\ref{tableauxex}})
we have that $\Bs \lhd_{0} \bT$.

\medskip
Let $ \blambda \in \MP$ 
be a multipartition and let $ \gamma \in \N \times \N \times \{1,\dots,l\} \setminus [ \blambda] $.
Then we say that $ \gamma $ is an \emph{addable node} for $ \blambda $ if $ [ \blambda] \cup \gamma $ is the Young diagram of 
a multipartition. Dually we say that $ \gamma \in  [ \blambda] $
is a \emph{removable node} for $ \blambda $ if $ [ \blambda] \setminus \gamma $ is the Young diagram of a multipartition.
The set of addable (removable) nodes for $ \blambda $ is totally ordered under $ \unlhd_{\theta}$.

\medskip
For $ \blambda \in \MP$ we now define multipartitions $ \blambda_{\theta,0}, \ldots, \blambda_{\theta, n}  \in \MP $
recursively via $ \blambda_{\theta,0} := ( \emptyset, \ldots, \emptyset) $ 
and for $ i > 0 $ via $ [\blambda_{\theta,i}] := [\blambda_{i-1}]  \cup \gamma_{\theta, i}$ where
$ \gamma_{\theta, i} \in [\blambda] $ satisfies the condition that it is the largest addable node for $ \blambda_{i-1} $, with respect to $ \unlhd_{\theta}$.
We denote by $ \bT_{\theta}^{\blambda} $ the $ \blambda$-tableau which is given by $ \bT_{\theta}^{\blambda}(i) =
\gamma_{\theta, i}$. If $ \theta = \theta_{\infty}$ we write $ \bT_{\infty}^{\blambda}$ for
$ \bT_{\theta}^{\blambda} $ and if
$ \theta = \theta_{0}$ we write $ \bT_{0}^{\blambda}$ for
$ \bT_{\theta}^{\blambda} $.

\medskip
Suppose that 
$ \blambda \in  \OnePar$.
Then $   \bT_{\infty}^{\blambda} $ is the unique maximal element 
in $ \tab(\blambda) $ and $ \std(\blambda) $ with respect to $ \unlhd_{\infty} $. It is the $ \blambda$-tableau obtained by filling in the nodes of $ [ \blambda]$
from left to right along the columns. For example, for $ \blambda = ((1^3), (1^3), (1^2)) $ it
is
\begin{equation} \bT_{\infty}^{\blambda}= \left(\gyoung(;1,;2,;3),\gyoung(;4,;5,;6),\gyoung(;7,;8,:)\right). \end{equation}

Let still $\blambda \in \OnePar $. Then $ \bT_{0}^{\blambda}$  is the unique maximal element
in $ \tab(\blambda) $ and $ \std(\blambda) $ with respect to $ \unlhd_{0} $. It is 
the $\blambda$-tableau $\bT^{\blambda}$ 
in which $1,\dots,n$ are filled in increasingly along the rows of $ \blambda$. For example, for $ \blambda = ((1^3), (1^3), (1^2)) $ it
is 
\begin{equation}
  \bT_0^{\blambda}= \left(\gyoung(;1,;4,;7),\gyoung(;2,;5,;8),\gyoung(;3,;6,:)\right).
\end{equation}  
The tableau $ \bT_{\theta}^{\blambda} $ plays an important role in our paper, especially for $ \theta = \theta_0$, so let us prove
formally the claim on maximality of $ \bT_{\theta}^{\blambda} $.

\medskip

Let first $\mathfrak{S}_n$ be the symmetric group on $ \bn := \{1, \ldots, n \} $, and let $S=\{s_1,\dots,s_{n-1}\}$ be its subset
of \emph{simple transpositions}, i.e. for each $k=1,\dots,n-1 $ we have that $ s_k=(k,k+1)$.
It is well known that $\mathfrak{S}_n$ is a Coxeter group on $ S$.
For any multicomposition $ \blambda$ of $n$ we have that $\mathfrak{S}_n$ acts on the right on
$\tab(\blambda)$ by permuting the entries inside a given tableaux. Thus, if $ w = s_{i_1} s_{i_2} \cdots  s_{i_N} $ where $ s_{i_j} \in S $ and
if $ \bT \in \tab(\blambda)$ we have that $ \bT w= ( \cdots ((\bT s_{i_1}) s_{i_2} \cdots )  s_{i_N}) $.

We next need to introduce yet another order on $ \tab(\blambda)$. 
Let $ \blambda $ be a multipartition and
let $\bT,\Bs$ be ${\blambda}$-tableaux.
For $s \in{S}$ we define $\bT \stackrel{s}{\rightarrow} \Bs$ if 
$\Bs=\bT s$ and $\Bs \rhd_{\theta} \bT$. We let $ \succ_{\theta} $ be the order on $ \tab(\blambda) $
induced by $\bT \stackrel{s}{\rightarrow} \Bs$ for all $ s \in S$, that is 
$\Bs \succ_{\theta} \bT$
if there is a finite sequence $$\bT_0 \stackrel{s_{i_1}}{\rightarrow} \bT_1
\stackrel{s_{i_2}}{\rightarrow} \dots \stackrel{s_{i_k}}{\rightarrow} \bT_k$$ with $\bT_0=\bT$
and $\bT_k=\Bs$. We call $ \succ_{\theta}$ the weak order on $ \tab(\blambda)$. 
It is clear that $\Bs\succ_{\theta}\bT\Rightarrow\Bs\rhd_{\theta}\bT,$ but the converse is false in general.
Consider for example ${\bmu}=((1^3),(1^3),(1^2))$
and the $\bmu$-tableaux
$$ \bT = \left(\, \gyoung(;1,;6,;5),\gyoung(;2,;4,;8),\gyoung(;3,;7,:) \, \right), \, \, \, \, 
 \Bs=  \left(\, \gyoung(;1,;6,;5),\gyoung(;2,;7,;8),\gyoung(;3,;4,:) \, \right). $$
Then with respect to $ \theta = (0,0,\ldots, 0) $ we have that $ \bT \rhd_{\theta} \Bs $ but  $ \bT \nsucc_{\theta} \Bs $.  

\medskip
We can now prove the promised claim for $ \bT_{\theta}^{\blambda}$.

\begin{lemma}\label{t-lamb-max}
  Suppose that ${\blambda} \in \OnePar$.
\begin{itemize}
\setlength\itemsep{-1.1em}
\item[a)]
 Let $ \bT \in \tab(\blambda) $ and set $ \Bs = \bT s_k $. Suppose that 
  $ \bT(k) \lhd_{\theta} \bT(k+1) $. 
Then we have that $ \bT \prec_{\theta} \Bs $.  \\
\item[b)] We have that $\bT_{\theta}^{\blambda}$ is the unique maximal element
  in $ \tab(\blambda) $ and $ \std(\blambda) $ with respect to $ \prec_{\theta} $ and $ \lhd_{\theta}$.
\end{itemize}
\end{lemma}
\begin{proof}
  The nodes of $ \blambda $ are totally ordered with respect to $ \lhd_{\theta}$, and
  we have $$ \bT^{\blambda}(i) \lhd_{\theta} \bT^{\blambda}(j) \mbox{   iff    } i > j.$$
  Let $ \omega $ be the one-column partition $ \omega := (1^n)$. The nodes of
  $ \omega $ are also totally ordered, with respect to the usual dominance order $ \lhd$, 
and hence there is
  a unique order preserving bijection
  \begin{equation}
    \Phi_{\theta}:   \tab(\blambda) \rightarrow \tab(\omega).
\end{equation}
  For example, for $ \theta= \theta_0 $ and $ \blambda= ( (1^5), (1^2), (1^6) )$ we have that $ \omega= (1^{13})$ and so
\begin{equation}
  \Phi_{\theta}:
  \left(\, \gyoung(;2,;1,;3,;7,;\once,:),\gyoung(;4,;6,:,:,:,:),\gyoung(;5,;9,;<10>,;<13>,;<8>,;<12>) \, \right)
  \mapsto
  \left(\, \gyoung(;2,;4,;5,;1,;6,;9,;3,;<10>,;7,;<13>,;\once,;<8>,;<12>) \, \right).  \, 
\end{equation}  
Note that $ \Phi_{\theta}(\bT^{\blambda}_{\theta}) = \T^{\omega}$.   
Let us now prove $a) $ of the Lemma. We have that 
\begin{equation} 
  \Phi_{\theta}(\bT)= \left(\, \, \gyoungx(3.5,<\bTI>,<\bTII>,\vdts,<k+1>,\vdts,<\bTj>,\vdts,<k>,\vdts,<\bTn>) \, \right), 
\, \, \, \, \, \, \, \, \, \, \, \, \, \, \, \, \, \, \, \, 
  \Phi_{\theta}(\Bs)= \left(\, \, \gyoungx(3.5,<\bTI>,<\bTII>,\vdts,<k>,\vdts,<\bTj>,\vdts,<k+1>,\vdts,<\bTn>) \, \right)
   \end{equation}
and so we have 
\begin{equation}
 \Phi_{\theta}( \shape(\Bs \! \mid_j))   =  \Phi_{\theta}( \shape(\bT \! \mid_j))
\end{equation}  
for all $ j \neq k$ and 
\begin{equation}
\Phi_{\theta}(  \shape(\Bs \! \mid_k))   \rhd  \Phi_{\theta}( \shape(\bT \! \mid_k))
\end{equation}  
and so $a)$ follows.
In order to prove $ b) $ of the Lemma, we get from $a)$ that for any $\blambda$-tableau $ \bT \neq \bT_{\theta}^{\blambda} $ 
there is a sequence of simple reflections $ s_{i_1}, \ldots, s_{i_N} $ such
that
\begin{equation}\label{sequenceRe} \bT \lhd_{\theta} \bT s_{i_1} \lhd_{\theta} \bT s_{i_1} s_{i_2} \lhd_{\theta} \ldots \lhd_{\theta} \bT s_{i_1} s_{i_2}
  \cdots s_{i_N} = \bT_{\theta}^{\blambda},
\end{equation}  
that is $ \bT \prec_{\theta} \bT_{\theta}^{\blambda}$.
Since this holds for any $ \bT \neq \bT_{\theta}^{\blambda} $ we deduce that $ \bT_{\theta}^{\blambda} $
is the unique maximal tableau in $ \tab(\blambda) $ with respect
to both $ \prec_{\theta} $ and $ \lhd_{\theta}$. 
In order to show that $ \bT_{\theta}^{\blambda} $ is also the unique maximal tableau in $ \std(\blambda) $
we use that if $ \bT \in \std(\blambda) $ then each term of the chain 
(\ref{sequenceRe}) also belongs to $ \std(\blambda) $. 
The Lemma is proved.
\end{proof}

We observe that if $ \blambda $ is not a one-column multipartition then there is in general 
not a unique maximal element in $ \std(\blambda) $ with respect to $ \prec_0$ or $ \lhd_0$. 
Consider for example $ \blambda = ((1), (2) ) $ with its two standard $ \blambda$-tableaux
\begin{equation}
  \bT^{\blambda}=  \left(\, \gyoung(1),\gyoung(2;3) \, \right),\, \, \, \, \, \, \, \, \, 
  \Bs=  \left(\, \gyoung(3),\gyoung(1;2) \, \right).
\end{equation}  
These are both maximal in $ \std(\blambda) $ with respect to $ \prec_0$ and $ \lhd_0$. 
This observation is the main reason why the methods of our paper do not generalize in a straightforward
way to general multipartitions.

\medskip

Let $ l(\cdot) $ be the  length function on $\mathfrak{S}_n$, viewed as a Coxeter group, and let $ < $ be the Bruhat order on $\mathfrak{S}_n$
with the convention that the identity element $ 1\in \mathfrak{S}_n $ is the largest element.
Let $ \lambda $ be a usual partition.
For $ \T \in \tab(\lambda) $ we define $ d(\T) \in \mathfrak{S}_n$ by the condition $ \T^{\lambda} d(\T ) = \T$. Since
the action of $ \Si_n$ is transitive and faithful we have that $ d(\T)$ is well defined and unique.
For $ \blambda $ a one-column multipartition and $ \bT \in \tab(\blambda) $ we define $ d(\bT)$ in a similar way,
using $ \bT_{\theta}^{\blambda}$. 
Our next aim is to show a compatibility between the Bruhat order on $ \Si_n $ and the order $ \lhd_{\theta} $
on $ \tab(\blambda)$.
In the case of the usual dominance order $ \lhd $ on $ \tab(\lambda) $ this result was
proved originally by Ehresmann.
In fact we shall
deduce our version of the Theorem from the original Ehresmann Theorem. Let us recall it.

\begin{theorem}\label{Teo-Ehr-Original}
  Suppose that $ \lambda $ is a partition of $ n $ and that $ \s,  \T \in \tab(\lambda) $ are row standard.
  Then we have that $d(\s) < d(\T)$ if and only if $\s \lhd {\T}$.
\end{theorem}
Here is our generalization of this Theorem. 
\begin{theorem}\label{Teo-Ehr}
  Let ${\blambda}$ be a one-column multipartition of $ n$ and suppose that ${\bT}$ and ${\Bs}$ are ${\blambda}$-{tableaux}.
  Then $d(\Bs) < d(\bT)$ if and only if $\Bs \lhd_{\theta} {\bT}.$
\end{theorem}
\begin{proof}
  Let again $ \omega $ be the one-column partition $ \omega = (1^n)$ and let 
  $ \Phi_{\theta}:   \tab(\blambda) \rightarrow \tab(\omega) $ be the order preserving
  bijection that was introduced in the proof of Lemma
  \ref{t-lamb-max}.
Recall that in general $ \Phi( \bT^{\blambda}_{\theta}) = \T^{\omega}$. But from this it follows
that for any $ \bT \in \tab(\blambda) $ we have $ d( \bT) = d(\Phi_{\theta}(\bT)) $. On the other hand, we have that
$ \Bs \lhd_{\theta} \bT $ if and only if $ \Phi( \Bs) \lhd \Phi(\bT) $ and so the Theorem follows from
the original Ehresmann Theorem, that is Theorem \ref{Teo-Ehr-Original}.
\end{proof}

Let $ \blambda \in \OnePar$. Then 
we conclude from the Theorem that the order relations $ \lhd_{\theta} $ on $ \tab(\blambda) $ are all isomorphic.
However, the restrictions of the order relations $ \lhd_{\theta} $  to the relevant subsets $ \std(\blambda) $ are not isomorphic.

\medskip
In general $\unlhd_{\theta}$ is not a total order on the set of tableaux, only a partial order.
On the other hand, on the set of tableaux of one-column multipartitions of $ n $  
there is related
{\color{black}{stronger}}
order $ <_{\theta}$ which is a total order. It
is the lexicographical order, defined via 
\begin{equation}\label{lexiorder}
  \bT <_{\theta} \Bs \, \, \mbox{ if there is } 1 \leq k \leq n \mbox{ such that } \bT \! \mid_j   = \Bs \! \mid_j
  \mbox{ for } j < k
\mbox{ but } \bT \! \mid_k   \lhd_{\theta} \, \,  \Bs \! \mid_k.
\end{equation}
It induces 
a total order on one-column multipartitions of $ n $ via
\begin{equation}\label{lexiorderpar}
  \blambda <_{\theta} \bmu \mbox{ iff } \bT_{\theta}^{\blambda} <_{\theta}\bT_{\theta}^{\bmu}.
\end{equation}  

There is an extension of $ <_{\theta} $ to the set of all one-column multipartitions that shall be of importance to us. 
It is given as follows. Let $ \blambda $ and $ \bmu$ be one-column multipartitions of $ m $ and $ n$
and assume that $ m < n $. Then we define 
\begin{equation}\label{lexiorderpar}
  \blambda <_{\theta} \bmu \mbox{ iff } \bT_{\theta}^{\blambda} \le_{\theta}  \bT_{\theta}^{\bmu} \! \mid_m.
\end{equation}  
For example if $ \gamma $ is an addable node for $ \blambda $ and $ \bmu $ is defined via
$ [\bmu] :=  [\blambda] \cup \gamma $ then we always have that $ \blambda <_{\theta} \bmu $. 
In general for $ k< n $ we define 
\begin{equation} \blambda \!\! \mid_k = \shape( \bT_{\theta}^{\blambda} \! \! \mid_k).   \end{equation}
Suppose that $ \blambda $ and $ \bmu $ are multipartitions of $ m $ and $ n $ and that 
$ m < n $. Then by definition $ \blambda   \le_{\theta} \bmu \!\! \mid_m $ iff $ \blambda   <_{\theta} \bmu $.

\medskip
In the following we shall be mostly interested in the orders related to the zero weighting 
and when we write $ \lhd $, $ <$, $ \prec$, $ \bT^{\blambda}$, etc we refer to 
$ \lhd_0 $, $ <_0$, $ \prec_0$, $ \bT^{\blambda}_0$, etc. We shall also mostly be interested in one-column multipartitions and 
therefore 'multipartitions' shall in the following refer to 'one-column multipartitions', unless otherwise stated.

\section{Generalized blob algebras}
In this section we define the family of algebras that we are interested in. Let $ \F $ be a field 
of characteristic $ p $, where $ p $ is either a prime or zero, and suppose that $ q \in \F \setminus \{ 1 \} $ is a primitive $ e$'th root of unity.
(Thus if $  p > 0$ we have $ \mbox{gcd}(e,p) =1 $).
Let $\II:=\mathbb{Z}/e\mathbb{Z}$. Fix a positive integer $ l $. The elements of $\bi = (i_1, \ldots, i_n) $ of $ \II^n $
are called \emph{residue sequences modulo $ e$}, or \emph{simply residue sequences}.
For $ \bi \in (i_1, \ldots, i_n) \in \II^n $ and $ j \in \II$, we define the \emph{concatenation} $ \bi j \in \II^{n+1} $ via
$ \bi j := (i_1, \ldots, i_n, j )$. 
The symmetric group $ \Si_n $ acts on the left on $ \II^n$ via permutation of
the coordinates $ \II^n $, that is  $s_k \cdot \bi   :=(i_1, \ldots,  i_{k+1}, i_k, \ldots, i_n)$.

Let $ \hat{\kappa} = (\hat{\kappa}_1, \ldots, \hat{\kappa}_l ) \in  {\mathbb Z}^l $ where $ l $ is
as before. 
Such a $ \hat{\kappa} $ is 
denoted a \emph{multicharge}. We let $ \kappa_i \in  \II$ be the image of $ \hat{\kappa}_i $ under the natural
projection and define
$ \kappa := (\kappa_1, \ldots , \kappa_l)  \in \II^n$. 
We shall throughout choose a representative for each $ \kappa_i$, also denoted by $ \kappa_i $, 
between $ 0 $ and $ e-1$. 
\begin{definition}\label{strongadj}
We say that $ \hat{\kappa} $ is \emph{strongly adjacency-free} if it satisfies 
\begin{itemize}
\setlength\itemsep{-1.1em}
\item[i)] $\hat{\kappa}_{i+1} - \hat{\kappa}_{i} \geq n $  \\
 \item[ii)] $ \kappa_i-\kappa_j \neq 0, \pm 1  \, \,\mbox{ mod }\,  e   \mbox{ for all } i\neq j $  \\
 \item[iii)] $ \kappa_1 \neq \kappa_l +2 \,\, \mbox{ mod }\,  e$ \\
 \item[iv)] $ \kappa_1 < \kappa_2 < \ldots < \kappa_l.$
\end{itemize}
\end{definition}
We shall in the following always assume that $ \hat{\kappa} $ is strongly adjacency-free; in particular
the inequality $e > 2l$ should always hold.

Our notion of a strongly adjacency-free multicharge is a generalization of 
the notion of an \emph{adjacency-free multicharge}, which was introduced in \cite{LiPl}
although already implicitly present in \cite{MW} and \cite{PlazaRyom}.
The difference between the two notions are the conditions $ iii) $ and $ iv) $ which are omitted in \cite{LiPl}.  
These extra conditions will be useful later on for our analysis of Garnir tableaux.


\medskip

We can now define our main object of study.
\theoremstyle{definition}
\begin{definition}{\label{defiblob}}
  Given integers $e,l,n>1$ and a strongly adjacency-free multicharge $\hat{\kappa} $
  the generalized blob algebra $\BB= \B$ of level $l$ on $n$ strings is the unital, associative
  $\F$-algebra on generators $$\{\psi_1,\dots,\psi_{n-1}\}\cup{\{y_1,\dots,y_n\}}\cup{\{e(\bi) \mid \bi\in{\II^n}\}}$$ subject to the following relations
\begin{equation}\label{orto1}
e(\bi)e(\bj)=\delta_{\bi, {\bj}}e(\bi)
\end{equation}
\begin{equation}\label{eq12}
 e(\bi)=0  \, \, \,  \textrm{if } \,  i_1\not\in{\{\kappa_1,\dots,\kappa_l\}}
\end{equation}
\begin{equation}\label{eq13}
e(\bi)=0 \,\,   \textrm{if} \,\,  i_1 \in{\{\kappa_1,\dots,\kappa_l\}} \, \,  \textrm{and } \, i_2=i_1+1
\end{equation}
\begin{equation}\label{eq13a}
y_1 e(\bi)=0 \,\,   \textrm{if} \,\,  i_1 \in{\{\kappa_1,\dots,\kappa_l\}}
\end{equation}
\begin{equation}\label{sum1}
\sum_{\bi\in{\II^n}}{e(\bi)}=1
\end{equation}
\begin{equation}\label{eq3}
y_re(\bi)=e(\bi)y_r
\end{equation}
\begin{equation}\label{eq4}
\psi_re(\bi )=e( s_k \cdot \bi)\psi_r
\end{equation}
\begin{equation}\label{eq5}
y_ry_s=y_sy_r
\end{equation}
\begin{equation}\label{eq6}
\psi_ry_s=y_s\psi_r\quad\textrm{if}\quad s\neq r,r+1
\end{equation}
\begin{equation}\label{eq7}
\psi_r\psi_s=\psi_s\psi_r\quad\textrm{if}\quad |s-r|>1
\end{equation}
\begin{equation}\label{eq8}
\psi_ry_{r+1}e(\bi)=(y_r\psi_r-\delta_{i_r,i_{r+1}})e(\bi)
\end{equation}
\begin{equation}\label{eq9}
y_{r+1}\psi_re(\bi)=(\psi_ry_r-\delta_{i_r,i_{r+1}})e(\bi)
\end{equation}
\begin{equation}\label{eq10}
\psi_r^2e(\bi)=\left\{\begin{array}{cc}
                          0 & \quad \textrm{if}\quad i_r=i_{r+1}  \\
                          e(\bi) & \quad \textrm{if}\quad i_r\neq i_{r+1},i_{r+1}\pm1  \\
                           (y_{r+1}-y_r)e(\bi)& \quad \textrm{if}\quad i_{r+1}=i_r+1  \\
                           (y_r-y_{r+1})e(\bi)& \quad \textrm{if}\quad i_{r+1}=i_r-1
                        \end{array}\right.
\end{equation}
\begin{equation}\label{eq11}
\psi_r\psi_{r+1}\psi_re(\bi)=\left\{\begin{array}{cc}
                                             (\psi_{r+1}\psi_r\psi_{r+1}-1)e(\bi) & \quad \textrm{if} \quad i_{r+2}=i_r=i_{r+1}-1 \\
                                               (\psi_{r+1}\psi_r\psi_{r+1}+1)e(\bi)& \quad \textrm{if} \quad i_{r+2}=i_r=i_{r+1}+1 \\
                                              (\psi_{r+1}\psi_r\psi_{r+1})e(\bi) & \quad \textrm{otherwise.}
                                           \end{array}\right.
\end{equation}
\end{definition}

The above definition of $ \B $ is the one used in \cite{bowman} and \cite{LiPl}, but it is not the original definition of the 
generalized blob algebra as presented in \cite{MW}. In
the final section of our paper we prove that the two definitions do coincide. 
For the original blob algebra the coincidence of these two definitions was proved in
\cite{PlazaRyom}.

\medskip
Let us take the opportunity to give the precise definition of the KLR-algebra, already mentioned above.
It was introduced independently in \cite{KhovanovLauda} and \cite{Rouq}.
\begin{definition}\label{KLR-algebra}
  The {\color{black}{cyclotomic KLR-algebra of type $ {A}^{(1)}_{e-1}$, or simply
    the KLR-algebra}},
    is the $ \F$-algebra $ \R$ on generators 
$$\{\psi_1,\dots,\psi_{n-1}\}\cup{\{y_1,\dots,y_n\}}\cup{\{e(\bi) \mid \bi\in{\II^n}\}}$$
subject to the same relations as for the blob algebra $ \B $ 
except for relation
(\ref{eq13}) which is omitted.
\end{definition}
  Let $ \pi: \R \rightarrow \B$ be the projection map from the KLR-algebra to $ \B$.
  Then, for simplicity of notation, we shall in general write $ x $ for $ \pi(x)$ when 
 $ x \in \R $.

\medskip
It follows from the relations that there is an antiinvolution $ \ast $ of $ \B$, and of $ \R $, that fixes the generators.

\medskip

There is a diagrammatical way to view this definition which is of importance for our work.
{\color{black}{It was introduced by Khovanov and Lauda in \cite{KhovanovLauda}.}}
A \emph{Khovanov-Lauda diagram} $ D$, or simply a \emph{KL-diagram}, on $n$ strings consists of
$n$ points on each of two parallel edges (the top edge and the bottom edge) and $n$ strings connecting
the points of the top edge with the points
of the bottom edge. Strings may intersect, but triple intersections are not allowed. Each string
may be decorated with a finite number of dots,
but dots cannot be located on the intersection of two strings. Finally, each string is labelled with an element of $\II$.
This defines two residue sequences $t(D),b(D)  \in \II^n$ associated with the diagram $D$
obtained by reading the residues of the extreme points from left to right.
{\color{black}{For the details concerning this definition, the reader should consult \cite{KhovanovLauda}.}}

\begin{example}
Let $e=4$ and $n=6$. Let $D$ be the following KL-diagram:
$$
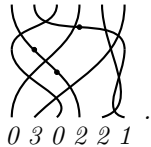
\begin{tikzpicture}[xscale=0.3,yscale=0.3]
\draw[thick] (0,0) to [out=90, in=270](3,5);
\draw[thick] (1,0) to [out=45,in=270] (5,5);
\draw[thick] (2,0) to [out=45,in=270] (0,3);
\draw[thick] (0,3) to [out=90,in=270] (2,5);
\draw[thick] (3,0) to [out=90,in=270] (0,5);
\draw[thick] (4,0) to [out=15,in=270] (4,5);
\draw[thick] (5,3) to [out=90,in=270] (1,5);
\draw[thick] (5,0) to [out=170,in=270] (5,3);
\node[below] at (0,0) {0};
\node[below] at (1,0) {3};
\node[below] at (2,0) {0};
\node[below] at (3,0) {2};
\node[below] at (4,0) {2};
\node[below] at (5,0) {1};
\draw[fill] (1,3) circle [radius=0.1];
\draw[fill] (2,2) circle [radius=0.1];
\draw[fill] (3,4) circle [radius=0.1];
\node at (6,0) {\color{black}{.}};
\end{tikzpicture}
$$
In this case the bottom sequence is $b(D)=(0,3,0,2,2,1)$ and the top sequence is $t(D)=(2,1,0,0,2,3).$
\end{example}

We can now define the diagrammatic algebra $\BB^{diag}= \B^{diag}$. As an $ \F$-vector space it consists
of the $\F$-linear
combinations of KL-diagrams on $n$ strings modulo planar isotopy and modulo the following relations:

\begin{equation}\label{mal-inicio}
\begin{tikzpicture}[xscale=0.5,yscale=0.5]
  \draw[thick](0,0)--(0,2);
  \draw[thick](1,0)--(1,2);
  \node at(2.5,1) {$\dots$};
  \draw[thick](4,0)--(4,2);
  \node[below] at (0,0) {$i_1$};
  \node[below] at (1,0) {$i_2$};
  \node[below] at (4,0) {$i_n$};
  \node at(5,1){$\quad=\quad0$};
  \node at(10,1){if $i_1\not\in{\{\kappa_1,\dots,\kappa_l\}}${\color{black}}};
\end{tikzpicture}
\end{equation}

\begin{equation}\label{otro-mal-inicio}
\begin{tikzpicture}[xscale=0.5,yscale=0.5]
  \draw[thick](0,0)--(0,2);
  \draw[thick](1,0)--(1,2);
  \node at(2.5,1) {$\dots$};
  \draw[thick](4,0)--(4,2);
  \node[below] at (0,0) {$i_1$};
  \node[below] at (1,0) {$i_2$};
  \node[below] at (4,0) {$i_n$};
  \node at(5,1){$\quad=\quad0$};
  \node at(12,1){if $i_1\in{\{\kappa_1,\dots,\kappa_l\}}$ and $ i_2 = i_1 +1${\color{black}}};
\end{tikzpicture}
\end{equation}

\begin{equation}\label{dot-al-inicio}
\begin{tikzpicture}[xscale=0.5,yscale=0.5]
  \draw[thick](0,0)--(0,2);
  \draw[thick](1,0)--(1,2);
  \draw[thick](4,0)--(4,2);
  \draw[fill] (0,1) circle [radius=0.1];
  \node[below] at (0,0) {$i_1$};
  \node[below] at (1,0) {$i_2$};
  \node[below] at (4,0) {$i_n$};
  \node at (5,1) {$\quad=\quad0$};
  \node at (2.5,1) {$\dots$};
  \node at(10,1){if $i_1 \in{\{\kappa_1,\dots,\kappa_l\}} $};
\end{tikzpicture}
\end{equation}

\begin{equation}\label{punto-arriba}
\begin{tikzpicture}[xscale=0.5,yscale=0.5]
\draw[thick] (0,0)--(2,2);
\draw[thick] (0,2)--(2,0);
\draw[fill] (1.5,0.5) circle [radius=0.1];
\draw[thick] (3,0)--(5,2);
\draw[thick] (3,2)--(5,0);
\draw[fill] (3.5,1.5) circle [radius=0.1];
\draw[thick] (7,0)--(7,2);
\draw[thick] (8,0)--(8,2);
\node at (2.5,1) {=};
\node at (6,1) {$-\delta_{ij}$\quad};
\node[below] at (0,0) {$i$};
\node[below] at (2,0) {$j$};
\node[below] at (3,0) {$i$};
\node[below] at (5,0) {$j$};
\node[below] at (7,0) {$i$};
\node[below] at (8,0) {$j$};
\end{tikzpicture}
\end{equation}

\begin{equation}\label{punto-abajo}
\begin{tikzpicture}[xscale=0.5,yscale=0.5]
\draw[thick] (0,0)--(2,2);
\draw[thick] (0,2)--(2,0);
\draw[fill] (1.5,1.5) circle [radius=0.1];
\draw[thick] (3,0)--(5,2);
\draw[thick] (3,2)--(5,0);
\draw[fill] (3.5,0.5) circle [radius=0.1];
\draw[thick] (7,0)--(7,2);
\draw[thick] (8,0)--(8,2);
\node at (2.5,1) {=};
\node at (6,1) {$-\delta_{ij}$\quad};
\node[below] at (0,0) {$i$};
\node[below] at (2,0) {$j$};
\node[below] at (3,0) {$i$};
\node[below] at (5,0) {$j$};
\node[below] at (7,0) {$i$};
\node[below] at (8,0) {$j$};
\end{tikzpicture}
\end{equation}
where $ \delta_{ij} $ is {\color{black}the} Kronecker delta. {\color{black} Moreover}
\begin{equation}\label{cruce-pasa}
\begin{tikzpicture}[xscale=0.5,yscale=0.5]
\draw[thick] (0,0) to [out=45,in=270](2,2);
\draw[thick] (2,0) to [out=90,in=315](0,2);
\draw[thick] (1,0) to [out=135,in=225] (1,2);
\draw[thick] (3,0) to [out=90,in=225](5,2);
\draw[thick] (5,0) to [out=135,in=270](3,2);
\draw[thick] (4,0) to [out=45,in=315](4,2);
\draw[thick] (7,0)--(7,2);
\draw[thick] (8,0)--(8,2);
\draw[thick] (9,0)--(9,2);
\node at (2.5,1) {=};
\node at (6,1) {$\quad+\alpha\quad$\quad};
\node[below] at (0,0) {$i$};
\node[below] at (1,0) {$j$};
\node[below] at (2,0) {$k$};
\node[below] at (3,0) {$i$};
\node[below] at (4,0) {$j$};
\node[below] at (5,0) {$k$};
\node[below] at (7,0) {$i$};
\node[below] at (8,0) {$j$};
\node[below] at (9,0) {$k$};
\end{tikzpicture}
\end{equation}
where $$\alpha=\left\{\begin{array}{cc}
                       -1 & \quad\textrm{if}\quad i=k=j-1 \\
                       1 & \quad\textrm{if}\quad i=k=j+1 \\
                       0 & \quad\textrm{otherwise}
                     \end{array}\right.$$

\begin{equation}\label{lazo}
\begin{tikzpicture}[xscale=0.5,yscale=0.5]
\draw[thick](0,0)to[out=90,in=270](1,1);
\draw[thick](1,1)to[out=90,in=270](0,2);
\draw[thick](1,0)to[out=90,in=270](0,1);
\draw[thick](0,1)to[out=90,in=270](1,2);
\node[below] at (0,0) {$i$};
\node[below] at (1,0) {$j$};
\node at (2,1){$=\beta$\quad};
\draw[thick] (3,0)--(3,2);
\draw[thick] (4,0)--(4,2);
\node[below] at (3,0) {$i$};
\node[below] at (4,0) {$j$};
\node at (5,1) {$+\gamma$\quad};
\draw[thick] (6,0)--(6,2);
\draw[thick] (7,0)--(7,2);
\node[below] at (6,0) {$i$};
\node[below] at (7,0) {$j$};
\draw[fill] (7,1) circle [radius=0.1];
\node at (8,1) {$-\gamma$\quad};
\draw[thick] (9,0)--(9,2);
\draw[thick] (10,0)--(10,2);
\node[below] at (9,0) {$i$};
\node[below] at (10,0) {$j$};
\draw[fill] (9,1) circle [radius=0.1];
\end{tikzpicture}
\end{equation}
where $$\beta=\left\{\begin{array}{cc}
                       1 & \quad\textrm{if}\quad |i-j|>1 \\
                       0& \quad\textrm{otherwise}
                     \end{array}\right.$$
and $$\gamma=\left\{\begin{array}{cc}
                      1 & \quad\textrm{if}\quad j=i+1 \\
                      -1 & \quad\textrm{if}\quad j=i-1 \\
                      0 &\quad\textrm{otherwise{\color{black}.}}
                    \end{array}\right.$$

The identity element $ 1 $ of $ \B^{diag}$ is the sum over all diagrams
$$\begin{tikzpicture}[xscale=0.5,yscale=0.5]
  \draw[thick](0,0)--(0,2);
  \draw[thick](1,0)--(1,2);
  \node at(2,1) {$\dots$};
  \draw[thick](3,0)--(3,2);
  \node[below] at (0,0) {$i_1$};
  \node[below] at (1,0) {$i_2$};
  \node[below] at (3,0) {$i_n$};
\end{tikzpicture}
$$
such that $ \bi := (i_1, i_2, \ldots, i_n ) $ belongs to $ \II^n$.

\medskip
The multiplication $DD'$ between two diagrams $D$ and $D'$ in $\B^{diag}$ is defined by vertical concatenation with $D$ above $D'$ if $b(D)=t(D')$.
If $b(D)\neq t(D')$ the product is defined to be zero.
We extend the product to all pairs of elements in $ \B^{diag}$ by linearity.

\medskip
The $ \F $-linear map from $ \B $ to $ \B^{diag} $ given by
\begin{equation}
\begin{tikzpicture}[xscale=0.5,yscale=0.5]
  \draw[thick](0,0)--(0,2);
  \draw[thick](1,0)--(1,2);
  \node at(2,1) {$\dots$};
  \node at (-1,1){$\mapsto$};
  \node at (-2,1){$e(\bi)$};
  \draw[thick](3,0)--(3,2);
  \node[below] at (0,0) {$i_1$};
  \node[below] at (1,0) {$i_2$};
  \node[below] at (3,0) {$i_n$};
  \node[below] at (3.3,1) {$,$};
\end{tikzpicture}
\begin{tikzpicture}[xscale=0.5,yscale=0.5]
  \draw[thick](0,0)--(0,2);
  \draw[thick](1.8,0)--(1.8,2);
  \draw[fill] (1.8,1) circle [radius=0.1];
  \node at(0.9,1) {$\dots$};
  \node at(2.8,1) {$\dots$};
  \node at(-1,1){$\mapsto$};
  \node at(-2.5,1){$y_re(\bi)$};
  \draw[thick](3.5,0)--(3.5,2);
  \node[below] at (0,0) {$i_1$};
  \node[below] at (1.8,0) {$i_r$};
  \node[below] at (3.5,0) {$i_n$};
  \node[below] at (3.8,1) {$,$};
\end{tikzpicture}
\begin{tikzpicture}[xscale=0.5,yscale=0.5]
  \draw[thick](1,0)--(1,2);
  \draw[thick](3,0)--(4,2);
  \draw[thick](3,2)--(4,0);
   \node at(2,1) {$\dots$};
  \node at(4.8,1) {$\dots$};
  \node at(0,1){$\mapsto$};
  \node at(-1.5,1){$\psi_re(\bi)$};
  \draw[thick](6,0)--(6,2);
  \node[below] at (1,0) {$i_1$};
  \node[below] at (3,0) {$i_r$};
  \node[below] at (4,0) {$i_{r+1}$};
  \node[below] at (6,0) {$i_n$};

\end{tikzpicture}
\end{equation}
defines an isomorphism between $\B$ and $\B^{diag}$.
In view of this, we shall write $\B^{diag}= \B$.

\medskip
We next show some useful relations that can be derived directly from the definitions.

\begin{lemma}\label{cruce-por-dot}
In $ \B$ we have:

\begin{center}
\begin{tikzpicture}[xscale=0.5,yscale=0.5]
  \draw[thick](0,0)--(1,2);
  \draw[thick](1,0)--(0,2);
  \node[below]at (0,0) {$i$};
  \node[below]at (1,0) {$i$};
  \node[below]at (3,0) {$i$};
  \node[below]at (4,0) {$i$};
  \node at (2,1){$\quad=\quad$};
  \draw[thick](3,0)to[out=90,in=270](4,1);
  \draw[thick](4,1)to[out=90,in=270](3,2);
  \draw[thick](4,0)to[out=90,in=270](3,1);
  \draw[thick](3,1)to[out=90,in=270](4,2);
  \draw[fill] (3,1) circle [radius=0.1];
  \node[] at (4.5,0) {$\color{black}.$};
\end{tikzpicture}
\end{center}
\end{lemma}
\begin{proof}
This is an immediate consequence of relations (\ref{punto-arriba}), (\ref{punto-abajo}) and (\ref{lazo}).
\end{proof}

\begin{lemma}\label{doblei}
In $\B$ we have:

\begin{center}
\begin{tikzpicture}[xscale=0.5,yscale=0.5]
  \draw[thick](0,0)--(0,2);
  \draw[thick](1,0)--(1,2);
  \node[below]at (0,-0.3) {$i$};
  \node[below]at (1,-0.3) {$i$};
  \node[below]at (3,-0.3) {$i$};
  \node[below]at (4,-0.3) {$i$};
  \node[below]at (6,-0.3) {$i$};
  \node[below]at (7,-0.3) {$i$};
  \node at (2,1){$\quad=\quad$};
  \draw[thick](3,0)to[out=90,in=270](4,1);
  \draw[thick](4,1)to[out=90,in=270](3,2);
  \draw[thick](3,2)--(3,2.3);
  \draw[thick](4,2)--(4,2.3);
  \draw[thick](0,2)--(0,2.3);
  \draw[thick](1,2)--(1,2.3);
  \draw[thick](6,2)--(6,2.3);
  \draw[thick](7,2)--(7,2.3);
  \draw[thick](0,0)--(0,-0.3);
  \draw[thick](1,0)--(1,-0.3);
  \draw[thick](3,0)--(3,-0.3);
  \draw[thick](4,0)--(4,-0.3);
  \draw[thick](6,0)--(6,-0.3);
  \draw[thick](7,0)--(7,-0.3);
  \draw[thick](4,0)to[out=90,in=270](3,1);
  \draw[thick](3,1)to[out=90,in=270](4,2);
  \draw[fill] (3,1) circle [radius=0.1];
  \draw[fill] (3,2) circle [radius=0.1];
  \node at (5,1){$\quad-\quad$};
  \draw[thick](6,0)to[out=90,in=270](7,1);
  \draw[thick](7,1)to[out=90,in=270](6,2);
  \draw[thick](7,0)to[out=90,in=270](6,1);
  \draw[thick](6,1)to[out=90,in=270](7,2);
  \draw[fill] (6,1) circle [radius=0.1];
  \draw[fill] (7,0) circle [radius=0.1];
  \node[] at (7.5,0) {$\color{black}.$};  
\end{tikzpicture}
\end{center}
\end{lemma}
\begin{proof}
This is a consequence of relations (\ref{punto-arriba}), (\ref{punto-abajo}) and Lemma \ref{cruce-por-dot}.
\end{proof}

\begin{lemma}\label{dot-pasa}
If $|i-j|>1$ then we have

\begin{center}
\begin{tikzpicture}[xscale=0.5,yscale=0.5]
  \draw[thick](0,0)--(0,2);
  \draw[thick](1,0)--(1,2);
  \draw[fill] (1,1) circle [radius=0.1];
  \node[below]at (0,0) {$j$};
  \node[below]at (1,0) {$i$};
  \node[below]at (3,0) {$j$};
  \node[below]at (4,0) {$i$};
  \node at (2,1){$\quad=\quad$};
  \draw[thick](3,0)to[out=90,in=270](4,1);
  \draw[thick](4,1)to[out=90,in=270](3,2);
  \draw[thick](4,0)to[out=90,in=270](3,1);
  \draw[thick](3,1)to[out=90,in=270](4,2);
  \draw[fill] (3,1) circle [radius=0.1];
  \node[] at (4.5,0) {$\color{black}.$};
\end{tikzpicture}
\end{center}
\end{lemma}
\begin{proof}
This is a direct consequence of the relations (\ref{punto-arriba}), (\ref{punto-abajo}) and (\ref{lazo}).
\end{proof}

\begin{lemma}\label{dot-salta}
If $|i-j|=1$ then we have
\begin{center}
\begin{tikzpicture}[xscale=0.5,yscale=0.5]
  \draw[thick](0,0)--(0,2);
  \draw[thick](1,0)--(1,2);
  \draw[fill] (1,1) circle [radius=0.1];
  \node[below]at (0,0) {$j$};
  \node[below]at (1,0) {$i$};
  \node[below]at (3,0) {$j$};
  \node[below]at (4,0) {$i$};
  \node[below]at (6,0) {$j$};
  \node[below]at (7,0) {$i$};
  \node at (2,1){$\quad=\quad$};
  \node at (5,1){$\quad\pm\quad$};
  \draw[thick](6,0)to[out=90,in=270](7,1);
  \draw[thick](7,1)to[out=90,in=270](6,2);
  \draw[thick](7,0)to[out=90,in=270](6,1);
  \draw[thick](6,1)to[out=90,in=270](7,2);
  \draw[thick](3,0)--(3,2);
  \draw[thick](4,0)--(4,2);
  \draw[fill] (3,1) circle [radius=0.1];
\end{tikzpicture}
\end{center}
where the positive sign appears when $j=i-1$ and the negative sign when $j=i+1.$
\end{lemma}
\begin{proof}
This is a direct consequence of relation (\ref{lazo}).
\end{proof}

\begin{lemma}\label{trio}
If $j=i+1$ then we have 
\begin{center}
\begin{tikzpicture}[xscale=0.5,yscale=0.5]
  \draw[thick](0,0)--(0,2);
  \draw[thick](1,0)--(1,2);
  \draw[thick](2,0)--(2,2);
  \node[below] at (0,0) {$i$};
  \node[below] at (1,0) {$j$};
  \node[below] at (2,0) {$i$};
  \node at (3,1) {$\quad=\quad$};
  \draw[thick](4,0)to[out=90,in=270](5,1);
  \draw[thick](5,0)to[out=90,in=270](6,1);
  \draw[thick](6,1)to[out=90,in=270](5,2);
  \draw[thick](5,1)to[out=90,in=270](4,2);
  \draw[thick](6,0)to[out=90,in=270](4,1);
  \draw[thick](4,1)to[out=90,in=270](6,2);
  \draw[fill] (4,1) circle [radius=0.1];
  \node[below] at (4,0) {$i$};
  \node[below] at (5,0) {$j$};
  \node[below] at (6,0) {$i$};
  \node at (7,1) {$\quad-\quad$};
  \draw[thick](8,0)to[out=45,in=270](10,1);
  \draw[thick](10,1)--(10,2);
  \draw[thick](10,0)to[out=135,in=270](9,1);
  \draw[thick](9,1)to[out=90,in=315](8,2);
  \draw[thick](9,0)to[out=135,in=270](8,1);
  \draw[thick](8,1)to[out=90,in=270](9,2);
  \node[below] at (8,0) {$i$};
  \node[below] at (9,0) {$j$};
  \node[below] at (10,0) {$i$};
\end{tikzpicture}
\end{center}
and if $j=i-1$ then we have that 
\begin{center}
\begin{tikzpicture}[xscale=0.5,yscale=0.5]
  \draw[thick](0,0)--(0,2);
  \draw[thick](1,0)--(1,2);
  \draw[thick](2,0)--(2,2);
  \node[below] at (0,0) {$i$};
  \node[below] at (1,0) {$j$};
  \node[below] at (2,0) {$i$};
  \node at (3,1) {$\quad=-\quad$};
  \draw[thick](4,0)to[out=90,in=270](5,1);
  \draw[thick](5,0)to[out=90,in=270](6,1);
  \draw[thick](6,1)to[out=90,in=270](5,2);
  \draw[thick](5,1)to[out=90,in=270](4,2);
  \draw[thick](6,0)to[out=90,in=270](4,1);
  \draw[thick](4,1)to[out=90,in=270](6,2);
  \draw[fill] (4,1) circle [radius=0.1];
  \node[below] at (4,0) {$i$};
  \node[below] at (5,0) {$j$};
  \node[below] at (6,0) {$i$};
  \node at (7,1) {$\quad+\quad$};
  \draw[thick](8,0)to[out=45,in=270](10,1);
  \draw[thick](10,1)--(10,2);
  \draw[thick](10,0)to[out=135,in=270](9,1);
  \draw[thick](9,1)to[out=90,in=315](8,2);
  \draw[thick](9,0)to[out=135,in=270](8,1);
  \draw[thick](8,1)to[out=90,in=270](9,2);
  \node[below] at (8,0) {$i$};
  \node[below] at (9,0) {$j$};
  \node[below] at (10,0) {$i$};
  \node[] at (10.5,0) {$\color{black}.$};
\end{tikzpicture}
\end{center}
\end{lemma}
\begin{proof}
This is a direct consequence of relation (\ref{cruce-pasa}) and Lemma \ref{cruce-por-dot}.
\end{proof}

\begin{lemma}\label{concatenation}
  { \color{black} Let $ n \ge 2 $ and let $ \iota_{n+1}$  be the
    concatenation on the right of a diagram in $ \B $ with a through line
  of fixed residue $ \iota $, as indicated in the following figure 
$$
\begin{tikzpicture}[xscale=0.3,yscale=0.3]
\draw[thick] (0,0) to [out=90, in=270](3,5);
\draw[thick] (1,0) to [out=45,in=270] (5,5);
\draw[thick] (2,0) to [out=45,in=270] (0,3);
\draw[thick] (0,3) to [out=90,in=270] (2,5);
\draw[thick] (3,0) to [out=90,in=270] (0,5);
\draw[thick] (4,0) to [out=15,in=270] (4,5);
\draw[thick] (5,3) to [out=90,in=270] (1,5);
\draw[thick] (5,0) to [out=170,in=270] (5,3);
\node[below] at (0,0) {$ i_1$};
\node[below] at (1,0) {$ i_2$};
\node[below] at (2,0) {$ i_3$};
\node[below] at (3,0) {$ i_4$};
\node[below] at (4,0) {$ i_5$};
\node[below] at (5,0) {$ i_6$};
\draw[fill] (1,3) circle [radius=0.1];
\draw[fill] (2,2) circle [radius=0.1];
\draw[fill] (3,4) circle [radius=0.1];

\node at (7,2) {$ \mapsto$};

\draw[thick] (9+0,0) to [out=90, in=270](9+3,5);
\draw[thick] (9+1,0) to [out=45,in=270] (9+5,5);
\draw[thick] (9+2,0) to [out=45,in=270] (9+0,3);
\draw[thick] (9+0,3) to [out=90,in=270] (9+2,5);
\draw[thick] (9+3,0) to [out=90,in=270] (9+0,5);
\draw[thick] (9+4,0) to [out=15,in=270] (9+4,5);
\draw[thick] (9+5,3) to [out=90,in=270] (9+1,5);
\draw[thick] (9+5,0) to [out=170,in=270] (9+5,3);

\draw[fill] (9+1,3) circle [radius=0.1];
\draw[fill] (9+2,2) circle [radius=0.1];
\draw[fill] (9+3,4) circle [radius=0.1];

\draw[thick, blue] (9+6,0)--(9+6,5);
\node[below] at (9+6,0){$i_{7}$};

\node[below] at (9+0,0) {$ i_1$};
\node[below] at (9+1,0) {$ i_2$};
\node[below] at (9+2,0) {$ i_3$};
\node[below] at (9+3,0) {$ i_4$};
\node[below] at (9+4,0) {$ i_5$};
\node[below] at (9+5,0) {$ i_6$};

\node at (16,0){\color{black}.};

\end{tikzpicture}
$$
Then $  \iota_{n+1}$ induces a (non-unital) algebra homomorphism 
$  \iota_{n+1}:   \B \rightarrow \BBB$.
It satisfies $ \iota_{n+1}(0) = 0$.}
\end{lemma}
\begin{proof}
  Each of the relations (\ref{mal-inicio}) to (\ref{lazo}) for $ \B$ maps under $ \iota_{n+1}$
  to a relation for $ \BBB$ and so 
$ \iota_{n+1} $ is well-defined. The second statement of the Lemma is obvious.
\end{proof}
We shall use the notation $  b \cdot \iota$ or $  b \, \iota$ for $\iota_{n+1}(b) $.
We remark that it can be shown that $ \iota_{n+1} $ is an embedding.

\section{A generating set $ \Basis$ for $ \B$.}
We now take the first steps towards the construction of our cellular basis for $ \B$.
\medskip

Let $ \blambda $ be a multipartition and let $ \gamma = (r,c,m)$ be a node of
$ [\blambda] $. 
Then 
we define the residue of $ \gamma $ via
\begin{equation}\label{thenwedefinetheresidue}
\textrm{res}(\gamma):=\kappa_m+c-r \in \II.
\end{equation}
Recall that a multipartition $ \blambda $ is assumed to be a one-column multipartition, unless otherwise stated. The nodes $ \gamma $ of a multipartition 
$ \blambda$ are of the form 
$ \gamma = (r,1,m)$ with residue
$ \textrm{res}(\gamma) =\kappa_m+1-r$.

Any $\blambda$-tableau $ \bT $ gives rise to a residue sequence $ \bi^{\bT} \in \II^n$ defined via
\begin{equation}
 \bi^{\bT} := (i_1, \ldots, i_n) \in \II^n  \, \mbox{ where} \,  i_j = \textrm{res}(\bT(j)).
\end{equation}

In the next couple of Lemmas and Corollaries we aim at showing that only the idempotents $ e( \bi^{\blambda}) $, with $ \blambda $ running over 
multipartitions, are needed in order to generate $ \B$.
Our proof for this is not straightforward and relies on several induction loops, all related to $ \blambda$. 
In essence our proofs are a chain of applications of the Lemmas \ref{cruce-por-dot} to \ref{concatenation} and 
could therefore have been formulated completely diagrammatically, in principle, but we choose to encode these 
Lemmas in an symbolic notation that we explain shortly. This symbolic notation has the advantage of enabling 
us to keep track of the induction parameter 
$ \blambda$. Our approach is therefore different from the approaches of  \cite{W}, \cite{bowman} that 
rely on manipulations of the diagrams themselves. Our proofs are rather comparable to the proofs of \cite{GLi}
and, in view of this, maybe surprisingly short, after all.

\medskip

Let $ \bmu_n^{max} = \color{black}{\bmu^{max}}$ be the multipartition introduced in (\ref{lambda-max}), which is
the unique maximal multipartition of $ n $  with respect to $\lhd$, and let us denote by $ \bT_n^{max}=
\bT^{max} $
the unique maximal $ \bmu_n^{max}$-tableau, as in Lemma \ref{t-lamb-max}.
We denote by $ \bi_n^{max} = \bi^{max} \in \II^n$ the corresponding residue sequence and by 
 $ e(\bi^{max}) \in \B$ the associated idempotent. We denote by $ [ \textrm{res} ( \bT^{max})]$
the corresponding residue diagram, obtained by writing $ \textrm{res}( \bT^{max}(k)) $
in the node $ \bT^{max}(k) $ of $ [\blambda]$.
For example, for $ n = 22, e = 10 $ and $ \kappa =( 0,2,4,7) $ we have the following residue diagram
\begin{equation}{\label{exampleMax}}
[ \textrm{res} ( \bT^{max})] =
  \left(\, \gyoung(;0,;9,;8,;7,;6,;5),\gyoung(;2,;1,;0,;9,;8,;7),
  \gyoung(;4,;3,;2,;<1>,;0,:), \gyoung(;7,;6,;5,;<4>,;3,:) \, \right)
\end{equation}
which gives rise to {\color{black}{the}} following residue sequence
\begin{equation} \bi^{max}= (0,2,4,7, 9,1,3,6,8,0,2,5,7,9,1,4,6,8,0,3,5,7) \in I_{10}^{22} \end{equation}
and corresponding idempotent
\begin{equation}\label{idempotentdiagramOLD}
\begin{tikzpicture}[xscale=0.5,yscale=0.5]
  \node at (-2.0,1.5) {$e(\bi^{max})=$};
  \draw[](0,0)--(0,3);
  \draw[](1,0)--(1,3);
  \draw[](2,0)--(2,3);
  \draw[](3,0)--(3,3);
  \draw[](4,0)--(4,3);
  \draw[](5,0)--(5,3);
  \draw[](6,0)--(6,3);
  \draw[](7,0)--(7,3);
  \draw[](8,0)--(8,3);
  \draw[](9,0)--(9,3);
  \draw[](10,0)--(10,3);
    \draw[](11,0)--(11,3);
    \draw[](12,0)--(12,3);
    \draw[](13,0)--(13,3);
    \draw[](14,0)--(14,3);
    \draw[](15,0)--(15,3);
    \draw[](16,0)--(16,3);
    \draw[](17,0)--(17,3);
    \draw[](18,0)--(18,3);
    \draw[](19,0)--(19,3);
    \draw[](20,0)--(20,3);
    \draw[](21,0)--(21,3);

  \node[below] at (0,0) {$0$};
  \node[below] at (1,0) {$2$};
  \node[below] at (2,0) {$4$};
  \node[below] at (3,0) {$7$};

  \node[below] at (4,0) {$9$};
  \node[below] at (5,0) {$1$};
  \node[below] at (6,0) {$3$};
  \node[below] at (7,0) {$6$};

  \node[below] at (8,0) {$8$};
  \node[below] at (9,0) {$0$};
  \node[below] at (10,0) {$2$};
  \node[below] at (11,0) {$5$};

  \node[below] at (12,0) {$7$};
  \node[below] at (13,0) {$9$};
  \node[below] at (14,0) {$1$};
  \node[below] at (15,0) {$4$};

  \node[below] at (16,0) {$6$};
  \node[below] at (17,0) {$8$};
  \node[below] at (18,0) {$0$};
  \node[below] at (19,0) {$3$};

  \node[below] at (20,0) {$5$};
  \node[below] at (21,0) {$7$};
  \node[] at (21.5,0) {$.$};
  \end{tikzpicture}
\end{equation}

We now introduce our symbolic notation. {\color{black}{Firstly}} we represent 
an idempotent like (\ref{idempotentdiagramOLD}) in the following way 
\begin{equation}\label{anidempotent}
e( \bi^{max}):=
(0,2,4,7 \mid  9,1,3,6 \mid 8, 0,2{\color{black},} 5 \mid 7, 1, 9, 4 \mid 6, 8, 0,3 \mid 5, 7 ) 
\end{equation}
where the separation lines $ \mid $ indicate jumps from a row to the next in $ \bmu^{max}$
(although the
separation lines are not always meant to have an exact meaning, but rather to be a help for the eye).
{\color{black}{Secondly}}
we introduce the following dot notation for expressions like $ y_{19} e( \bi^{max}) $ 
\begin{equation}
y_{19}e( \bi^{max}):=
(0,2,4,7 \mid  9,1,3,6 \mid 8, 0,2{\color{black},} 5 \mid 7, 1, 9, 4 \mid 6, 8, \overset{\bullet}{0},3 \mid 5, 7 ).
\end{equation}

{\color{black}
  \medskip
For any $ a \in \B $ we denote by $ \langle a \rangle $ the two-sided ideal in $ \B $ generated by $ a $.  
When $ a,b \in \B $ and $ b \in \langle a \rangle $ we say that \emph{$b$ factorizes over $a$}. 

\medskip
We write $ \bi \overset{k}{\sim} \bj $ if $ \bi = s_k \bj $ where $ i_k \neq i_{k+1}  \pm 1$
and we let $ \sim $ be the equivalence relation on $ \II^l  $ generated by all the $ \overset{k}{\sim}$'s. 
If $ \bi \overset{k}\sim \bj $ we say that $ \bi $ is obtained from $ \bj $ by \emph{freely moving} the string
of residue $ i_{k+1} $ past the string of residue $ i_k $. We shall often use this concept as follows. Suppose that    
$ \bi \sim \bj $. Then we have both $ e(\bi) \in \langle e(\bj) \rangle $ and $ e(\bj) \in \langle  e(\bi) \rangle $,
that is $ e(\bi)$ factorizes over $ e(\bj)$ and vice versa.
Indeed, if $ \bi \overset{k}{\sim} \bj $ then by relation (\ref{lazo})
we have that $ e(\bi) = \psi_k e(\bj) \psi_k $ as well as 
$ e(\bj) = \psi_k e(\bi) \psi_k $, from which the general case follows.
In particular,  we have in this situation that $ e(\bi) = 0 $ if and only if $ e(\bj) = 0 $. 
The same way one sees that if $ \bi \sim \bj $ where $ \bi  =w \bj $ for $w \in  \Si_n $, then 
for all $ r $ we have $ y_r e(\bi) \in \langle y_s e(\bj) \rangle $ and $ y_s e(\bj) \in \langle  y_r e(\bi) \rangle $
where $ s = w \cdot r $.

\medskip
If $ \bi \sim \bj $ we shall also write $ e( \bi ) \sim e( \bj )$ and $ y_r e( \bi ) \sim y_s e( \bj )$ where
$ r $ and $ s $ are related as before. When using the 
symbolic notation as in (\ref{anidempotent}) we associate with $  \sim $ a similar meaning.}

}

\medskip
We aim at proving that $ y_k e(\bi^{max}) = 0 $ for all $ k =1, \ldots, n $. This is straightforward for small $ k $, but gets 
more complicated when $ k $ grows. Let us illustrate the argument on a few small values of $ k $, using the above example (\ref{idempotentdiagramOLD}).

\medskip
For $ k= 1 $ we must show that
\begin{equation}
y_{1}e( \bi^{max})=
(\overset{\bullet}{0},2,4,7 \mid  9,1,3,6 \mid 8, 0,2{\color{black},} 5 \mid 7, 1, 9, 4 \mid 6, 8, {0},3 \mid 5, 7 )
\end{equation}
is equal to zero; this is however an instance of relation (\ref{dot-al-inicio}). For $ k=2 $ we must show that
\begin{equation}
({0},\overset{\bullet}{2},4,7 \mid  9,1,3,6 \mid 8, 0,2{\color{black},} 5 \mid 7, 1, 9, 4 \mid 6, 8, {0},3 \mid 5, 7 ) =0.
\end{equation}
Here we may move $2$ freely past $0$ and so
\begin{equation}
  ({0},\overset{\bullet}{2},4,7 \mid  \ldots  \mid 6, 8, {0},3 \mid 5, 7 )  \sim
  ({\color{black}{\overset{\bullet}{2},0}}, 4,7 \mid  \ldots \mid 6, 8, {0},3 \mid 5, 7 )   =0
\end{equation}
where the last equality follows from (\ref{dot-al-inicio}), once again.
The same
kind of argument shows that $ y_{3}e( \bi^{max})= y_{4}e( \bi^{max})=0$.
For these small values of $ k $, one can formulate these arguments diagrammatically. 
Here is the case $ k =4$:
\begin{equation}\label{idempotentdiagram}
\begin{array}{l}
 \color{black}{y_4e(  \bi^{max})   =  \psi_3  \psi_2 \psi_1
   \big(y_1e( s_1 s_2 s_3 \bi^{max}) \big) \psi_1  \psi_2 \psi_3=} \\  \\
 \begin{tikzpicture}[xscale=0.5,yscale=0.5]
  \draw[](0,0)--(0,3);
  \draw[](1,0)--(1,3);
  \draw[](2,0)--(2,3);
  \draw[fill] (-1,1.5) circle [radius=0.1];
    \draw[] (3,0) to [out=90,in=-90] (-1,1.5);
  \draw[] (-1,1.5) to [out=90,in=-90] (3,3);
  \draw[](4,0)--(4,3);
  \draw[](5,0)--(5,3);
  \draw[](6,0)--(6,3);
  \draw[](7,0)--(7,3);
  \draw[](8,0)--(8,3);
  \draw[](9,0)--(9,3);
  \draw[](10,0)--(10,3);
    \draw[](11,0)--(11,3);
    \draw[](12,0)--(12,3);
    \draw[](13,0)--(13,3);
    \draw[](14,0)--(14,3);
    \draw[](15,0)--(15,3);
    \draw[](16,0)--(16,3);
    \draw[](17,0)--(17,3);
    \draw[](18,0)--(18,3);
    \draw[](19,0)--(19,3);
    \draw[](20,0)--(20,3);
    \draw[](21,0)--(21,3);
  \node[below] at (0,0) {$0$};
  \node[below] at (1,0) {$2$};
  \node[below] at (2,0) {$4$};
  \node[below] at (3,0) {$7$};
  \node[below] at (4,0) {$9$};
  \node[below] at (5,0) {$1$};
  \node[below] at (6,0) {$3$};
  \node[below] at (7,0) {$6$};
  \node[below] at (8,0) {$8$};
  \node[below] at (9,0) {$0$};
  \node[below] at (10,0) {$2$};
  \node[below] at (11,0) {$5$};
  \node[below] at (12,0) {$7$};
  \node[below] at (13,0) {$9$};
  \node[below] at (14,0) {$1$};
  \node[below] at (15,0) {$4$};
  \node[below] at (16,0) {$6$};
  \node[below] at (17,0) {$8$};
  \node[below] at (18,0) {$0$};
  \node[below] at (19,0) {$3$};
  \node[below] at (20,0) {$5$};
  \node[below] at (21,0) {$7$};
  \node[below] at (22,1.8) {$=0$};
  \end{tikzpicture}  
\end{array}
\end{equation}
    {\color{black}where the last equality follows from the fact that
      $  y_1e( s_1 s_2 s_3 \bi^{max})  $, that is the middle part of the
      diagram (\ref{idempotentdiagram}), is equal to zero.}

 Let us now go on showing that $ y_k e(\bi^{max}) = 0 $ for $ k = 5,6,7,8$ corresponding to the second row
of the residue diagram $ [ {\rm res}(\bT^{max})]$. For $ k  = 5 $ we 
must show that
\begin{equation}\label{kfive}
y_{5}e( \bi^{max})=
({0},2,4,7 \mid  \overset{\bullet}{9},1,3,6 \mid 8, 0,2, 5 \mid 7, 1, 9, 4 \mid 6, 8, {0},3 \mid 5, 7 )=0.
\end{equation}
But $ \overset{\bullet}{9} $ moves freely past $ 7,4,2$ and so we have  
\begin{equation}
({0},2,4,7 \mid  \overset{\bullet}{9},1,3,6 \mid \ldots  \mid \ldots \mid 5, 7 ) \sim 
({0},\overset{\bullet}{9},{\color{black}{2}},4,7 \mid  1,3,6 \mid \ldots  \mid \ldots  \mid 5, 7 ) 
\end{equation}
which we must show to be zero. But using Lemma \ref{dot-salta} we have that 
\begin{equation}
  ({0},\overset{\bullet}{9},{\color{black}{2}}, 4,7 \mid  \ldots  \mid 5, 7 ) \in {\color{black}{\big \langle}}
  (\overset{\bullet}{0},{9},{\color{black}{2}},4,7 \mid  \ldots  \mid 5, 7 )  \, {\color{black}{, }} \, 
  ({9},0,{\color{black}{2}},4,7 \mid  \ldots   \mid 5, 7 ) {\color{black}{\big \rangle}} 
\end{equation}
{\color{black}{where $ \langle \cdot \rangle $ once again denotes ideal generation.}}
Here the first {\color{black}{ideal generator}} is zero by relation (\ref{dot-al-inicio}) whereas the second
{\color{black}{ideal generator}} is zero by 
relation (\ref{otro-mal-inicio}). The other cases $ k = 6,7,8 $ are treated essentially the same way.

Let us now consider the cases where $ k $ {\color{black}{corresponds}} to the third row of $ [ {\rm res}(\bT^{max})]$,
that is we show that $ y_{k}e( \bi^{max})=0 $ for $ k= 9,10,11,12$. For $ k = 9 $ we must show that 
\begin{equation}
y_{9}e( \bi^{max})=
({0},2,4,7 \mid  {9},1,3,6 \mid \overset{\bullet}{8}, 0,2, 5 \mid 7, 1, 9, 4 \mid 6, 8, {0},3 \mid 5, 7 )=0.
\end{equation}
But $ \overset{\bullet}{8} $ moves freely past 
$ 6,3 $ and $1$ and so we have 
\begin{equation}
y_{9}e( \bi^{max}) \sim
({0},2,4,7 \mid  {9},\overset{\bullet}{8},1,3,6 \mid  0,2, 5 \mid 7, 1, 9, 4 \mid 6, 8, {0},3 \mid 5, 7 )
\end{equation}
which we must show to be zero. But by Lemma \ref{dot-salta}
we have that 
\begin{equation}
  ({0},2,4,7 \mid  {9},\overset{\bullet}{8},1,3,6 \mid  \ldots  \mid 5, 7 ) \, {\color{black}{\in}} \, 
{\color{black}{\big \langle}}  ({0},2,4,7 \mid  \overset{\bullet}{9},{8},1,3,6 \mid  \ldots  \mid 5, 7 ) \,  , \, 
  ({0},2,4,7 \mid  {8},9,1,3,6 \mid  \ldots  \mid 5, 7 ) {\color{black}{\big \rangle}}.
\end{equation}
Here the first {\color{black}{generator}} is zero by (\ref{kfive}) and for
the second {\color{black}{generator}} we have that
\begin{equation}
  ({0},2,4,7 \mid  {8},9,1,3,6 \mid  \ldots  \mid 5, 7 )  \sim
  (7, 8,{0},2,4 \mid  9,1,3,6 \mid  \ldots  \mid 5, 7 )  
\end{equation}
which is zero by relation (\ref{otro-mal-inicio}). The other cases $ k = 10,11,12$ are treated similarly.
For $ k $ {\color{black}{corresponding}} to the next block, the inductive argument becomes more complicated
and we prefer to present it as part of the proof of the general 
statement $ y_k e(\bi^{max})=0 $.

\begin{lemma}\label{base-de-induccion}
In $ \B$ we have for all $1\leq k\leq n $ the following relations
\begin{equation}\label{claim1}y_k e(\bi^{max})=0= e(\bi^{max})y_k. \end{equation}
\end{lemma}
\begin{proof}
By (\ref{eq3}) we know that $ y_k $ and $ e(\bi^{max})$ commute and so we only need to prove the first relation.
 

\medskip
We prove it by induction on $ n $. For $ n=1 $ it is trivial. We next prove it for
a fixed $ n $, assuming that it holds for $ n_1 <  n $. For this fixed $ n$, we use induction on 
$ k $. 

\medskip
The basis step for this induction is $ 1 \le k \le l$, which is however easily handled using the same arguments 
as in the above example (\ref{anidempotent}) 
and the case $ l+1 \leq k \leq 2l $ where
{\color{black}{$ k $}}
belongs to the second
row of $ \bmu^{max}$ can also be treated this way. 
\medskip
Let us now consider the case
{\color{black}{$(m-1)l+1 \leq k \leq ml $}}
where $ m\geq 3$. 
Since {\color{black}{$(m-1)l+1 \leq k \leq ml $}}
we have that ${\color{black} k} $ belongs to the {\color{black}{$ m$'th}}
row of $ [\bmu^{max}]$. 
Suppose that $ \kappa_1^j, \ldots, \kappa_{l}^{j} $ are the residues of the $ j$'th row of  
$ [ {\rm res}(\bT^{max})]$
and that the residue of $ \bT^{max}(k) $ is $ A $. Then we must show
that
\begin{equation}\label{mustshowthat}
y_ke( \bi^{max})=
(\ldots \mid \kappa^{m-1}_1, \ldots ,A+1, \ldots, \kappa^{m-1}_l 
\mid  \kappa^{m}_1, \ldots ,\overset{\, \,  \bullet}{A}, \ldots, \kappa^{m}_l   \mid \ldots \, )  =0.
\end{equation}
Here $ A+1 $ is the residue of the node on top of $ \bT^{max}(k) $ and so we can move $  \overset{\, \,  \bullet}{A}$ 
freely over the residues between them. Hence (\ref{mustshowthat}) is equivalent to 
\begin{equation}\label{mustshowthat}
(\ldots \mid \kappa^{m-1}_1, \ldots ,A+1, \overset{\, \,  \bullet}{A},  \ldots, \kappa^{m-1}_l 
\mid  \kappa^{m}_1, \ldots ,\widehat{A}, \ldots, \kappa^{m}_l  \mid \ldots \, )  =0
\end{equation}
which by Lemma \ref{dot-salta} is equivalent to {\color{black}{the ideal}}
\begin{equation}{\label{sequence2}}
\begin{array}{l}
{\color{black}{\big \langle}} (\ldots \mid \kappa^{m-1}_1, \ldots ,\overset{\, \,  \bullet}{(A+1)}, {A},  \ldots, \kappa^{m-1}_l 
\mid  \kappa^{m}_1, \ldots ,\widehat{A}, \ldots, \kappa^{m}_l  \mid \ldots \, ),  \\
(\ldots \mid \kappa^{m-1}_1, \ldots ,A, {A+1},  \ldots, \kappa^{m-1}_l 
\mid  \kappa^{m}_1, \ldots ,\widehat{A}, \ldots, \kappa^{m}_l  \mid \ldots \, ) {\color{black}{\big \rangle}}
\end{array}
\end{equation}
{\color{black}{being zero}}.
Here the first {\color{black}{ideal generator}} is zero by induction since 
\begin{equation}
(\ldots \mid \kappa^{m-1}_1, \ldots ,\overset{\, \,  \bullet}{(A+1)} ) =0 \end{equation}
by the inductive hypothesis 
on $ n$: this is the residue sequence of a $ \bT^{max}_{n_1} $ where $ n_1 < n $.
Here we also used that concatenation maps zero to zero by Lemma \ref{concatenation}.
We therefore focus on the second {\color{black}{ideal generator}} of ({\ref{sequence2}}), that is 
\begin{equation}\label{eqdia}
(\ldots \mid \kappa^{m-1}_1, \ldots ,A, {A+1},  \ldots, \kappa^{m-1}_l 
\mid  \kappa^{m}_1, \ldots ,\widehat{A}, \ldots, \kappa^{m}_l  \mid \ldots \, )
\end{equation}
which is obtained from the original sequence $ e( \bi^{max})$ by moving $ A$ past $ A+1 $. We have that 
$ y_k e(\bi^{max}) =0$ if and only if this sequence (\ref{eqdia}) is zero.
In (\ref{eqdia}) 
we now move $ A $ further to the left until it hits its first obstacle which will be $ A-1$: this is so due the combinatorial structure of $ [ \bT^{max}]$ and
strong adjacency-freeness of $ \hat{\kappa}$.
On top of the node of residue $ A $ there is a node of residue $ A-1$ that can be freely moved to the right 
until it stands next to $ A$. Doing this we find that 
(\ref{eqdia}) is zero if
\begin{equation}\label{eqdiaI}
(\ldots A(A-1)A \ldots  \mid \kappa^{m-1}_1, \ldots ,\widehat{A}, {A+1},  \ldots, \kappa^{m-1}_l 
\mid  \kappa^{m}_1, \ldots ,\widehat{A}, \ldots, \kappa^{m}_l  \mid \ldots \, ) 
\end{equation}
is zero. We now apply 
Lemma \ref{trio} to the triple $ A(A-1)A$ and get that (\ref{eqdiaI}) is zero if {\color{black}{the ideal }}
\begin{equation}{\label{sequence5}}
\begin{array}{l} {\color{black}{\big \langle}}
(\ldots \overset{\, \,  \bullet}{A}A(A-1) \ldots  \mid \kappa^{m-1}_1, \ldots ,\widehat{A}, {A+1},  \ldots, \kappa^{m-1}_l 
\mid  \kappa^{m}_1, \ldots ,\widehat{A}, \ldots, \kappa^{m}_l  \mid \ldots \, ) {\color{black}{, }}
 \\
(\ldots (A-1)AA \ldots  \mid \kappa^{m-1}_1, \ldots ,\widehat{A}, {A+1},  \ldots, \kappa^{m-1}_l 
\mid  \kappa^{m}_1, \ldots ,\widehat{A}, \ldots, \kappa^{m}_l  \mid \ldots \, ) {\color{black}{\big \rangle}}
\end{array}
\end{equation}
is zero.
As before, by induction on $ n $ the first {\color{black}{generator}} is here equal to zero
and so $ y_k e(\bi^{max}) = 0 $ if and only if the second term of
({\ref{sequence5}}) is zero. We now go on the same way, moving $ A-1 $ to the left, until it hits a residue $ A-2$
and as before $ y_k e(\bi^{max}) = 0 $ if the interchanging of those nodes produces a diagram which is zero.
Continuing in this way, the interchanging of nodes will finally take place in the first two rows of $ [\bmu^{max}]$, where by relations (\ref{mal-inicio}) and
(\ref{otro-mal-inicio}) it does produce zero.
\end{proof}

We have the following consequence of the Lemma. 
\begin{corollary}\label{cor-base-de-induccion}
  Suppose that $ \iota \in \II $ and that the concatenation $ \bi_n^{max} \iota $ is not of the form
  $ \bi^{\blambda} $ for $ \blambda $ any multipartition of $n+1$. Then we have that
\begin{equation}\label{claim} e(\bi_n^{max} \iota)=0. \end{equation}
\end{corollary}
\begin{proof}
We have that
\begin{equation}\label{corI}
e( \bi_n^{max} \iota )=
(\ldots \mid \kappa^{m-1}_1,  \ldots, \kappa^{m-1}_l 
\mid  \kappa^{m}_1,  \ldots, \kappa^{m}_l   \mid \ldots \mid \iota\, ).
\end{equation}
By the strong adjacency-freeness $ \iota $ moves here freely to the left until it hits another $ \iota $ or a pair
$ \iota(\iota-1)$. In the first case, using Lemma \ref{doblei} we replace the appearing
$ \iota \iota $ by $ \overset{\, \, \bullet}{\iota}i $,
and get by the Lemma that $ e( \bi_n^{max} \iota )= 0$, as claimed. In the second case, we replace
$ \iota(\iota-1) \iota$ by a linear combination of $ \overset{\, \, \bullet}{\iota}\iota (\iota-1) $
and $ (\iota-1) \iota \iota$. Proceeding as in the Lemma, we finally find
that this is zero.
\end{proof}

Let us illustrate the Corollary on the
example
\begin{equation}
  \left(\, \gyoung(;0,;9,;8,;7,;6,;5),\gyoung(;2,;1,;0,;9,;8,;7),
  \gyoung(;4,;3,;2,;<1>,;0,:), \gyoung(;7,;6,;5,;<4>,;3,:) \, \right)
\end{equation}
already considered in ({\ref{exampleMax}}). Here we can use $ \iota \neq 4,6,9,2 $ in the Corollary.
We then conclude from the Corollary that
$$  e(0,2,4,7, 9,1,3,6,8,0,2,5,7,9,1,4,6,8,0,3,5,7, \iota) =0  $$
for these choices of $ \iota$.

\medskip 
We generalize 
the previous Lemma and Corollary to arbitrary multipartitions in the following way.
{\color{black}{Recall that $ <  $ is the total order introduced in (\ref{lexiorderpar}). }}

\begin{lemma}\label{uppertriangularI}
 For $ \blambda $ any multipartition of $ n $ {\color{black}{and for $ 1 \le k \le n $}} we have that
\begin{equation}{\label{trian1}}
y_k e(\bi^{\blambda}) = e(\bi^{\blambda}) y_k =  \sum_{ \bmu > {\blambda}} D_{\bmu} 
\end{equation}
where the sum runs over multipartitions $ \bmu $ of $n$ 
{\color{black}{and $ D_{\bmu}$ factorizes over $ e(\bi^{\bmu})   $.}}
Suppose moreover that $ D_{\blambda} $ is any element of $ \B $ 
{\color{black}{and that $D_{\blambda}$ factorizes over $e(\bi^{\blambda}) $}} and {\color{black}{assume}} 
$ \iota \in \II $.
Then we have that
\begin{equation}{\label{trian2}} D_{\blambda}\cdot \iota= \sum_{ \bmu > {\blambda}} C_{\bmu} \end{equation}
where $ \bmu $ runs over multipartitions of $ n+1$
{\color{black}{and $ C_{\bmu}$ factorizes over $ e(\bi^{\bmu})$}.}
Furthermore, if $ \bi^{\blambda} \iota $ is not of the form $ \bi^{\bnu} $ for any multipartition $ \bnu $ of
$ n+1$ then we have that
\begin{equation}{\label{trian21}} D_{\blambda}\cdot \iota= \sum_{ \bmu  \mid_n > {\blambda}} C_{\bmu} \end{equation}
where once again the sum runs over multipartitions $ \bmu $ of $n+1$ 
{\color{black}{and $ C_{\bmu}$ factorizes over $ e(\bi^{\bmu}).$}}
\end{lemma}

\begin{proof}
  We first give an example which might be useful to have in mind while going
  through the
  arguments of the actual proof.
For $ n= 28$, $ e=9 $ and $ \blambda = ( (1^6), (1^4), (1^9), (1^9)) $ we have the following
residue diagram for $ \bT^{\blambda}$
\begin{equation}  \left(\, \gyoung(;0,;8,;7,;6,;5,;4,:,:,:),\gyoung(;2,;1,;0,;8,:,:,:,:,:),\gyoung(;4,;3,;2,;1,;0,;8,;7,;6,;5),
  \gyoung(6,;5,;4,;3,;2,;1,;0,;8,;7) \right).
\end{equation}

\medskip
In this case, in order to {\color{black}{prove}} ({\ref{trian1}})
we must show for $1 \le  i \le {\color{black}{27}} $
that $ y_{i}e({\bi^{\blambda}})  $
is a linear combination
$\sum_{ \bmu > {\blambda}} D_{\bmu}  $ as indicated
and for ({\ref{trian21}}) we must show that
for $ \iota  \in \II \setminus \{4,6\} $ we have that $ D_{\blambda} \cdot \iota  $ is a linear combination
$\sum_{ \bmu \mid_n > {\blambda}} C_{\bmu}  $ as indicated.

We now prove all statements of the Lemma by induction on $ n$, the basis case
$ n= 1 $ being straightforward.
We first prove ({\ref{trian1}}) by induction on $ k $. For $ k < n $ we use the inductive hypothesis on $ n $ to write $ y_k e(\bi^{\blambda} \! \mid_k)  $ 
in the form 
\begin{equation}\label{ASBEFORE}
y_k e(\bi^{\blambda} \! \! \mid_k) = \sum_{ \bmu > {\blambda  \mid_k}} D_{\bmu}
\end{equation}
where the sum runs over multipartitions $ \bmu $ of $ k $
{\color{black}{and $ D_{\bmu} \in  \langle e(\bi^{\bmu}) \rangle  $.}}
Let $ \bi^{\blambda} = (i_1, i_2, \ldots, i_n ) $. We then 
get $y_k e(\bi^{\blambda})=  y_k e(\bi^{\blambda} \! \! \mid_k \!  i_{k+1} \cdots i _n) $ in 
the form
\begin{equation}\label{hastheform}
y_k e(\bi^{\blambda}) = \sum_{ \btau > \bmu > {\blambda  \mid_k}} D_{\btau}
\end{equation}  
by concatenating each $ D_{\bmu} $ on 
the right with $ i_{k+1} \cdots i _n $ and using in each step the inductive hypothesis for ({\ref{trian2}}).
Here $ \bmu $ is as
in {\color{black}{(\ref{ASBEFORE})}}
whereas $ \btau $ runs over multipartitions of $n$. 
But $ \btau > \bmu > {\blambda \!  \mid_k} $ 
implies $ \btau > {\blambda  } $ and so (\ref{hastheform}) has the form 
indicated in ({\ref{trian1}}).

In order to show ({\ref{trian1}}) for $ k = n $ we return 
to our symbolic notation. We have 
\begin{equation}
e( \bi^{\blambda})=
(\kappa_1^1, \ldots, \kappa_{l_1}^{1} \mid
\kappa_1^2, \ldots, \kappa_{l_2}^{2}
\mid \cdots \mid
\kappa_1^r, \ldots, \kappa_{l_r}^{r})
\end{equation}
where $ \kappa_1^j, \ldots, \kappa_{l_j}^{j} $
are the residues of the $ j$'th row of $ [\blambda]$. In this notation, in order to show ({\ref{trian1}})
we must show that
\begin{equation}\label{must-show}
y_n e( \bi^{\blambda})=
(\kappa_1^1, \ldots, \kappa_{l_1}^{1} \mid
\kappa_1^2, \ldots, \kappa_{l_2}^{2}
\mid \cdots \mid
\kappa_1^r, \ldots, \overset{\, \, \bullet}{A})=\sum_{ \bmu > {\blambda}} D_{\bmu}
\end{equation}
where
$ A = \kappa_{l_r}^{r}$.

We now move $\overset{\,\, \bullet}{A} $ freely to the left until it meets its first obstacle, which by
strong adjacency-freeness is
$ A+1 $ coming from the node on top of the node of $\overset{\,\, \bullet}{A} $. 
We next use Lemma \ref{dot-salta}
to replace our sequence involving $  (A+1) \overset{\,\, \bullet}{A} $ by a linear combination of sequences involving 
$   (\overset{\,\, \bullet}{A+1}) A$ and $    A (A+1)$. As
{\color{black}{in the proof of (\ref{hastheform})}}
the first term involving
$(  \overset{\,\, \bullet}{A+1}) A$ is of the indicated form by induction hypothesis and we must therefore consider the
second term $  A(A+1)$.
We here move $ A $ freely to the left until it meets its first obstacle which must be
$ A $, $ A+1$ or $A-1$. If it is $A$ we use Lemma \ref{doblei} to replace $ AA $ by $ \overset{\,\, \bullet}{A}A$
and can once again use the induction hypothesis. If it is $ A-1$, the situation gives rise to a triple $ A(A-1)A$ where
the first $ A$ comes from the residue on top of the node of $A-1$. On this triple, we use Lemma \ref{trio}
to rewrite $ A(A-1)A$ as a linear combination of $ \overset{\,\, \bullet}{A}A (A-1) $ and
$ (A-1)AA$.
Here the first term is dealt with using the induction hypothesis for {\color{black}{(\ref{trian1})}}, whereas
{\color{black}{the second term is dealt with using the induction
    hypothesis for (\ref{trian2})}}.

We now consider the third case where $ A$ meets $A+1$. (In the previous Lemma \ref{base-de-induccion},
this case did not occur). But this case corresponds to a 'gap' in the diagram, where $ A$ can
be positioned giving rise to the diagram $ \bmu $ of a multipartition that satisfies $ \bmu > \blambda$.
Summing up, this proves the inductive step of ({\ref{trian1}}). The $ \bmu $'s that appear in
the final expansion ({\ref{trian1}}) are exactly those that arise from this last case.

\medskip
Let us now focus on the claims ({\ref{trian2}}) and ({\ref{trian21}}).
Clearly it is enough to show them for $ D_{\blambda} = e(\bi^{\blambda}) $ so let us do that.
We first note that ({\ref{trian2}}) is a consequence of ({\ref{trian21}}). 
Indeed, if $ \bi^{\blambda} \iota  $ is not of the form $ \bi^{\bnu} $ for any multipartition $ \bnu$
we have from ({\ref{trian21}}) that
\begin{equation}
  D_{\blambda}\cdot \iota= \sum_{ \bmu  \mid_n > {\blambda}} C_{\bmu} = 
\sum_{ \bmu  > {\blambda}} C_{\bmu}
\end{equation}
where we for the last equality used that in general $ \bmu > \bmu \! \mid_n $,
see the definition of $ > $ given in (\ref{lexiorderpar}).
On the other hand, if $ \bi^{\blambda} \iota = \bi^{\bnu}  $ for a multipartition $ \bnu$ of $ n+1 $, then we have that 
$ \bnu > \blambda $  
and $e(\bi^{\blambda} \iota) =  e( \bnu )  = C_{\bnu}$ and so 
({\ref{trian2}}) also holds in this case.

\medskip
Let us now prove ({\ref{trian21}}) by downwards induction on $ <$.
For $ \bi^{\blambda } = \bi^{max} $, it holds by Corollary \ref{cor-base-de-induccion}.
We now fix an arbitrary multipartition $ \blambda $ and assume that ({\ref{trian21}})
has been proved 
for multipartitions $ \bnu $ such that $ \bnu > \blambda$. Then 
in the above sequence notation, and writing $ A $ for $ \iota$, for ({\ref{trian21}}) we must show that
\begin{equation} e(\bi^{\blambda})  \cdot \iota = 
(\kappa_1^1, \ldots, \kappa_{l_1}^{1} \mid
\kappa_1^2, \ldots, \kappa_{l_2}^{2}
\mid \cdots \mid
\kappa_1^r, \ldots, \kappa_{l_r}^{r} \mid A )  = \sum_{ \bmu \mid_n > {\blambda}} C_{\bmu}
\end{equation}
where $ A $ is positioned in the $n+1$'st position.
Since we assume that the sequence is not of the form $ \bi^{\bnu} $ for $ {\bnu} $ for any multipartition
we can move $ A $ to the left until it meets its first obstacle, which must be $ A $, $ A-1 $ or $A+1$.
If it is $ A $ we proceed essentially as before:
we use Lemma \ref{doblei} to replace $ AA $ by $ \overset{\,\, \bullet}{A}A$
and can now use the induction hypothesis. Indeed, if $ \overset{\,\, \bullet}{A}$ is
situated in the $ k$'th position we are dealing with
$y_k e(\bi^{\blambda})=  y_k e(\bi^{\blambda} \! \! \mid_k \!  i_{k+1} \cdots i _n i _{n+1}) $
where $ i _{n+1} = \kappa_{l_r}^{r} $ and so on for the other $i_j$'s.
Using the inductive hypothesis for $ n $ on 
({\ref{trian1}}) and ({\ref{trian2}})
we get, arguing as in connection with
(\ref{hastheform}), that 
\begin{equation}\label{arguing1}
  y_k e(\bi^{\blambda} \! \! \mid_k \!  i_{k+1} \cdots i _n) =
 \sum_{ \btau > {\blambda  }} D_{\btau}
\end{equation}
where $ \btau $ runs over multipartitions of $ n $.
Finally, we use the inductive hypothesis for $ < $ to write 
\begin{equation}\label{arguing2}
e(\bi^{\blambda})  \cdot \iota
=   y_k e(\bi^{\blambda} \!  \! \mid_k \!  i_{k+1} \cdots i _n  i _{n+1}) =
\sum_{ \btau > {\blambda  }} D_{\btau} \cdot i _{n+1} =
\sum_{ \bmu > \btau  > {\blambda  }} D_{\bmu} = \sum_{ \bmu \mid_n   > {\blambda  }} D_{\bmu}
\end{equation}
where the last equality follows from the fact that $ \btau $ and $ \bmu $ run over multipartitions of $ n $ and $ n+1 $.
Hence (\ref{arguing2}) 
has the form required for ({\ref{trian21}}).

If the first obstacle is $ A-1 $ we essentially argue as before: 
the situation gives rise to a triple $ A(A-1)A $ which we
rewrite, using Lemma \ref{trio}, as a linear combination
of $ \overset{\,\, \bullet}{A}A (A-1) $ and $ (A-1) AA$.
Arguing as for (\ref{arguing1}) and (\ref{arguing2}) we get the term involving
$ \overset{\,\, \bullet}{A}A (A-1) $ in the form indicated in 
({\ref{trian2}}), whereas for the term involving $ (A-1) AA$
{\color{black}{we use the inductive hypothesis for (\ref{trian2})}}.
  
Finally, if the first obstacle is $ A+1$ we also argue as before, essentially.
Indeed, in this situation there is a gap where $ A $ can be placed. This gives rise to a multipartition $ \btau $ of
$ k $ such that $ \btau > \blambda \! \! \mid_k $ where $ k$ is the position of $ A$ and so we get, arguing as before, that
\begin{equation}
  e(\bi^{\blambda})  \cdot \iota  =
 e(\bi^{\btau}   i_{k+1} \cdots i _n  i_{n+1}) =  \sum_{ \bmu \mid_n > {\blambda  }} D_{\bmu}.
\end{equation}
This finishes the proof of the Lemma.
\end{proof}

\begin{corollary}\label{idempotent-expansion}
For each $ \bi \in \II^n $ there is an expansion in $ \B$ of the form 
\begin{equation}\label{idemptentexpansion}
e(\bi) = \sum_{ \bmu } D_{\bmu}
\end{equation}
where the sum runs over multipartitions $ \bmu $ of $n$
{\color{black}{and $ D_{\bmu} $ factorizes over $ e(\bi^{\bmu})   $}}.
\end{corollary}
\begin{proof}
We argue by induction on $ n $, the base case $ n=1 $ being trivial.
Assuming that 
(\ref{idemptentexpansion}) holds for $ n-1$ we prove it for $ n$. 
Suppose that $ \bi= (i_1,  \ldots, i_{n-1}, i_n) $ and set $ \bi_{n-1} = (i_1,  \ldots, i_{n-1}) $.
Then by induction we have that 
\begin{equation}
e(\bi_{n-1}) = \sum_{\, \,\, \, \bmu_{n-1} } D_{\bmu_{n-1}}
\end{equation}
where $ \bmu_{n-1}$ runs over multipartitions of $(n-1)$ and where 
$ D_{\bmu_{n-1}} $ {\color{black}{factorizes over}} $ e(\bi^{\bmu_{n-1}})   $. 
Using ({\ref{trian2}}) of the previous Lemma \ref{uppertriangularI} we then get 
\begin{equation}
e(\bi) = e(\bi_{n-1}) i_{n} = \sum_{\, \,\, \, \bmu_{n-1} } D_{\bmu_{n-1}} i_{n} = 
\sum_{\, \,\, \, \bmu_{n-1} }\sum_{\, \,\, \, \bnu > \bmu_{n-1} } D_{\bnu} 
\end{equation}
and so $ e(\bi)  $ is of the form claimed in (\ref{idemptentexpansion}).
\end{proof}

For any $ w \in \Si_n $ we choose once and for all a reduced expression $   s_{i_1} s_{i_1} \cdots s_{i_N} $ and define $ \psi_{{w}} \in \B $ via 
this expression
\begin{equation}\label{official}
\psi_{{w}} := \psi_{i_1} \psi_{i_1} \cdots \psi_{i_N}.
\end{equation}
Note that $ \psi_{{w}} $ depends on the choice of reduced expression, not just on $ w$.
We denote by \emph{official reduced expression for $w$} the expression used in (\ref{official}).
If $ w_1=   s_{j_1} s_{j_1} \cdots s_{j_N} $ is another, \emph{'unofficial'}, reduced expression for $ w $ 
then the error term in using $ w_1 $ instead of $ w $ can be controlled, in the sense that we have that 
\begin{equation}\label{official}
  \psi_{{w}} - \psi_{j_1} \psi_{j_1} \cdots \psi_{j_N} = \sum_{  \underline{k}\in
    {{\color{black}{\No^n}} , v \in \Si_n,  w < v}}
  c_{\underline{k} ,v} \, y^{\underline{k}} \,\psi_v =
  \sum_{  \underline{k}\in {{\color{black}{\No^n}} , v \in \Si_n,  w < v}}
  d_{\underline{k} ,v} \,\,\psi_v \,  y^{\underline{k}}  
\end{equation}
where $ c_{\underline{k},v}, d_{\underline{k},v} \in \F $ and 
where for $ \underline{k} = (k_1, \ldots, k_n )  \in  \No^n$ we define $ y^{\, \underline{k}} :=
  y_1^{k_1} \cdots y_n^{k_n}  \in \B$. 

  \medskip
Let $ \blambda \in \OnePar $ be a one-column multipartition and suppose that $ \Bs, \bT \in \tab(\blambda) $. 
For the associated group elements $ {d(\Bs)}, {d(\bT)} \in \Si_n$ we have $ \psi_{d(\Bs)}, \psi_{d(\bT)} \in \B $
defined via the official reduced expression for $ d(\Bs) $ and $ {d(\bT)} $.
We then set 
\begin{equation}
m_{\Bs \bT} = \psi_{d(\Bs)}^{\ast} e(\bi^{\blambda}) \psi_{d(\bT)} \in \B
\end{equation}
and define $ \Basis \subseteq \B $ via 
\begin{equation}\label{basisdef}
 \Basis := \{ m_{\Bs \bT} \mid \Bs, \bT \in \std(\blambda), \blambda \in \OnePar \}.
\end{equation}
A main goal of our paper is to show that $ \Basis $ is a cellular basis for $ \B$. 
Our first step towards this goal is to show that $  \Basis $ is a generating set for $ \B$. 
We start with the following Lemma.

\begin{lemma}\label{westart}
{\color{black}{Suppose that $ D_{\blambda} \in \B $ factorizes over $ e(\blambda)   $.}}
Then there is an expansion of the form 
\begin{equation}\label{otrainduction}
D_{\blambda} = 
 \sum_{\Bs, \bT \in \tab(\bmu), \, \bmu \ge \blambda}   c_{\Bs \bT} m_{\Bs \bT} 
\end{equation}
where $  c_{\Bs \bT} \in \F $.
\end{lemma}
\begin{proof}
It is known that 
\begin{equation}\label{KLRbasis}
{ \mathcal S } := \{ e(\bi) \,y^{\, \underline{k}}\, \psi_w  \mid \bi \in \II^n, \underline{k} \in \No^n, w \in \Si_n \} 
\end{equation}
spans the KLR-algebra $ \R$ over $ \F$, see (2.7) of \cite{brundan-klesc} and
section 2.3 of \cite{KhovanovLauda}. In fact, any permutation of the three factors 
of $ { \mathcal S } $ also gives an $\F$-spanning set for $ \R $ over $ \F$. But by definition $\B$ is a quotient of
$ \R$ and so these sets also span $ \B$ over $ \F$.

We now prove (\ref{otrainduction}) using downwards induction on $ < $.
The induction basis is given by the multipartition $ \blambda := \bmu_n^{max}$, introduced in (\ref{lambda-max}).
We may assume that $ D_{\blambda} =   a\, e(\bi^{\blambda}) \,b $ where $ a, b \in B $, since
$D_{\blambda} $ is a linear combination of such expressions.
We now expand $ a $ in terms of the variation of $ { \mathcal S }$ that uses the product order 
$ \psi_w  y^{\, \underline{k}}\,  e(\bi)  $ 
and then expand $ b  $ in terms of $ { \mathcal S }$. Inserting, we find expressions of the form 
\begin{equation}\label{expansionNumero1}
  D_{\blambda} =  \sum_{ v, w, \underline{k}_1, \underline{k}_2}  c_{ v, w, \underline{k}_1, \underline{k}_2} \psi_{v}
  y^{\, \underline{k}_1} e(\bi^{\blambda}) y^{\, \underline{k}_2} \psi_{w} = 
\sum_{ v,w}  c_{ v, w} \psi_{v}
   e(\bi^{\blambda})  \psi_{w} 
\end{equation}
where we used Lemma \ref{base-de-induccion} for the second equality.
For each appearing $ v,w $ we must now show 
that $ \psi_{v} e(\bi^{\blambda})  \psi_{w} $ is a linear combination of 
$ m_{\Bs \bT}  $ where $ \Bs, \bT \in \tab(\blambda) $.
We set $ \Bs := \bT^{\blambda} v^{-1} $ and $ \bT := \bT^{\blambda} w $. 
Then we have by definition that $ d(\Bs) = v^{-1} $ and $ d(\bT) = w $
and so 
\begin{equation}\label{expansionNumero2}
  D_{\blambda} =  
\sum_{ v,w}  c_{ v, w} \psi_{v}
e(\bi^{\blambda})  \psi_{w}  =
\sum_{ \Bs, \bT}  c_{ \Bs \bT} m_{\Bs \bT}
\end{equation}
and so we obtain the required expansion for $ D_{\blambda} $, at least in the basis case
$ \blambda = \bmu_n^{max}$. 

We next show the existence of the expansion 
(\ref{otrainduction}) for $ D_{\blambda} $ for a general $ \blambda $, assuming that it exists for all $ \bmu > \blambda$.
Once again we may assume that $ D_{\blambda} =  a \,e(\bi^{\blambda}) \,b $ where $ a, b \in \B $ and 
once again we expand $ a $ in terms of the variation of $ { \mathcal S }$ that uses the product order 
$ \psi_w  y^{\, \underline{k}}\,  e(\bi)  $ 
and $ b  $ in terms of $ { \mathcal S }$. Inserting, we now get an expression of the form 
\begin{equation}\label{expansionNumero1}
  D_{\blambda} =  \sum_{ v, w, \underline{k}_1, \underline{k}_2}  c_{ v, w, \underline{k}_1, \underline{k}_2} \psi_{v}
  y^{\, \underline{k}_1} \, e(\bi^{\blambda}) y^{\, \underline{k}_2} \, \psi_{w} = 
\sum_{ v,w}  c_{ v, w} \psi_{v}
   e(\bi^{\blambda})  \psi_{w} + 
\sum_{ \bmu > \blambda}   D_{\bmu} 
\end{equation}
where we this time used Lemma \ref{uppertriangularI} for the last equality.
Arguing as we did in the inductive basis step we now rewrite
$ \sum_{ v,w}  c_{ v, w} \psi_{v}   e(\bi^{\blambda})  \psi_{w} $ as a linear combination of $ m_{\Bs \bT} $'s
and then get
\begin{equation}\label{expansionNumero1}
  D_{\blambda} =  \sum_{ \Bs, \bT \in \tab(\blambda) } c_{\Bs \bT}  m_{\Bs \bT} +
\sum_{ \bmu > \blambda}   D_{\bmu} {\color{black}{.}}
\end{equation}

We now use the inductive hypothesis on the terms $ D_{\bmu}  $ to conclude the proof of the Lemma.
\end{proof}

\begin{lemma}\label{westart2}
The subset of $ \B $ given by 
\begin{equation}\label{tabgen}
 \{ m_{\Bs \bT} \mid \blambda \in \OnePar, \, \Bs, \bT \in \tab(\blambda) \}
\end{equation}
spans $ \B $ over $ \F$.
\end{lemma}
\begin{proof}
Choose $ b \in \B $ and expand it in terms of $ { \mathcal S }$ as follows
\begin{equation}{\label{firstexpansion}}
b = \sum c_{\bi, \underline{k}, w} \, e(\bi) \,y^{\, \underline{k}}\, \psi_w
\end{equation}
where $ c_{\bi, \underline{k}, w} \in \F$.
Using Corollary \ref{idempotent-expansion} we write each appearing $ e(\bi) $ as a linear combination 
of $ D_{\bmu}$'s where $ \bmu$ runs over multipartitions
{\color{black}{and $ D_{\bmu} $ factorizes over $ e(\bi^{\bmu})$.}}
Inserting this in ({\ref{firstexpansion}}) we find that any $ b \in \B $ is a linear combination of $D_{\bmu}$'s. We
can then apply the previous Lemma \ref{westart} to conclude the proof of the Lemma.
\end{proof}

Our next goal is to show that the non-standard tableaux are not needed in (\ref{tabgen}).
Our method for proving this is an adaption of Murphy's method using Garnir tableaux,  see \cite{Mat} and \cite{Murphy1}. 

\medskip
Let $\blambda$ be a multipartition and $\Bg$ a $\blambda$-tableau. We say that $\Bg$ is a \emph{Garnir tableau}
if there is an $1\leq i< n$ such that 
\begin{itemize}
\setlength\itemsep{-1.1em}
\item[a)]
$\Bg$ is not standard, but $\Bg s_i$ is standard. \\
\item[b)]
If $s\in S$ and $\Bg s\rhd\Bg$ then $s=s_i.$
\end{itemize}
Here are some examples
\begin{equation}\label{herearesome}  
\left(\, \gyoung(;2,;<{1}>,:,:),\gyoung(;3,;5,;6,;7),\gyoung(;4,:,:,:)   \right), \, \, \, \, 
\left(\, \gyoung(;3,;2,:,:),\gyoung(;4,;5,;6,;7),\gyoung(;1,:,:,:)   \right), \, \, \, \, 
\left(\gyoung(;1,;6,;<\doce>,;<\quince>),\gyoung(;2,;7,;\trece,:),\gyoung(;3,;\once,;\diez,:),\gyoung(;4,;8,;<\catorce>,:), \gyoung(;5,;9,:,:)\right).
\end{equation}

We note that a Garnir tableau of shape $ \blambda $ is \emph{not} uniquely 
determined by its 'point of non-standardness' as can be seen on the first two examples
of (\ref{herearesome}). This is opposed to the classical situation.

In order to get a better description of Garnir tableaux we introduce some further notation.
Let $ \blambda $ be a one-column multipartition
and let $\gamma=(r,1, m) $ be a node of 
$ [\blambda] $, which does not belong to the first row of $ [ \blambda]$.
We then denote by $\gamma^+$ the node $(r-1,1, {\color{black}{m}}) $ of $ [ \blambda]$, that is $ \gamma^+ $ is the
node of $ [ \blambda]$ that is situated on top of
$ \gamma$ in $ [\blambda] $.
We then define the \emph{Garnir snake} 
of $\gamma$ as the following interval in $ [ \blambda] $ with respect to $ \lhd $
\begin{equation}
  \Snake(\gamma) := [\gamma,\gamma^+] = \{ \tau \in [\blambda] \mid  \gamma \unlhd \tau \unlhd  \gamma^+ \}
  {\color{black}{.}}
\end{equation}  
We also define 
\begin{equation}
  \bn_{Snake(\gamma)} := \{ i \in \bn \mid \bT^{\blambda}(i ) \in [\gamma,\gamma^+] \}
\end{equation}
that is $ \bn_{Snake(\gamma)} $ is the set of numbers that are used to fill 
in $ \Snake(\gamma)  $ for $ \bT^{\blambda}$.

\medskip
For $ \blambda \in \OnePar$ and $ \gamma = (r,1,m) $ a node of $ [\blambda] $, not belonging to the first row, we define \emph{the classical Garnir
tableau} $ \Bg_{clas, \gamma}$ 
by setting $ \Bg_{clas, \gamma}(i ) := \bT^{\blambda}(i ) $ for $ i \notin \bn_{Snake(\gamma)} $ 
and by requiring that 
the numbers from $ \bn_{Snake(\gamma)} $ are filled in consecutively from left to right in $ \Snake(\gamma)$
except for an upwards jump from $ \gamma $ to $ \gamma^+ $. 
Here is an example 
with $ \gamma= (3,1,3)$
\begin{equation}\label{classicalGarnirExample}
\Bg_{clas, \gamma}=\left(\gyoung(;1,;6,;8),\gyoung(;2,;7,;9),\gyoung(;3,;\once,;\diez),\gyoung(;4,:,:), \gyoung(;5,;\doce,;\trece)\right).
\end{equation}
It should be noted that $ \Bg_{clas, \gamma} $ is \emph{not} a Garnir in the classical sense, as considered for example by Murphy and Mathas.
On the other hand, it 
is similar to the classical Garnir tableaux in the sense that if we view 
the components of $ \blambda $ as the columns of an ordinary partition (possibly with 
'missing' nodes as in the example) then $ \Bg_{clas, \gamma} $ becomes a Garnir tableau in the classical sense.

\medskip
We need another class of Garnir tableaux that we denote $ \tilde{\Bg}_{ \gamma}$.
They are defined by filling in the numbers from 
$ \bn_{Snake(\gamma)} $
into $ \Snake(\gamma)  $ in increasing order, beginning with $ \gamma $, then $ \gamma^+ $ and
the other nodes of the row of $ \gamma^+ $ and finally the remaining nodes of the row of $ \gamma$.
Here is an example with $ \gamma = (3,1,3) $ 
\begin{equation}\label{tildeGarnir}
\tilde{\Bg}_{ \gamma}=\left(\gyoung(;1,;6,;\once),\gyoung(;2,;7,;\doce),\gyoung(;3,;9,;8),\gyoung(;4,:,:), \gyoung(;5,;\diez,;\trece)\right).
\end{equation}

\medskip
Recall the weak order $\succ$ on $ \tab(\blambda) $. The following Lemma relates it
to Garnir tableaux. Set first $ \nstd(\blambda) := \tab(\blambda) \setminus \std(\blambda) $, that
is $ \Bs \in \nstd(\blambda)$ if and only if $ \Bs $ is a non-standard $ \blambda$-tableau.
\begin{lemma}\label{lemma garnir1}
  Suppose that $\bT \in \nstd(\blambda)$. Then
\begin{itemize}
\setlength\itemsep{-1.1em}
\item[a)]
   The tableau $\bT$ is a maximal in $\nstd(\blambda) $ with respect $\succ$
    if and only if $\bT$ is a Garnir tableau. \\
\item[b)] If $\bT$ is a maximal in $\nstd(\blambda) $ with respect $\rhd$ then $\bT$ is a Garnir tableau.
\end{itemize}
\end{lemma}
\begin{proof}
Let us first prove $a)$ of the Lemma.
Assume that $\bT$ is a maximal tableau in $\text{NStd}(\blambda)$ with respect to $\succ$.
Then for all $ s_i \in S $ we have that either $ \bT s_i \lhd \bT $ or $ \bT s_i \in \std(\blambda) $. 
If $ \bT s_i \lhd \bT $ for all $ i $ we have that $ \bT = \bT^{\blambda} $ which contradicts that
$ \bT \in \nstd(\blambda) $. Hence there is an $ s_{i_0} $ such that $ \bT s_{i_0} \rhd \bT $ and for this $ s_{i_0} $ we
have $ \bT s_{i_0} \in \std(\blambda) $ by maximality of $ \bT $ in $ \nstd(\blambda)$. On the other hand, there can
only be one $ s_{i_0} $ with this property. Indeed, suppose that also $ \bT s_{j_0} \rhd \bT $. Setting
$ \Bu : = \bT s_{i_0} $ and $ \Bv : = \bT s_{j_0} $ we have that $ \Bu $ and $  \Bu  s_{i_0} s_{j_0} $ are standard
tableaux, whereas $ {\color{black}{ \Bu s_{i_0}}}$ is non-standard. This is only possible if $ i_0 = j_0$ and so $ \bT  $ is
a Garnir tableau, as claimed.

Now assume that $\bT$ is not a maximal tableau in $\nstd(\blambda)$
with respect to $\succ.$ Then there is an $s\in S$ such that $\bT s\rhd\bT$ and $\bT s\in \nstd(\blambda).$
This implies that $\bT$ is not a Garnir tableau.

We now show $ b) $ of the Lemma.
If $\bT$ is a maximal tableau in $\nstd(\blambda)$ with respect to $\rhd$ then $\bT$ is
also a maximal tableau in $\nstd(\blambda)$ with respect to $\succ$, since $ \succ $ is a weaker order than
$ \rhd$, and so $ \bT $ must be a Garnir tableau by $a) $.
This proves $ b) $ of the Lemma.
\end{proof}

 The converse of $b) $ of the Lemma does not hold as can be seen in the following example.
 Let $\blambda=(1^2,1^2,1^2,1^2,1)$ and define 
\begin{equation}\label{theconverse}
 \Bg_1=\left(\gyoung(;1,;7),\gyoung(;2,;8),\gyoung(;5,;4),\gyoung(;6,;9),\gyoung(3,:)\right), \Bg_2=
 \left(\gyoung(;1,;3),\gyoung(;2,;8),
 \gyoung(;5,;4),\gyoung(;6,;9),\gyoung(7,:)\right).
 \end{equation}
 Then both $\Bg_1 $ and $ \Bg_2$ are Garnir tableaux, and it is easy to see that $\Bg_1 \rhd \Bg_2 $ and
 so $\Bg_2$ is not a maximal tableau in $\nstd(\blambda)$ with respect to $\rhd.$

\medskip

\begin{corollary}\label{fundamental-garnir1}
  Let $\bT$ be a $\blambda$-tableau which is non-standard.
  Then there exists a Garnir tableau $\Bg$ and a $w\in \Si_n$
  such that $\bT=\Bg w$ and $l(d(\bT))=l(d(\Bg))+l(w)$.
\end{corollary}
\begin{proof}
This is a consequence of $ a) $ of Lemma \ref{lemma garnir1}. 
\end{proof}  


Let us now give our characterization of Garnir tableaux.
\begin{lemma}\label{description garnir}
  Given a multipartition $\blambda$ of $n$ and let $\Bg$ be a $\blambda$-tableau. Then 
 $\Bg$ is a Garnir tableau if and only if there is a node $\gamma \in[\blambda]$, not belonging to
  the first row, and an $ i_0 \in \bn $ such that
\begin{itemize}
\setlength\itemsep{-1.1em}
\item[(1)]
 $\Bg(i_0) = \gamma $ and $\Bg(i_0+1) =\gamma^+$. \\
\item[(2)] For all $ i \neq i_0 $ we have $ \Bg(i) \rhd \Bg(i+1) $. \\
\item[(3)] For all $ i \in \bn \setminus  \bn_{\Snake(\gamma)} $ we have that $ \Bg(i) = \bT^{\blambda}(i)$.
\end{itemize}
\end{lemma}
\begin{proof}
  Suppose first that $\Bg$ is a Garnir tableau. Then $\Bg$ is not standard and maximal with respect to $ \prec$ 
  and hence there is an $ i_0 \in\bn$ such that $\Bg s_{i_0}$ is standard.
  The entries $i_0 $ and $ i_0+1$ belong to the same component (column) of $[\blambda]$
  and $\Bg(i_0+1)\rhd \Bg(i_0)$. Let $\gamma=\Bg(i_0+1)$ and $\beta=\Bg(i_0)$.
  Suppose that $\beta^+ \neq \gamma$ and choose $ a \in \bn $ such that $ \Bg(a)=\beta^+$.
  Then $\gamma\rhd\beta^+$ and since $\Bg s_{i_0}$ is standard we have that $i_0<a<i_0+1,$ a contradiction.
  Therefore $\beta=\gamma^+$ and
  by definition $\Bg(i_0)^+=\Bg(i_0+1)$.

  Since $\Bg$ is a Garnir tableaux, we have for $i\neq i_0$ that $\Bg\rhd\Bg s_i$
  and then $\Bg(i)\rhd\Bg(i+1)$, see $a$) of Lemma \ref{t-lamb-max}.

  Let us say that $ i \in \bn $ defines a \emph{simple non-inversion} if $ \Bg(i) \rhd \Bg(i+1) $ and that 
$ i \in \bn $ defines a \emph{simple inversion} if $ \Bg(i) \lhd \Bg(i+1) $. With this terminology we 
have so far proved that $ i_0 $ is the only simple inversion of $ \bn$, all other elements are simple non-inversions.

  Let $k_0=\min(\Bg^{-1}{(\Snake(\gamma)})) $ and $k_1=\max(\Bg^{-1}({\Snake(\gamma)})) $. Since $ i_0 $ is the only inversion of $ \bn $ we have that 
  $ k_0-1  $ appears before $ k_0 $ in $ \Bg$ whereas
  $ k_0-2  $ appears before $ k_0-1 $ and so on until $1$. On the other hand, no $ j > k_0  $ can appear before $ k_0 $ in $ \Bg$, since 
for the smallest such $ j $ we would have that $ j -1 $ is a inversion distinct from $ i_0$.
We have thus showed that for $ i = 1,2, \ldots, k_0 -1 $ we have that $ \Bg(i) = \bT^{\blambda}(i) $.
Similarly, one shows that also for $ i = k_1 +1, k_1 +2,\ldots, n  $ we have that $ \Bg(i) = \bT^{\blambda}(i) $.
Thus we have that $ \Bg^{-1}{(\Snake(\gamma)}) = \bn_{\Snake(\gamma)} $ and that 
$ \Bg$ verifies the conditions (1), (2) and (3) of the Lemma.

Finally, if $ \Bg $ is a $ \blambda$-tableau verifying the conditions (1), (2) and (3) of the Lemma, 
then clearly $ \Bg $ is a Garnir tableau.
\end{proof}

For the next Lemma we need condition $ iii) $
from Definition \ref{strongadj} of strong adjacency-freeness.

\begin{corollary}\label{garnireq}
Let $ \blambda $ be a multipartition and let $ \gamma \in [\blambda]$. 
Suppose that $ \Bg_1 $ and $ \Bg_1 $ are Garnir tableaux of the same shape $ \blambda$ 
with respect to the same $ \gamma $ as in part $ (1) $ of the previous Lemma
\ref{description garnir}. Then $ e(\bi^{\Bg_1}) \sim e(\bi^{\Bg_2}) $.
\end{corollary}
\begin{proof}
  It is enough to prove that for any Garnir tableau $\Bg= \Bg_1$, satisfying the conditions of the Corollary,
  we have that $ \Bg_1 \sim \Bg_{clas, \gamma} $. Let $ \g$ be the one line (ordinary) partition
  $ \g = ( |\bn_{Snake(\gamma)} | ) $. Then we can view $  \Bg \! \mid_{\bn_{Snake(\gamma)}}  $
  as a $ \g$-tableau $ \T(\Bg) $ by reading the numbers in $ Snake(\gamma)$ from left to right.
 The Garnir tableaux from (\ref{theconverse}) correspond for example to the $ \g$-tableaux
\begin{equation}\label{correspondforexample}
  \T(\Bg_1)  =\gyoung(;7;8;<\color{red}4>;<\color{red}5>;6;3) \, , \, \, \,\, \, \,\, \, \,\, \, \, \T(\Bg_2) =
  \gyoung(;3;8;<\color{red}4>;<\color{red}5>;6;7) 
\end{equation}
where $ \g = (6) $, whereas $ \Bg_{clas, \gamma} $ in general corresponds to
$ \T^{\g} $ (on the numbers $ \bn_{Snake(\gamma)} $), that is  
\begin{equation}
 \T^{\g} =  \gyoung(;3;4;<\color{red}5>;<\color{red}6>;7;8)
\end{equation}
in this case.
Since $ \hat{\kappa} $ is strongly adjacency free, we have on the other hand that
the residues of all of the nodes of $ Snake(\gamma) $, except $ \gamma $ and $ \gamma^+ $, differ by
$2$ or more. Let now $ w \in  \Si_n$ be such that $ \T(\Bg)w =  \T^{\g} $ and choose 
a reduced expression $ w = s_{i_1 } \cdots  s_{i_N } $ for $ w$. 
Then, for all $ j $, we have that $ s_{i_{j+1} } $ does not interchange the numbers appearing in the nodes
corresponding to $ \gamma $ and $ \gamma^{+} $ in $ \T_j := \T(\Bg) s_{i_1 } \cdots  s_{i_j }$.
For example, for $ \T(\Bg_1) $ in (\ref{correspondforexample})
the sequence $  s_{i_1 }, \ldots,   s_{i_N } $ never interchanges two numbers in the positions colored with
red, and similarly for $ \T(\Bg_2) $. The Corollary follows from this.
\end{proof}

We have the following Lemma. 
\begin{lemma}\label{firstMain}
  Suppose that $ \blambda \in \OnePar $ and that $ \Bs, \bT \in \tab(\blambda) $.
  If $ \bT \in \nstd(\blambda) $ then there is an expansion
  \begin{equation}\label{upexpansiona}
    m_{\Bs \bT} = \sum_{\bT_1 \in \std(\blambda),  \bT_1 \rhd \bT,  } c_{ \Bs \bT_1} m_{\Bs \bT_1}  +
    \sum_{ \bmu > \blambda, \Bs_2, \bT_2 \in \std(\bmu)} c_{\Bs_2 \bT_2} m_{\Bs_2 \bT_2} 
  \end{equation}
where $c_{ \Bs \bT_1}, c_{\Bs_2 \bT_2} \in \F $. A similar statement holds for $ \Bs$.
\end{lemma}
\begin{proof}
We shall argue via downwards induction on $ \blambda $ with respect to $ < $. Let us first consider the case
$ \blambda = \bmu_n^{max}$.
We consider $ m_{\Bs  \bT} $ for  $\Bs, \bT \in \tab(\blambda) $
and suppose that $ \bT \in \nstd(\bmu^{max}) $.
We show using downwards induction on $ \bT $
with respect to $ \lhd $ that 
$  m_{\Bs \bT}  $, {\color{black}{for $ \bT \in  \nstd(\bmu^{max}) $}},
can be written in the form given by (\ref{upexpansiona}).

In view of $ b) $ of Lemma \ref{lemma garnir1} 
the basis step for this induction is given by $ \bT = \Bg $ a Garnir tableau.
Let us do it.
By relation (\ref{eq4}) we have that 
\begin{equation}\label{indstep}
 m_{\Bs \Bg} =   \psi_{ d(\Bs) }^{\ast} e( \bi^{ {max}}) \psi_{ d(\Bg) }  = 
\psi_{ d(\Bs) }^{\ast} \psi_{ d(\Bg) } e( \bi^{\Bg})  
\end{equation}
and so for the basis step to work it is enough to prove that $  e( \bi^{\Bg})  = 0$. 
Let $ \gamma \in [\bmu^{max}]$ be the node associated with $ \Bg $ as in Lemma \ref{description garnir}. 
Using Lemma \ref{garnireq}
we may assume that
\begin{equation}
  e( \bi^{\Bg})  \sim e( \bi^{\tilde{\Bg}_{\gamma}}).
\end{equation}
Let $ j = \tilde{\Bg}_{\gamma}^{-1}(\gamma) $. 
Applying Corollary \ref{cor-base-de-induccion}
to the restriction of $ \tilde{\Bg}_{ \gamma} $ to the numbers $ \{1,2, \ldots, j-1 \} $ and $ \iota = \rm{res}(\gamma)$ we now get that 
$ e( \bi^{\tilde{\Bg}_{\gamma}}) =0 $, and so
also $ e( \bi^{\Bg}) =0 $
which proves the claim in this case.

Let us now consider the case of a general non-standard $ \bmu_n^{max}$-tableau $ \bT$.
Using Corollary \ref{fundamental-garnir1} there exists a Garnir tableau $\Bg$ and a $w\in \Si_n$ such
that $\bT=\Bg w$ and $l(d(\bT))=l(d(\Bg))+l(w)$. 
Hence there exists a reduced expression for $ d(\bT) $ of the form $ d(\bT) = s_{i_i} \cdots s_{i_N} s_{j_i} \cdots s_{j_M} $ where 
$ d(\Bg ) =  s_{i_i} \cdots s_{i_N}  $ and $ w = s_{j_i} \cdots s_{j_M} $. If this reduced expression is the official one
for $ d(\bT) $ 
we have that 
\begin{equation}
  m_{\Bs \bT} =   \psi_{ d(\Bs) }^{\ast}   e( \bi^{ {max}}) \psi_{d( \Bg) }  \psi_{ w }  = 0
\end{equation}
by the inductive basis, proved above.  
If it is not the official expression for $ d(\bT) $ we have 
by (\ref{official}) that the error term that occurs when changing to the
official expression is given by a linear combination of terms of the form
$    y^{\underline{k}} \psi_v $ where $ \underline{k} \in \No^n $ and $ v > d(\bT) $. Now for any non-trivial
factor $ y^{\underline{k}} $ we have that $  e( \bi^{ {max}}) y^{\underline{k}}  $ is zero by Lemma
\ref{base-de-induccion} and for the terms $ \psi_v $ we have by
Theorem \ref{Teo-Ehr} 
that $ v = d(\bT_1) $ with $ \bT_1 \rhd \bT$, and so we may use the inductive hypothesis on the
non-standard $ \bT_1$'s that may occur.

\medskip
Let us now consider a general multipartition $ \blambda \neq \bmu_n^{max} $.
We consider $ m_{\Bs  \bT} $ for  $\Bs \in \tab(\blambda),  \bT \in \nstd(\blambda) $ and
once again use downwards induction on $ \bT $ with respect to $ \lhd $ to show that 
$  m_{\Bs \bT}  $, {\color{black}{for $ \bT \in  \nstd(\bmu^{max}) $}},
can be written in the form given by (\ref{upexpansion}).
For $ \bT  $ maximal in $  \nstd(\blambda) $ we have that 
$ \bT = \Bg $ is a Garnir tableau for $ \blambda$ 
and so, arguing the same way as we did for
(\ref{indstep}), we get
\begin{equation}\label{indstep1}
 m_{\Bs \Bg} =   \psi_{ d(\Bs) }^{\ast} e( \bi^{ {max}}) \psi_{ d(\Bg) }  = 
\psi_{ d(\Bs) }^{\ast} \psi_{ d(\Bg) } e( \bi^{\Bg}).
\end{equation}
Passing to $ \tilde{\Bg}_{\gamma}$ as we did get in the inductive basis case, and 
using 
({\ref{trian2}}) and 
({\ref{trian21}}) of Lemma \ref{uppertriangularI},
we then get 
\begin{equation}
  m_{\Bs \Bg} = \sum_{\bmu > \blambda} D_{\bmu} =
   \sum_{\Bs, \bT \in \tab(\bmu), \, \bmu > \blambda}   c_{\Bs \bT} m_{\Bs \bT} 
  \end{equation}
where we used Lemma \ref{westart} for the second equality.
We then use the inductive hypothesis on each appearing $ m_{\Bs \bT}  $,
to rewrite in terms of $ m_{\Bs_1 \bT_1}  $ for $ \Bs_1 $ and $ \bT_1 $ standard tableaux. This concludes 
the case $ \bT = \Bg$.

Finally, for the general non-standard $ \blambda$-tableau $ \bT$ we have that
\begin{equation}
  m_{\Bs \bT} =   \psi_{ d(\Bs) }^{\ast}   e( \bi^{ \blambda}) \psi_{d( \Bg) }  \psi_{ w } =
   \sum_{\bT_1 \in \std(\blambda),  \bT_1  \rhd \bT}   c_{\Bs \bT} m_{\Bs \bT_1} + \sum_{\bmu > \blambda} D_{\bmu} 
\end{equation}
where the second equality arises from the error terms $ \psi_{ d(\Bs) }^{\ast} e(\bi^{\blambda}) y^{\underline{k}} \psi_v$.
But as before we can apply the induction hypothesis on each $ D_{\bmu} $ rewriting 
it in terms of $ m_{\Bs_1 \bT_1}  $ where $ \Bs_1 $ and $ \bT_1 $ are standard tableaux. This concludes 
the general $ \bT $-case. Finally the $ \Bs $-case follows from the $ \bT$-case by applying
$ \ast$ and so the Lemma is proved.
\end{proof}

From the Lemma we deduce the following Corollary. It is the main result of this section.

\begin{corollary}
The subset $ \Basis $ of $ \B $ given by 
\begin{equation}\label{stdgen}
\Basis := \{ m_{\Bs \bT} \mid \blambda \in \OnePar, \, \Bs, \bT \in \std(\blambda) \}
\end{equation}
spans $ \B $ over $ \F$.  
\end{corollary}
\begin{proof}
 This is a consequence of Lemma \ref{westart} and Lemma \ref{firstMain}.
\end{proof}
  
\section{Linear Independence of $ \Basis$. }
In this section we show that the set $ \Basis $ constructed in (\ref{basisdef}) is a linearly independent set.
Our methods used so far, essentially being manipulations with
the defining relations for $ \B$, are not sufficient for proving this and in fact it cannot even be proved that 
$ m_{\Bs \bT} $ is non-zero with these methods.

\medskip
To show the linear independence of $ \Basis $ we shall rely on the seminal work by Brundan-Kleshchev and
Rouquier, see \cite{brundan-klesc}, \cite{Rouq}
that establishes an isomorphism between the cyclotomic KLR-algebra $ \R$ and
the cyclotomic Hecke algebra $ \HH$.

\medskip
Let us give the precise definition of the relevant cyclotomic Hecke algebra.
\begin{definition}\label{hecke algebra}
  Let $\F$, $e$ and $ \hat{\kappa} \in {\mathbb Z}^l $ be as above, and let $ q \in  \F \setminus \{ 1 \}  $
be an $e \!$'th primitive root of unity. 
The \emph{cyclotomic Hecke algebra} $ \HH(q, \kappa)$ is 
the $\F$-algebra with generators $L _1 ,\ldots,L_n ,$ $T_1 ,\ldots ,T _{n-1}$ and relations
\begin{equation}
(L _1 - q^{\kappa_1}) \cdots (L _1 -q^{\kappa_l}) = 0
\end{equation}
\begin{equation}\label{quadratic}
  (T_r  + 1) (T_r - q)  = 0
\end{equation}
\begin{equation}\label{braid}  
  T _s  T_{s+1} T_s  = T _{s+1} T_s T_{s+1}
\end{equation}
\begin{equation}\label{jm}
L_rL_s=L_sL_r, \,  T_r L_r   = L_{r+1} (T_r - q + 1)
\end{equation}

\begin{equation}
T_r L_s  = L_s T_r    \mbox{ if }  |r - s| > 1 \mbox{ and } 
T_r T_s     = T_s T_r    \mbox{ if }  s \neq r,r + 1
\end{equation}
for all admissible $r, s$.
\end{definition}
It follows from the relations that there is antiinvolution $ \ast$ of $ \HH$, fixing the generators
$ T_i $ and $ L_i$.
We have that $T_r$ is invertible with $T_r^{-1} = q^{-1}(T_r -q +1 )$.
From this one gets that
\begin{equation} \label{jucysmurphy}
L_{r+1} = q^{-1} T_r L_r T_r
\end{equation}
and so $ L_2, \ldots, L_n $ are actually redundant for generating $ \HH$.
The elements $ L_i $ are called \emph{Jucys-Murphy elements} for $ \HH$.

\medskip
Let $ \q $ be a variable and 
let $ \K $ be the quotient field of the polynomial ring $ \F[\q]$. 
Let $ \OO $ be the subring of $ \K$ given by $ \OO := \{ \frac{ f(\q)}{g(\q)} \mid f(\q), g(\q) \in \F[\q],   g(q) \neq 0 \} $.
Then $ \OO $ is a local ring with maximal ideal $ \m := (\q-q) = \{ \frac{ f(\q)}{g(\q)} \in \OO \mid    f(q) = 0 \} $. The evaluation map $ \OO \rightarrow \F, \frac{ f(\q)}{g(\q)} \mapsto \frac{ f(q)}{g(q)} $ induces
an isomorphism $ \OO/\m \cong \F$ and so the triple $ (\OO, \F, \K) $ is a modular system.

\medskip
Let $\HHO= \HHO(\q, \kappa) $ be the $ \OO$-algebra given by the same presentation used
for $ \HH$, but replacing $ q $ by $ \q \in \OO$, and let similarly
$\HHK= \HHK(\q, \kappa) $ be the $ \K$-algebra given by the same presentation used
for $ \HH$, but replacing $ q $ by $ \q \in \K$. It is known that $ \HHO$ is free over $ \OO $ of rank $ l^n n!$.
Furthermore, we have that $ \HHO \otimes_{\OO} \F \cong \HH $ where $ \F $ is made
into an $ \OO$-algebra via evaluation in $ q$, and that
$ \HHO \otimes_{\OO} \K \cong \HHK $, via extension of scalars. It follows that
$ \HH $ and $ \HHK$ both have dimension $ l^n n!$.

\medskip
The representation theory of $ \HH $ is governed by $ \MP$, that is $l$-multipartitions of $ n$.
Let $ \blambda $ be an element of $ \MP $ and let $ \Bs \in \tab(\blambda) $.
Then we define
the content function of $ \Bs $ via the formula
\begin{equation}\label{contentF}  c_{\Bs}(i) = q^{\textrm{res}(\Bs(i))}  \in \F
\end{equation}  
where $ \textrm{res} $ is
as in (\ref{thenwedefinetheresidue}). 
Note that since $ q $ is an $e$'th primitive root of unity, this makes sense.
The content function for $ \HHO $ and $ \HHK$ is defined via 
\begin{equation}\label{content}
    c_{\Bs}^{\OO}(i) = c_{\Bs}^{\K}(i) = \q^{\hat{\kappa}_k +c-r} \in \OO \subseteq \K
\end{equation}
where $ \Bs(i) = (r,c,k) $. 
By the condition $ i) $ on the multicharge $ \hat{\kappa} $, the content function
satisfies the separability condition given in \cite{Mat-So} and so $ \HHK$ is a
semisimple algebra.

\medskip
The following concepts and results have their origin in Murphy's papers. 
Let $ \std(n):= \cup_{\blambda \in \Par} \std(\blambda)$.
For $ \Bs $ any element of $ \std(n) $ we define 
\begin{equation}\label{idempotentK}
  F_{\Bs} := \prod_{k=1}^n \prod_{\substack{ \bT \in \std(n) \\ c_{\Bs}^{\K}(k) \neq c_{\bT}^{\K}(k) }} \frac{ L_k - c_{\bT}^{\K}(k)}{ c_{\Bs}^{\K}(k) - c_{\bT}^{\K}(k)} \in \HHK.
\end{equation}
It is known that the $ F_{\Bs} $'s form a complete system of orthogonal idempotents.
The $ F_{\Bs}$'s are simultaneous eigenvectors for the action of the $ L_i$'s and the corresponding eigenvalues are
given by the contents: 
\begin{equation}\label{by the contents}
L_i   F_{\Bs} =   F_{\Bs} L_i  = c_{\Bs}^{\K}(i) F_{\Bs}.
\end{equation}

Unfortunately, a construction in $ \HH$ similar to (\ref{idempotentK}) does not lead to
idempotents in $ \HH$. Note also that 
$F_{\Bs} \notin \HHO$ because of the denominators.
In order to get idempotents in $ \HHO $ and $ \HH$, we consider the sum over the $ F_{\Bs}$'s for $ \Bs $ belonging
to a class of a certain equivalence relation on tableaux, that we now explain. Let $ \Bs$ and $ \bT $ be
tableaux
for multipartitions $ \blambda $ and $ \bmu$. Then we set
$ \Bs \sim_e \bT $ if $ \res( \Bs(i) ) = \res( \bT(i) ) \,\, {\rm mod}\, e $ for all $ i $, or equivalently 
$ c_{\Bs}(i) = c_{\bT}(i)  $ for all $i$. This indeed defines
an equivalence class on the set of all tableaux. We denote by $ [\Bs] = [\Bs]_e $ the class under
$ \sim_e $  represented by
$ \Bs  $ and set 
\begin{equation}\label{andletusnowdefine}
E_{[ \Bs] } := \sum_{ \bT \in  [\Bs] \cap \std(n)} F_{\bT}.
\end{equation}  
Then Mathas has proved in \cite{MatCoef}, building on Murphy's ideas in the symmetric group case, that $ E_{[ \Bs] } $ belongs to $ \HHO $ and hence $ E_{[ \Bs] } \otimes_{\OO} 1 $ belongs to $\HH$.
We shall write $ E_{[ \Bs] } $ for $ E_{[ \Bs] } \otimes_{\OO} 1 $ as well. Clearly the $ E_{[ \Bs] } $'s are
orthogonal idempotents in both $ \HH $ and $ \HHO$.

Any equivalence class $ [\Bs ] $ gives rise to a residue sequence
$ \bi^{\Bs} := (i_1, i_2, \ldots, i_n  ) \in \II^n$ via $ i_j := c_{\Bs}(j) $. By construction, $ \bi^{\Bs} $ is independent of
the choice of representative of $ [\Bs] $.

\medskip

The Brundan-Kleshchev and Rouquier isomorphism Theorem establishes an isomorphism of $ \F$-algebras
$ f: \R \cong  \HH$. We need to explain the images of the generators under $ f$.

In the case of $ f(e(\bi)) $, Brundan and Kleshchev describe it as the idempotent for the 
generalized eigenspace for the joint action of the $ L_i$'s, that is 
\begin{equation}
 f(e(\bi)) \HH= \{ h \in \HH \mid (L_k - i_k)^m h = 0 \mbox{ for some } m > 1 \}.
\end{equation}  
There is however a more concrete description of $ f(e(\bi)) $ due to Hu-Mathas, see \cite{hu-mathas}. It is of importance
to us because it allows us to lift $ f(e(\bi)) $ to $ \HHK$, via {\color{black}{(\ref{andletusnowdefine})}}. It is given by the formula
\begin{equation}\label{givenbythe}
f(e(\bi) ) = \left\{ \begin{array}{ll} E_{[\Bs]} & \mbox{if } \bi = \bi^{\Bs} \mbox{ for some } \Bs \in \std(n)\\ 0 & \mbox{otherwise. }  \end{array}   \right.
\end{equation}
In order to describe $ f(y_i)$ and $ f( \psi_i) $ it is 
enough to describe $ f(y_i)E_{[\Bs]} $ and $ f( \psi_i) E_{[\Bs]}$, since we have that 
$ \sum_{[\Bs]}  E_{[\Bs]} = 1$.
In \cite{brundan-klesc} 
$ f(y_i)$ is described as the 'nilpotent part of the Jucys-Murphy element $ L_i$', or more precisely
\begin{equation}\label{firstlift}
  f(y_i  ) E_{[\Bs]} = \left(1- \frac{1}{c_{\Bs}(i)} L_i \right) E_{[\Bs]}.
\end{equation}
We have a lift of this to $ \HHK $ as well. Supposing that $c_{\Bs}(i) = q^{\kappa_m +c-r} \in  \F $ we let 
$ \widehat{c_{\Bs}}(i) := \q^{\hat{\kappa}_m +\hat{c}-\hat{r}}$ 
where $ \hat{c}-\hat{r} \in  \mathbb Z $
is any preimage of $ c-r \, \, \rm{ mod } \, e $. Then our lift of (\ref{firstlift}) is
\begin{equation}\label{thirdlift}
  \left(1- \frac{1}{c_{\Bs}(i)} L_i \right)  \sum_{ \bT \in  [\Bs]} F_{\bT} =
\sum_{ \bT \in  [\Bs]} \left(1- \frac{c^{\K}_{\bT}(i)}{\widehat{c_{\Bs}}(i)} \right) F_{\bT}   \in \HHK.
 \end{equation}

The $ y_i$'s are nilpotent elements of $ \R $. Using this, 
Brundan and Kleshchev define in \cite{brundan-klesc} formal 
power series $ P_i(\bi), Q_i(\bi) $ in $ \F[[y_i, y_{i+1}]] $. They give the formula
\begin{equation}\label{secondlift}
\psi_i e(\bi) =
(T_i + P_r(\bi))Q_i (\bi)^{-1} e(\bi)
\end{equation}
which defines $ f(\psi_i)$
since we already know $ f(y_i) $ and $ f( e(\bi)) $.

\medskip
To make use of these formulas we shall rely on 
$ \{ f_{\Bs \bT} \mid \Bs, \bT \in \std(\blambda), \blambda \in \Par \} $, the \emph{seminormal basis}
for $ \HHK$, constructed by Mathas in \cite{MatCoef}. We have that
\begin{equation}
  F_{\Bs} f_{\Bs_1 \bT_1} F_{\bT} = \delta_{ \Bs, \Bs_1} \delta_{ \bT, \bT_1} f_{\Bs \bT}
\end{equation}
where 
$ \delta_{ \Bs, \Bs_1} $ and $  \delta_{ \bT, \bT_1} $ are Kronecker delta functions, and
so $ \{ f_{\Bs \bT} \} $ is a $ \K$-basis for $ \HHK $ consisting of eigenvectors for the action of
the $ L_i$'s.

\medskip
We need the following analog of the classical formulas for the action of $ s_i $ on the seminormal basis of
the group algebra of the symmetric group. In this particular case, they are due to Mathas, see Proposition 2.7 of
\cite{MatCoef}.

\begin{proposition}\label{YSF}
  Let $ \Bs $ and $ \Bu $ be standard $ \blambda $-tableaux and let $ \bT = \Bs s_i $.
  If $ \bT $ is standard then

\begin{equation}\label{78}
    f_{ \Bu \Bs} T_i = \left\{ \begin{array}{ll} \frac{ (q-1) c^{\K}_{ \bT}(i) } { c^{\K}_{ \bT}(i) - c^{\K}_{ \Bs}(i)}     f_{ \Bu \Bs} +     f_{ \Bu \bT} & \mbox{ if } \Bs \rhd_{\infty} \bT    \\ & \\
      \frac{ (q-1) c^{\K}_{ \bT}(i) } { c^{\K}_{ \bT}(i) - c^{\K}_{ \Bs}(i)}     f_{ \Bu \Bs} +
      \frac{( q  c^{\K}_{ \Bs}(i) - c^{\K}_{ \bT}(i)) (   c^{\K}_{ \Bs}(i) - q c^{\K}_{ \bT}(i))   } { (c^{\K}_{ \bT}(i) - c^{\K}_{ \Bs}(i))^2}
      f_{ \Bu \bT} & \mbox{ if } \Bs \lhd_{\infty} \bT
      \end{array} \right. 
\end{equation}
whereas if $ \bT $ is non-standard then  
  \begin{equation}\label{79}
    f_{ \Bu \Bs}     T_i = \left\{ \begin{array}{ll}   q f_{ \Bu \Bs} & \mbox{ if } i  \mbox{ and } i+1 \mbox{ are in the same
        row of  } \Bs \\ & \\
      - f_{ \Bu \Bs} & \mbox{ if } i  \mbox{ and } i+1 \mbox{ are in the same column of  } \Bs {\color{black}{.}}
    \end{array} \right.
  \end{equation}   
There are versions of (\ref{78}) and (\ref{79}), with $ T_i $ multiplying on the left.
\end{proposition}  

Actually there are some minor sign errors at this point in \cite{MatCoef}.
In fact, our formulas (\ref{78}) are completely identical with the formulas used by Mathas in \cite{MatCoef}, but
only our formulas are correct since Mathas' quadratic relations take the form
$ (T_r -1) (T_r+q ) = 0 $ whereas ours are $ (T_r +1) (T_r-q ) = 0 $, see (\ref{quadratic}).

Note that the formulas of the
Proposition depend on the order $ \unlhd_{\infty} $, although we believe that it is possible
to obtain similar formulas depending on $ \unlhd_0$. Note also that it follows from the formulas
that $ \spa_{\K} \{ f_{ \Bs \bT} \mid \shape(\Bs) = \blambda_0 \} $ is a two-sided ideal of
$ \HHK$ where $ \blambda_0 $ is any fixed multipartition.
Finally, note that all coefficients appearing in the formulas are nonzero.
In the case of the second coefficient of (\ref{78}), this is a consequence of the
   condition $i) $ on the multicharge  $ \hat{\kappa} $.

We have the following formula relating the seminormal basis to the $ F_{\bT}$'s 
\begin{equation}\label{knownconstanst}
F_{\bT} = \frac{1}{\gamma_{\bT}} f_{\bT \bT} 
\end{equation}
where $ \bT $ is any standard tableau of a multipartition $  \blambda $ and where $ {\gamma_{\bT}} \in  \K^{\times}$ is
a known constant.

\medskip
We need the following Lemma.
\begin{lemma}\label{tableauxclasses}
  Let $ \blambda \in \OnePar $ be a one-column multipartition and let $ \bT^{\blambda} $ be the maximal
  $ \blambda$-tableau, as above. Suppose that 
  $ \Bs \in [ \bT^{\blambda} ] {\color{black}{ \setminus \{\bT^{\blambda}}} \}$ and that $\shape(\Bs) \in \OnePar$. Then $ \Bs >  \bT^{\blambda}   $.
\end{lemma}
\begin{proof}
  Let $ \Bs \in [ \bT^{\blambda} ] \setminus \{\bT^{\blambda} \}  $ and
  let $ i \in \bn $ be minimal such that $ \Bs(i) \neq \bT^{\blambda}(i) $. 
  The nodes $ \Bs(i) $ and $ \bT^{\blambda}(i) $ have the same residues since
  $ \Bs \sim_e \bT^{\blambda} $ and so strong adjacency-freeness of $ \hat{\kappa}$, together with the fact that
  $ \Bs $ is standard, implies that $ i $ is
  situated higher in $ \Bs $ than in $ \bT^{\blambda} $, that is $ \Bs(i ) \rhd \bT^{\blambda}(i)$.
  But then we have either $ \Bs >  \bT^{\blambda}   $ or $ \shape(\Bs) \notin  \OnePar$ which proves the Lemma.
\end{proof}

With these preparations, we can now prove the linear independence of our proposed basis.

\begin{theorem}
  The set $ \Basis = \{ m_{\Bs \bT} \mid \blambda \in \OnePar, \, \Bs, \bT \in \std(\blambda) \} $
  introduced in (\ref{stdgen}) is linearly independent over $ \F $ and hence it is a basis for $ \B$.
\end{theorem}
\begin{proof}
  Let us assume that there is a non-trivial linear dependence between the elements of $ \Basis$
  \begin{equation}\label{ldp}  
\sum_{ \Bs, \bT } \lambda_{ \Bs \bT } m_{\Bs \bT} = 0. 
  \end{equation}
  Letting $ \pi: \R \rightarrow \B$ be the projection map from the KLR-algebra to the blob-algebra
  and {\color{black}{taking inverse images}} on both sides of (\ref{ldp}) we then get 
  \begin{equation}
\sum_{ \Bs, \bT } \lambda_{ \Bs \bT } m_{\Bs \bT} + p = 0
  \end{equation}  
{\color{black}{for some}} $ p \in \ker \pi $ and so 
  \begin{equation}\label{takesplace}
\sum_{ \Bs, \bT } \lambda_{ \Bs \bT } f(m_{\Bs \bT}) + f(p) = 0.
  \end{equation}   
We now note that any $ f(m_{\Bs \bT}) =  f(\psi_{ d(\Bs)}^{\ast}e( \bi^{ \blambda})  \psi_{ d(\bT)}) $ can be written 
as a linear combination of terms of the form $   T_v^{\ast}  g_v(y) E_{ [\bT^{ \blambda}]} f_w (y) T_w  $
where $g_v(y) , f_w (y) \in 
\F[y_1, \ldots, y_n] $ for some $ v,w \in \Si_n $ with $ v \ge  d( \Bs) $ and $ w \ge  d( \bT) $ and where $g_{d( \Bs)}(y) $ and $  f_{d( \bT)}(y) $ are invertible, that is
of nonzero constant terms.
That this is possible follows from (\ref{givenbythe}) and an observation due to Hu and Mathas, 
see the proof of Lemma 5.4 of {\color{black}{\cite{hu-mathas}.}}
Combining this expansion with Lemma \ref{uppertriangularI} we get
that
\begin{equation}\label{analyse}
f(m_{\Bs \bT}) =  T_{d(\Bs)}^{\ast}   E_{ [\bT^{ \blambda}]}  T_{d(\bT)} +
\sum_{ v > d(\Bs), w > d(\bT) } \mu_{ v,w } \,  T_{v}^{\ast}   E_{ [\bT^{ \blambda}]}  T_{w} 
+ \sum_{ \bmu > \blambda } f(D_{\bmu}) + f(p_1)
\end{equation}
where $ D_{\bmu} {\color{black}{\in \langle e( \bi^{\bmu}) \rangle}}$,  $ \mu_{ v,w } \in \F$ and $ p_1 \in \ker \pi$.
This expression for $ f(m_{\Bs \bT}) $ takes place in $ \HH $, but can be lifted to $ \HHO $ 
via (\ref{givenbythe})
and then embedded in $ \HHK $. Let us now analyse the various ingredients of (\ref{analyse}), starting with $ f(p_1)$.
We have that
\begin{equation}
  \ker \pi = \langle  e(\bi)   \mid  i_1 \in \{ \kappa_1, \ldots, \kappa_l\},  i_2 = i_1 +1 \,{\rm  mod }\,\, e \rangle \subseteq \R
\end{equation}  
corresponding to the omission of relation (\ref{eq13}).
Using {\color{black}{(\ref{andletusnowdefine})}} and (\ref{givenbythe}) we then get that
\begin{equation}\label{wethengetthat}  f(p_1) = \sum_{ \Bs \in \std(n)} \sum_{ \bT \in  [\Bs]}  a^{\Bs}_{\bT,1}   F_{\bT} \,
  a^{\Bs}_{\bT,2}  
\end{equation}
where $  a_{\bT,1}^{{\color{black}{\Bs}}}, a_{\bT,2}^{{\color{black}{\Bs}}} \in \HHK $ and where $ \Bs \in \std(n) $ 
satisfies $ {\rm res}(\Bs(1)) \in \{ \kappa_1, \ldots, \kappa_l\} $ and $  {\rm res}(\Bs(2)) = {\rm res}(\Bs(1)) +1 \,{\rm  mod }\,\, e$. These conditions, together with the conditions on $\hat{\kappa} $, imply that for each $ \bT \in [\Bs] $
we have $ \shape(\bT) \notin \OnePar$.
Combining this with Proposition \ref{YSF} and (\ref{knownconstanst})
we get that
\begin{equation}\label{combining}  f(p_1) \in \spa_{\K} \{ f_{\Bs \bT} \mid \Bs, \bT \in \std(\blambda), \blambda \notin \OnePar \}.
\end{equation}
Let us now consider the terms $ f(D_{\bmu}) $ of (\ref{analyse}). We have that 
\begin{equation}
f(D_{\bmu}) = \sum_{ \bT \in  [\bT^{\bmu}]}  a_{\bT,1}   F_{\bT} \,a_{\bT,2}  
\end{equation}
where $  a_{\bT,1}, a_{\bT,2} \in \HHK $. For each appearing $ \bT $ we have $ \bT > \bT^{\bmu}  $ by 
Lemma \ref{tableauxclasses}. Combining this with $ \bmu > \blambda$, that is  $ \bT^{\bmu} > \bT^{\blambda}$, 
we get that $  \bT > \bT^{\blambda}  $ and so there is a $ k $ such that $ \bT \! \mid_k = \bT^{\blambda} \! \!  \mid_k $
and $ \bT(k+1) \rhd  \bT^{\blambda}(k+1) $. But then $ \bT(k+1) \notin [ \blambda] $, which implies that
$ \shape(\bT) > \blambda$. Hence we have that 
\begin{equation}\label{combining10}  f(D_{\bmu}) \in \spa_{\K} \{ f_{\Bs \bT} \mid \Bs, \bT \in \std(\bnu), \bnu > \blambda \}.
\end{equation}
Similarly, for all tableaux $ \bT $ in $[ \bT^{\blambda} ] $ we have that $ \shape(\bT) >\blambda $. 
Hence from (\ref{analyse}), (\ref{combining}) and (\ref{combining10}) 
we get that 
\begin{equation}\label{analyse2}
f(m_{\Bs \bT}) \in   T_{d(\Bs)}^{\ast}   F_{ { \blambda}}  T_{d(\bT)} +
\sum_{ v > d(\Bs), w > d(\bT) } \mu_{ v,w } \,  T_{v}^{\ast}    F_{ { \blambda}} T_{w}  +
\spa_{\K} \{ f_{\Bs \bT} \mid \Bs, \bT \in \std(\bnu), \bnu > \blambda \mbox{ or }
\bnu \notin \OnePar \}
\end{equation}
where $ \blambda $ as a subscript refers to $ \bT^{\blambda}$.

Let us now focus on $ T_{d(\Bs)}^{\ast}   F_{ { \blambda}}  T_{d(\bT)} $. 
Let $ d(\bT) = s_{i_1}  s_{i_2}  \cdots s_{i_N}  $ be a reduced expression for $ d(\bT)$.
  When calculating $ f_{\blambda} T_{ d(\bT)} $ using this expression and Proposition \ref{YSF},
  we obtain an expression for $ f_{\blambda} T_{ d(\bT)} $ as a
  $ \K$-linear combination of certain $  f_{\blambda \Bu }  $'s.
  But by the formulas of the Proposition, for each appearing $ \Bu $ we have that $ d(\Bu ) $ is a subexpression of
  $ s_{i_1}  s_{i_2}  \cdots s_{i_N}  $ and so by our version of the Ehresmann Theorem, that is Theorem
\ref{Teo-Ehr}, we have that $\bT \unlhd \Bu $ for each 
  occurring $  f_{\blambda \Bu }  $. 
  Letting $ \bT_k := \bT^{\blambda} s_{i_1} \ldots  s_{i_k} $ we have 
  $ \bT_{k+1} \lhd \bT_{k} $ for all $ k=1, \ldots, N-1$ and so 
in the above expansion of $ f_{\blambda} T_{ d(\bT)} $ the term 
$ f_{\blambda \bT }  $ corresponds exactly to the subexpression of $ s_{i_1}  s_{i_2}  \cdots s_{i_N}  $
where no $ s_i$ is omitted. By the remarks following the Proposition, the
corresponding coefficient $ \alpha_{\bT} $ is nonzero and 
so we have  
\begin{equation}\label{simi}
  f_{\blambda} T_{ d(\bT)}  =\alpha_{\bT} f_{ \bT^{\blambda} \bT}   + \sum_{ \Bu \rhd \bT} \alpha_{\Bu} f_{ \bT^{\blambda} \Bu}
\end{equation}
where $ \alpha_{\Bs},  \alpha_{\Bu} \in \K$ and where $ \alpha_{\bT} \neq 0$. Acting on the left
with $ T_{ d(\Bs)}^{\ast} $, and arguing the same way as we did for (\ref{simi}), we obtain an expansion 
\begin{equation}\label{weobtainan}
  T_{ d(\Bs)}^{\ast}  f_{\blambda} T_{ d(\bT)}  =\alpha_{\Bs \bT} f_{  \Bs \bT }   + \sum_{\Bu, \Bv \rhd \bT}
  \alpha_{\Bu \Bv} f_{ \Bu \Bv} 
\end{equation}
where $ \alpha_{\Bv \Bu}, \alpha_{\Bs \bT} \in \K $ and where $  \alpha_{\Bs \bT} \neq 0$.
Let us now focus on the term $ T_{v}^{\ast}    F_{ { \blambda}} T_{w} $ of (\ref{analyse2}).
But arguing as was done for $ T_{d(\Bs)}^{\ast}   F_{ { \blambda}}  T_{d(\bT)} $, we can
write $ T_{v}^{\ast}    F_{ { \blambda}} T_{w}$ as a linear combination of 
$ f_{ \Bv \Bu} $'s. Moreover, since 
 $ v > d(\Bs) $ and $ w > d(\bT) $ we get for each appearing
$  \Bu $ and $  \Bv $ the relations $ \Bu \rhd \Bs $ and $ \Bv   \rhd \bT  $.

All together we can now write (\ref{analyse2}) in the form 
\begin{equation}\label{analyse3}
f(m_{\Bs \bT}) \in   \alpha_{ {\Bs \bT }}   f_{\Bs \bT } + 
\sum_{ \Bu \rhd \Bs, \Bv   \rhd \bT}  \, \alpha_{ {\Bu \Bv }}  f_{\Bu \Bv } +
\spa_{\K} \{ f_{\Bs \bT} \mid \Bs, \bT \in \std(\bnu), \bnu > \blambda \mbox{ or }
\bnu \notin \OnePar \}
\end{equation}
where $ \alpha_{ {\Bs \bT }}  \in \K^{\times} $ and $ \alpha_{ {\Bu \Bv }}  \in \K $.

Let us finally return to the linear dependency (\ref{takesplace}).
Let us extend the order $ \lhd$ to pairs $ \{ (\Bs, \bT ) \in  \std(\blambda)^2 \mid \blambda \in \OnePar \} $ via
$ (\Bs, \bT ) \lhd (\Bs_1, \bT_1 ) $ if $ \Bs \lhd \Bs_1 $ and $ \bT \lhd \bT_1 $ and let us choose 
$ (\Bs_0, \bT_0) $ minimal such
that $ \lambda_{\Bs_0 \bT_0} \neq 0$. Let $ \blambda_0 = \shape(\Bs_0) $.
Using (\ref{analyse3}) we can rewrite (\ref{takesplace}) in terms of the $ f_{ \Bs \bT}$'s.
In this expression, there are no cancellations for the 
coefficient of $ f_{ \Bs_0  \bT_0}$'s which is therefore $\lambda_{\Bs_0 \bT_0} \cdot \alpha_{\Bs_0 \bT_0}\neq 0  $.
But this is in contradiction with the fact that the $f_{ \Bs \bT}$'s form a basis for $ \HHK$ and so the Theorem is proved.
\end{proof}  

\section{Cellularity of $ \Basis$ and JM-elements}
In this section we obtain our main results, showing that $ \Basis $ is a cellular basis for $ \B$ with respect
to $ \lhd$, endowed with a family of JM-elements.
Let us first recall the definition of a cellular algebra, as given by Graham and Lehrer, see \cite{GL}.
Since we are interested in the $ \F$-algebra $ \B$, we here consider only the special case of algebras defined over a field.

\begin{definition}\label{cellular algebra} 
Let $ k$ be a field and
suppose that $\A$ is a $k$-algebra. Suppose that $(\Lambda, \leq)$ is a poset such that
for each $\lambda \in \Lambda $ there is a finite set $T(\lambda)$ and elements $c_{{st}}^{\lambda} \in \A$
such that $$ B =\{ c_{{st}}^{\lambda} \mbox{  } | \mbox{  } \lambda\in \Lambda , {s},{t} \in T(\lambda)  \}$$
is a $k$-basis of $\A$. Then the pair $(B, \Lambda)$ is called a cellular basis for $\A$ if
the following conditions hold:
\begin{description}
\item[(i)] The $k$-linear map $*:\A \rightarrow \A$ given by $(c_{{st}}^{\lambda})^{*}=c_{{ts}}^{\lambda}$
  is an algebra anti-automorphism of $\A$.
  \item[(ii)] If ${s},{t} \in T(\lambda) $ and $a\in \A$
    then there exist scalars $r_{{u}sa} \in k $ such that
\begin{equation}\label{multstruc}
    ac_{{st}}^{\lambda} \equiv
    \sum_{{u} \in T(\lambda)} r_{{u s a}}   c_{{ut}}^{\lambda}  \mod A^{\lambda}
 \end{equation}   
      where $\A^{\lambda}$ is the $\F$-subspace of $\A$ spanned by $\{ c_{{ab}}^{\mu} \mbox{  } | \mbox{  } \mu > \lambda , {a},{b} \in T(\mu)  \}$.
\end{description}
If $\A$ has a cellular basis then we say that $\A$ is a cellular algebra.
\end{definition}

\medskip
Note that the coefficients $ r_{{u}sa}$ of (\ref{multstruc}) may depend on ${u}, {s}$ and $a$, 
but the point is 
that the $ r_{{u} s a }$'s do not depend on ${t}$. Now suppose that $\A$ is a $\mathbb{Z}$-graded $k$-algebra  and that each $c_{{st}}^{\lambda}$ is homogeneous with respect to the grading.
Suppose that there exists a function $$   \deg: \coprod_{\lambda \in \Lambda} T(\lambda) \rightarrow \mathbb{Z}$$
such that $\deg c_{{st}}^{\lambda} = \deg {s} + \deg {t}$. Then we say that $ B =\{ c_{{st}}^{\lambda} \mbox{  } | \mbox{  } \lambda\in \Lambda, {s},{t} \in T(\lambda)  \}$ is a \emph{graded cellular basis} for $\A$. If $\A$ has a graded cellular basis then we say that $\A$ is a \emph{graded cellular algebra}.
The concept of a graded cellular algebra was formally introduced in \cite{hu-mathas}.

\medskip

In the previous sections we have proved that $ \Basis$ is a basis for $ \B$ and in fact one can even
deduce from the results of 
these sections that $ \Basis$ is a graded cellular basis for $ \B$, with respect to $ <$.
However, we aim at proving the stronger statement that $ \Basis$ is a graded cellular basis with respect to $ \lhd$.
The key combinatorial
ingredient that allows us to pass from $ < $ to $ \lhd $ is given by the following two Lemmas.

\begin{lemma}\label{tableauxclassesdominant}
  Let $ \blambda \in \OnePar $ be a one-column multipartition and let $ \bT^{\blambda} $ be the maximal
  $ \blambda$-tableau, as before. Suppose that 
  $ \bT \in [ \bT^{\blambda} ] \setminus \{\bT^{\blambda}\}$
  and that $\shape(\Bs) \in \OnePar$. Then $ \shape(\bT) \rhd  {\blambda}   $. 
\end{lemma}

\begin{proof}
  Set $ \bmu := \shape(\bT)$.
  By Lemma \ref{order-bijection} it is enough to find a bijection $\Theta:[\blambda]\rightarrow [\bmu]$
  such that $\Theta(\gamma)\unrhd\gamma$ for all $\gamma\in[\blambda].$
  Our candidate for this bijection is $\Theta:=\bT\circ(\bT^{\blambda})^{-1}$.
  Surely $\Theta$ is a bijection so let us check that $\Theta$ satisfies the order condition.
  Assume to the contrary that there is $ \gamma = \bT^{\blambda}(k) \in [\blambda] $ such
  that $\Theta(\gamma)\lhd \gamma$, or equivalently
  $\bT(k)\lhd \bT^{\blambda}(k)$, and let $ k_0 $ be the minimal such $ k $.
  Let $\bT^{\blambda}(k_0)=(r_0,1,j_0)$ and $\bT(k_0)=(r,1,j)$.
  By strong adjacency-freeness of $ \kappa$, and the fact that
  $\bT^{\blambda}(k_0)$ and $\bT(k_0) $ have the same residue, 
  we have that $r>r_0+1$, that is $ \bT(k_0) $ is located at least
  two rows below $\bT^{\blambda}(k_0)$. But by minimality of $ k_0$ we have that
  $ \bT(k) $ is located above $\bT^{\blambda}(k)$ for all $ k < k_0$. This is impossible since
  $ \bT $ is standard.
\end{proof}

For the next Lemma we need the conditions $iii) $ and $ iv) $ from Definition 
\ref{strongadj} of strong adjacency-freeness.

\begin{lemma}\label{tableauxclassesdominantGarnir}
  Let $ \blambda \in \OnePar $ be a one-column multipartition and let $ \Bg $ be 
Garnir tableau of shape $ \blambda $. 
Let $ \bT \in [ \Bg ] \setminus \{ \Bg \} $ and suppose that $ \shape(\bT) \in \OnePar$. Then
$ \shape(\bT) \rhd  {\blambda}$.

\end{lemma}
\begin{proof}
We shall follow the same approach as in the proof of the previous Lemma.  
  Set $ \bmu := \shape(\bT)$.
  {\color{black}{As in the previous Lemma 
  it is enough to find a bijection $\Theta:[\blambda]\rightarrow [\bmu]$
  such that $\Theta(\beta)\unrhd\beta$ for all $\beta\in[\blambda].$}}
  This time the candidate for the bijection is $\Theta:=\bT\circ(\Bg)^{-1}$.
  {\color{black}{This}}
  $\Theta$ is {\color{black}{also clearly}}
  a bijection so we must check that $\Theta$ satisfies the order condition.
  Assume to the contrary that there is $ {\color{black}{\beta}} = \Bg(k) \in [\blambda] $ such
  that $\Theta({\color{black}{\beta}})\lhd {\color{black}{\beta}}$, or equivalently
  $\bT(k)\lhd \Bg(k)$, and let $ k_0 $ be the minimal such $ k $.
  Let $\Bg(k_0)=(r_0,1,j_0)$ and $\bT(k_0)=(r,1,j)$.
  Using the previous Lemma, {\color{black}{and part $ (3) $ of the characterization of
      Garnir tableaux given in Lemma \ref{description garnir},
    we conclude that $ \Bg(k_0) \in  \Snake(\gamma) $, where $ \gamma $ is the special
    node for the Garnir tableau $ \Bg $, according to Lemma \ref{description garnir}.}}
    But then from
  strong adjacency-freeness of $ \hat{\kappa} $ we conclude that $ r= r_0 +2$, since
  there are no nodes of the same residue
  in consecutive rows of $ \blambda$, {\color{black}{that is $\bT(k_0)=(r,1,j)$ is situated two rows
      below $\Bg(k_0)=(r_0,1,j_0)$.}}
  {\color{black}{On the other hand, using condition $iv) $ of the Definition \ref{strongadj}
      of strong adjacency-freeness,
      we get that those nodes in the $ r$'th row of $[ {\rm res}(\bT^{\blambda})] $ that have the same
      residues as nodes in the $ r_0$'th row, are all shifted one to the right. In other words, we have that $ j=j_0+1$. 
      But this produces a gap between $\bT(k_0)$ and $ \Snake(\gamma) $ and so $\bT$ cannot be standard.
  }}The Lemma is proved.

\medskip
  
Let us illustrate this last point on the following
example with $ \blambda = ( (1^{11}), (1^{11}), (1^{11}),(1^{10}),
  (1),  (1^{2}) ) $, $ e = {\color{black}{13}} $ and  $ \gamma = ({\color{black}{9}},1,2) $:
\begin{equation}
{\color{black}{\Bg}}= 
\left( \gyoung(;1,;<7>,;\doce,;<16>,;<20>,;<24>,;<28>,;<32>,;<\treintatresR>,;<40>,;<\cuarentacuatro>),
\gyoung(;2,;8,;<13>,;<17>,;<21>,;<25>,;<29>,;<\color{red}35>,;<\color{red}34>,;<41>,;<45>), 
\gyoung(;3,;9,;<14>,;<18>,;<\veintidos>,;<26>,;<30>,;<\color{red}36>,;<38>,;<42>,;<46>),
\gyoung(;4,;\diez,;<15>,;<19>,;<23>,;<27>,;<31>,;<\color{red}37>,;<39>,;<43>,;<47>), 
\gyoung(;5,,,,,,,,,,),\gyoung(;6,;<\once>,,,,,,,,,)
\right),
{\color{black}{[ {\rm res}(\bT^{\blambda})]}}= 
\left( \gyoung(;0,;<\doce>,;\once,;\diez,;9,;8,;7,;6,;<\color{red}5>,;4,;3),
\gyoung(;2,;1,;0,;<\doce>,;\once,;\diez,;9,;<\color{red}8>,;<\color{red}7>,;<6>,;<5>), 
\gyoung(;4,;3,;2,;1,;0,;<\doce>,;\once,;\diezR,;9,;8,;7),
\gyoung(;6,;5,;4,;3,;2,;1,;0,;<\doceR>,;\once,;\diez,;9), 
\gyoung(;8,,,,,,,,,,),\gyoung(;\diez,;9,,,,,,,,,)
\right). 
\end{equation}  
The numbers appearing in $ \Snake(\gamma) $ of $ \Bg $ have been colored red.
{\color{black}{We are supposing that $ \bT \lhd \Bg $. Consider the case
    where $\Bg(k_0)=(8,1,2)$, that is $ k_0 = 35$.
Then for $ \bT $ to be standard we must have either $ \bT(35) = \Bg(40) $ or
$ \bT(35) = \Bg(41) $. But $ \bT(35)$ is of residue $ 8$ whereas neither $ \Bg(40) $ nor
$ \Bg(41) $ is of residue 8, and so we get the desired contradiction in this case. The other cases for
$\Bg(k_0) $ are
treated similarly.}}

{\color{black}{For completeness, we now give a}} tableau $ \bT $ in $ [\Bg] $.
{\color{black}{One checks easily that $ \bT \rhd \Bg $}}.
\begin{equation} \bT=
\left( \gyoung(;1,;<7>,;\doce,;<16>,;<20>,;<24>,;<28>,;<32>,;<\treintatresR>,;<40>,;<\cuarentacuatro>),
\gyoung(;2,;8,;<13>,;<17>,;<21>,;<25>,;<29>,,,,), 
\gyoung(;3,;9,;<14>,;<18>,;<\veintidos>,;<26>,;<30>,;<\color{red}36>,;<38>,;<42>,;<46>),
\gyoung(;4,;\diez,;<15>,;<19>,;<23>,;<27>,;<31>,;<\color{red}37>,;<39>,;<43>,;<47>), 
\gyoung(;5,;<\color{red}34>,;<41>,;<45>,,,,,,,),\gyoung(;6,;<\once>,;<\color{red}35>,,,,,,,,)
\right).
\end{equation}
\end{proof}  

We can now generalize the first statement of Lemma \ref{uppertriangularI}.

\begin{lemma}\label{uppertriangularIdom}
 For $ \blambda $ any one-column multipartition and any $ k $ we have that
\begin{equation}{\label{trian1dom}}
  y_k e(\bi^{\blambda}) = e(\bi^{\blambda}) y_k =  \sum_{ \Bs, \bT \in \std(\blambda), \bmu \rhd \blambda}
  c_{\Bs \bT}  m_{\Bs \bT}
\end{equation}
where the sum runs over one-column multipartitions $ \bmu $ of $n$ and $ c_{\Bs \bT} \in \F$.
\end{lemma}
\begin{proof}
We first note that by construction of the $ m_{\Bs \bT}$'s we have that 
\begin{equation}\label{byconstruction}
  e(\bi )  m_{\Bs \bT} = \left\{ \begin{array}{ll} m_{\Bs \bT} & \mbox{ if } \bi = \bi^{\Bs} \\ 0
    & \mbox{ otherwise.}  \end{array} \right.
\end{equation}
Let us now consider the expansion of
$ y_k e(\bi^{\blambda}) $ in the basis $ \Basis$:
\begin{equation} y_k e(\bi^{\blambda})
=  \sum_{ \Bs, \bT \in \std(\blambda), \blambda \in \OnePar}
  c_{\Bs \bT}  m_{\Bs \bT}
\end{equation}
where $ c_{\Bs \bT}  \in \F$.
We have that 
\begin{equation}
\sum_{ \Bs, \bT \in \std(\blambda), \blambda \in \OnePar}
  c_{\Bs \bT}   m_{\Bs \bT} = 
  y_k   e(\bi^{\blambda})  = 
  e(\bi^{\blambda}) y_k e(\bi^{\blambda}) 
=  \sum_{ \Bs, \bT \in \std(\blambda), \blambda \in \OnePar}
  c_{\Bs \bT}  e(\bi^{\blambda}) m_{\Bs \bT}
\end{equation}
and hence we get via (\ref{byconstruction}) that $ \bT \in [ \bi^{\blambda} ] $ 
whenever $ c_{\Bs \bT} \neq 0$ and so also $ \shape(\Bs) \rhd \blambda $, via Lemma \ref{uppertriangularIdom}.
The Lemma is proved.
\end{proof}

We can also generalize the third statement (\ref{trian21}) of Lemma \ref{uppertriangularI} in the relevant case of
a Garnir tableau $ \Bg$.
\begin{lemma}\label{thirdrelevant}
Let $ \Bg$ be a Garnir tableau for the multipartition $ \blambda$. Then we have an expansion of the form 
\begin{equation}
e(\bi^{\Bg}) = \sum_{\Bs, \bT \in \std(\bmu), \bmu \rhd \blambda } c_{\Bs  \bT} m_{\Bs \bT} 
\end{equation}
where $ c_{\Bs \bT} \in \F$.
\end{lemma}
\begin{proof}
From the Lemmas \ref{uppertriangularI} and \ref{westart} we have the expansion
\begin{equation}
e(\bi^{\Bg}) = \sum_{\Bs, \bT \in \std(\bmu), \bmu > \blambda } c_{\Bs  \bT} m_{\Bs \bT} 
\end{equation}
with unique coefficients $ c_{\Bs  \bT} \in \F $ since the $  m_{\Bs \bT}$'s are a basis.
Thus arguing as in the previous Lemma \ref{uppertriangularIdom} we get that $ {\Bs}\in [\Bg]   $ 
and so $ \shape(\Bs)  \rhd \blambda$ by Lemma \ref{tableauxclassesdominantGarnir}. 
\end{proof}

The following Lemma generalizes Lemma \ref{firstMain}, replacing $ < $ 
by $ \lhd$.
\begin{lemma}\label{firstMainI}
  Suppose that $ \blambda \in \OnePar $ and that $ \Bs, \bT \in \tab(\blambda) $.
  If $ \bT \in \nstd(\blambda) $ then there is an expansion
  \begin{equation}\label{upexpansion}
    m_{\Bs \bT} = \sum_{\bT_1 \in \std(\blambda),  \bT_1 \rhd \bT,  } c_{ \Bs \bT_1} m_{\Bs \bT_1}  +
    \sum_{ \bmu \rhd \blambda, \Bs_2, \bT_2 \in \std(\bmu)} c_{\Bs_2 \bT_2} m_{\Bs_2 \bT_2} 
  \end{equation}
where $c_{ \Bs \bT_1}, c_{\Bs_2 \bT_2} \in \F $. A similar statement holds for $ \Bs$.
\end{lemma}
\begin{proof}
  We go through the proof of Lemma \ref{firstMainI}, checking that 
  each occurrence of $ > $ can be replaced by $ \rhd$.
  There are two types of occurrences of $ > $. The first ones are in reference to 
  ({\ref{trian1}}) of Lemma \ref{uppertriangularI}. But here 
  Lemma \ref{uppertriangularIdom} allows us to replace $ > $ by $ \rhd$.
The second ones are the use of Garnir tableaux in (\ref{indstep}) and (\ref{indstep1}). 
But in view of Lemma \ref{thirdrelevant} we can also here replace $ > $ by  $ \rhd$.
\end{proof}

The following Lemma corresponds to the JM-property of the $ y_k$'s, that we shall consider in more detail
later on.
\begin{lemma}\label{Jucys-Murphy}
  Suppose that $  m_{\Bs \bT} $ is an element of $ \Basis$. Then we have that
  \begin{equation}\label{higherterms}
y_k m_{\Bs \bT}  = \sum_{ \Bs_1 \rhd \Bs} c_{\Bs_1 \bT} m_{\Bs_1 \bT} +  \mbox{  higher terms} 
  \end{equation}
  where $ c_{\Bs_1 \bT} \in \F $ and 
  where 'higher terms' means a linear combination of $ m_{\Bs_2 \bT_2} $ where $ \shape(\Bs_2) \rhd \shape(\Bs)$.
  A similar formula holds for $y_k $ acting on the right of $ m_{\Bs \bT}$.
\end{lemma}  
\begin{proof}
  We have that $ m_{\Bs \bT}^{\ast} = m_{\bT \Bs} $ and so
we get the formula for $  m_{\Bs \bT} y_k $ by applying $ \ast$ to the formula for $ y_k m_{\Bs \bT} $.
Suppose that $ d(\Bs) =  s_{i_1}   \cdots  s_{i_{N-1}}  s_{i_N} $ is the
official reduced expression for $ d(\Bs) $ so that we have
$ \psi_{d(\Bs)} = \psi_{i_1} \cdots \psi_{i_{N-1}}   \psi_{i_N}$.
We now have from
relations (\ref{eq6}), (\ref{eq7}), (\ref{eq8}) and (\ref{eq9}) that 
\begin{equation} y_k m_{\Bs \bT}   = 
  y_k \psi_{d(\Bs)}^{\ast} e( \bi^{\blambda}) \psi_{d(\bT)} = 
\left\{ \begin{array}{ll} 
  \psi_{i_N} y_k \psi_{i_{N-1}} \cdots \psi_{i_1} e( \bi^{\blambda}) \psi_{d(\bT)} & \mbox{ if } i \neq i_N,  i_N +1 \\
  \psi_{i_N}  y_{k \pm 1} \psi_{i_{N-1}} \cdots   \psi_{i_1} e( \bi^{\blambda}) +
\delta \psi_{i_{N-1}} \cdots \psi_{i_1} e( \bi^{\blambda})
  & \mbox{ if } i =  i_N,  i_N +1 \\  
\end{array} \right.
  \end{equation}  
where $ \delta = 0, \pm 1$. Using relations (\ref{eq6}), (\ref{eq7}), (\ref{eq8}) and (\ref{eq9})  once again,
we continue commuting the appearing $ y_{k\pm 1}$'s to
the right as far as possible, until they meet $ e( \bi^{\blambda}) $.
This gives rise to a linear combination of terms of the form
\begin{equation}
\pm   \psi_{j_K}  \psi_{j_{K-1}} \cdots \psi_{j_1} e( \bi^{\blambda}) \psi_{d(\bT)}
\end{equation}
where $ s_{j_1}     \cdots  s_{j_{K-1}} s_{i_K} $ is a strict subexpression of
$ s_{i_1}  \cdots s_{i_{N-1}} s_{i_N} $, together with 
$    \psi_{d(\Bs)}^{\ast} y_j e( \bi^{\blambda}) \psi_{d(\bT)} $ for some $ j$, 
corresponding to $ y_k $ commuted all the way through $ \psi_{d(\Bs)}^{\ast} $,
But this last term belongs 
to the 'higher terms', by the previous Lemma \ref{uppertriangularIdom}.
The other terms that arise are linear combinations of $ m_{\Bs_1 \bT} $'s
where $  \Bs_1 \rhd \Bs $ by the proof of 
Theorem \ref{firstMain}. This proves the Lemma.
\end{proof}

We can now prove the promised cellularity of $ \Basis$. 

\begin{theorem}\label{Cellular}
  The pair $ (\Basis, \OnePar) $ is a graded cellular basis for $ \B$ with respect to $ \lhd $, in the sense of
  Definition \ref{cellular algebra}.
\end{theorem}  
\begin{proof}
  Condition ({\bf i}) of Definition \ref{cellular algebra} is easily verified so let us concentrate on
  the multiplication Condition ({\bf ii}). It is enough to check it for $ a $ any of the generators $ e(\bi) $,
  $ y_i $ and $ \psi_i $. Here the case $ a  =e(\bi) $ is easy and the case $ a = y_i$ is given by Lemma
  \ref{Jucys-Murphy}, so we are left with the case $ a = \psi_i$.
  We here consider right multiplication
on $ m_{\Bs \bT } $ with $ \psi_i$.  We first write $ \psi_{ d(\bT)} \psi_i $ as a linear combination of
the elements
$ { \mathcal S } = \{ e(\bi) \,y^{\, \underline{k}}\, \psi_w  \mid \bi \in \II^n, \underline{k} \in {\mathbb N}^n, w \in \Si_n \} $
from (\ref{KLRbasis}). Upon right multiplication we get that $ m_{\Bs \bT } \psi_i $ is a linear
combination of $  \psi_{d(\Bs)}^{\ast} e(\bi^{\blambda}) \psi_w $ modulo higher terms. 
For each appearing $ w $ we consider $ \bT_1 := \bT^{\blambda} w $ 
and get that $ \psi_{d(\Bs)}^{\ast} e(\bi^{\blambda}) \psi_w  = m_{\Bs \bT_1} $.
If $ \bT_1 $ is standard we have that $ m_{\Bs \bT_1} \in \Basis $.
Otherwise, we use Lemma \ref{firstMainI} to rewrite 
$ m_{\Bs \bT_1} $ in terms of elements of $ \Basis$, modulo higher terms.
Hence Condition ({\bf ii}) has been verified and since $ \Basis $ consists of homogeneous elements
we are done.
\end{proof}

We remark that $ \B$ even satisfies the stronger property of being a \emph{quasi-hereditary algebra}. 
This follows from Remark 3.10 of \cite{GL}.

\medskip

The following definition appears for the first time in \cite{Mat-So}. It formalizes important properties of
Jucys-Murphy elements. These properties go back to Murphy's work on the symmetric group
and the Hecke algebra of finite type $ A_n $, see \cite{Murphy2}, \cite{Murphy} and \cite{Murphy3}.
\begin{definition}{\label{JM}}
Let $A$ be an $\F$-algebra which is cellular with respect to $\mathcal{C}=\{c_{\s\T}\mid
 \lambda\in \Lambda,\s,\T\in T(\lambda)\}$. Suppose also that each set $ T(\lambda) $ is
endowed with a poset structure with order relation $ \rhd_{\lambda} $.
Then we say that a commuting subset
 $\mathcal{L}=\{L_1,\ldots,L_M\}\subseteq A $ is a family of JM-elements for $A$
 with respect to $\mathcal{C}$ if it satisfies that $ L_i^{\ast} = L_i $ for all $ i $ and if
 there exists a set of
 scalars $\{c_{\T}(i)\mid \T\in T(\lambda),\; 1\leq i\leq M\}$, denoted the content functions for $ \lambda $, such
 that for all $\lambda\in \Lambda$ and $\T\in T(\lambda)$ we have that
\begin{equation}{\label{contents}}
 c_{\s \T}L_i=c_{\T}(i)c_{\s \T}+ \mathop{\sum_{\V\in T(\lambda)}}_{\V \rhd_{\lambda} \T}r_{\s \V}c_{\s \V} \mod A^{\lambda}
\end{equation}
$ \mbox{ for some }r_{\s \V}\in \F$.
\end{definition}

We can now prove the following main Theorem of our paper, proving that
the Jucys-Murphy elements introduced in (\ref{jm}) give rise to JM-elements in
the sense of the previous Lemma.

\begin{theorem}\label{Cellular}
  Let $ L_i \in  \HH(q, \kappa)$ be the Jucys-Murphy element introduced in (\ref{jm})
  and define $ \JM_i := f^{-1}(L_i) \in \R$. Then the set 
 $ \{ \JM_i \mid i=1, \ldots, n \}  $ is
  a family of $ JM$-elements for $ \B $ with respect to the cellular basis $ \Basis$.
  The corresponding content function is the one introduced in (\ref{contentF}):
\begin{equation}  c_{\Bs}(i) = q^{\textrm{res}(\Bs(i))}.  
\end{equation}  
\end{theorem}  
\begin{proof}
  By Theorem 1.1 of Brundan and Kleshchev's work, \cite{brundan-klesc}, we have that
  \begin{equation}
\JM_k = \sum_{\bi \in \II^{n} } q^{i_k}(1- y_k) e(\bi ) 
  \end{equation}    
  from which we get
\begin{equation} \JM_k e(\bi^{\Bs}) = (c_{\Bs}(k)- y_k ) e(\bi^{\Bs})   
\end{equation}
for any standard tableau $ \Bs$. The Theorem now follows from Lemma \ref{Jucys-Murphy}.
\end{proof}

\section{Comparison with the original definition of $ \B$. }
In this section we show that $ \B $ is isomorphic to the original generalized blob algebra, introduced
by Martin and Woodcock in \cite{MW}.  Our proof is an extension of an argument
presented in \cite{PlazaRyom}. That argument followed the suggestion of one of the referees of \cite{PlazaRyom}.

\medskip
Let $ \HHtwo$ be the cyclotomic Hecke algebra for $ n=2 $, as introduced
in Definition \ref{hecke algebra}. It follows from strong adjacency-freeness of $ \hat{\kappa} $
that $ \HHtwo$ is a semisimple $ \F$-algebra. Following \cite{MW}, for $ j= 1, \ldots, n $ we let $   e_2^j $  
be the primitive, central idempotents associated with the one-dimensional module 
given by the multipartition $ \blambda_2^j := (\emptyset, \ldots, (2), \ldots,  \emptyset) $ of 2, that has 
the partition $ (2)$ positioned in the $j$'th position.
Since $ \HHtwo \subseteq \HH$ we may consider $ e_2^j $ as an element of $ \HH $ and
so we may consider $ {\mathcal I}_n \subseteq \HH $,  the two-sided ideal generated by $ e_2^j $ for $ j = 1, \ldots, n $.
The generalized blob algebra  $\B^{\prime} $ introduced in \cite{MW} was now defined via
\begin{equation}
\B^{\prime}  := \HH/ {\mathcal I}_n.
\end{equation}  
In \cite{MW}, concrete formulas for $ e_2^j $ were found. For $ l = 2 $ these formulas gave rise to an isomorphism
between 
$ \B^{\prime} $ and the usual blob algebra.
The following Lemma gives another description of $ e_2^j $. 
\begin{lemma}
Let $  F_{\bT^{\blambda_2^j}} \in \HHKtwo $ be the idempotent defined in (\ref{idempotentK}).
Then $  F_{\bT^{\blambda_2^j}} \in  \HHOtwo $ and $ e_2^j = F_{\bT^{\blambda_2^j}} \otimes_{\OO} \F$.
\end{lemma}
\begin{proof}
  It follows from strong adjacency-freeness of $ \hat{\kappa} $ that the only standard tableau in the
  class $ [ \bT^{\blambda_2^j}]$ is $ \bT^{\blambda_2^j}$ itself and so
\begin{equation}\label{andletusnowdefineSECOND}
E_{[  \bT^{\blambda_2^j}] } = \sum_{ \bT \in  [ \bT^{\blambda_2^j}] \cap \std(n)} F_{\bT} = F_{  \bT^{\blambda_2^j} }.
\end{equation}
Since $ E_{[  \bT^{\blambda_2^j}] } \in \HHOtwo $ this 
shows that $  F_{\bT^{\blambda_2^j}} \in  \HHOtwo $. On the other hand,  we have
by (\ref{by the contents}) that
\begin{equation}\label{cond1}
  L_i   F_{\bT^{\blambda_2^j}}   = c_{\bT^{\blambda_2^j}}(i) F_{\bT^{\blambda_2^j}} =
  \left\{ \begin{array}{ll} q^{\kappa_j} F_{\bT^{\blambda_2^j}} & \mbox{if } i = 1 \\
    q^{\kappa_j+1} F_{\bT^{\blambda_2^j}} & \mbox{if } i = 2  \end{array} \right.
\end{equation}
and moreover, using (\ref{YSF}) and (\ref{knownconstanst}), we have that
\begin{equation}\label{cond2}
  T_1 F_{\bT^{\blambda_2^j}}   = qF_{\bT^{\blambda_2^j}}.
\end{equation}
The two conditions (\ref{cond1}) and (\ref{cond2}) characterize $ e_2^j $ uniquely and so the Lemma is proved.
\end{proof}

We can now prove the promised isomorphism between the two definitions of the
generalized blob algebra.
\begin{theorem}
Viewing $  F_{\bT^{\blambda_2^j}} $ as elements of $ \HH$ we have the 
following equality in $ \R$
\begin{equation}
f^{-1}( F_{\bT^{\blambda_2^j}}) = \sum_{\substack{\bi \in \II^n \\ i_1 = \kappa_j , i_2 = \kappa_j +1}}  e(\bi)
\end{equation}
corresponding to relation (\ref{eq13}) of $ \B$.
In particular, $ \B^{\prime} = \B$. 
\end{theorem}
\begin{proof}
  We have that $ 1 = \sum_{\bi \in \II^n }  e(\bi) = \sum_{ \Bs \in \std(n) } f^{-1}(E_{\Bs}).$
  On the other hand we have that
\begin{equation}
  F_{\bT^{\blambda_2^j}} E_{\Bs} =  \sum_{ \bT \in \std(n)} F_{\bT^{\blambda_2^j}} F_{\bT}
  = \left\{ \begin{array}{ll} E_{\Bs} & \mbox{ if } i_1 = \kappa_j , i_2 = \kappa_j +1 \\
0 & \mbox{otherwise} \end{array} \right.
\end{equation}  
and so the Theorem follows.
\end{proof}


\sc 
diego.lobos@pucv.cl, Universidad de Talca/Pontificia Universidad Cat\'olica de Valparaiso, Chile. \newline
steen@inst-mat.utalca.cl, Universidad de Talca, Chile.

\end{document}